\documentclass[preprint,3p,number]{elsarticle}
\usepackage[utf8]{inputenc}
\usepackage{amsmath}
\usepackage{amssymb}
\usepackage{amsthm}
\usepackage{amsfonts}
\usepackage{bm}
\usepackage{titlesec}
\usepackage{mathtools}
\usepackage{esvect}
\usepackage{ stmaryrd }
\usepackage[ruled,vlined]{algorithm2e}
\usepackage{tikz}
\usepackage{graphicx}
\usepackage{caption}
\usepackage{subcaption}
\usepackage{ textcomp }
\usepackage{media9}
\usepackage{amsthm}
\usepackage{floatrow}
\usepackage{hyperref}
\usepackage{cleveref}
\usepackage{hhline}
\usepackage[titletoc]{appendix}
\usepackage{caption}
\usepackage[normalem]{ulem}
\usepackage{cancel}
\usepackage{lscape}
\usetikzlibrary{positioning}

\newfloatcommand{capbtabbox}{table}[][\FBwidth]

\theoremstyle{definition}

\newtheorem{theorem}{Theorem}[section]

\newtheorem{lemma}[theorem]{Lemma}

\let\kron\otimes
\let\vec\bm

\newcommand{\oned}{\rm 1D}
\newcommand{\low}{\rm L}
\newcommand{\high}{\rm H}

\newcommand{\diff}[1]{\mathop{}\!{\mathrm{d}#1}} 
\newcommand{\diag}[1]{{\rm diag}\LRp{#1}} 
\newcommand{\td}[2]{\frac{{\rm d}#1}{{\rm d}{ {#2}}}} 
\newcommand{\pd}[2]{\frac{\partial#1}{\partial#2}}

\newcommand{\nor}[1]{\left\| #1 \right\|} 
\newcommand{\LRp}[1]{\left( #1 \right)} 
\newcommand{\LRs}[1]{\left[ #1 \right]} 
\newcommand{\LRb}[1]{\left| #1 \right|} 
\newcommand{\LRc}[1]{\left\{ #1 \right\}} 
\newcommand{\LRl}[1]{\left. #1 \right|} 
\newcommand{\jump}[1] {\ensuremath{\llbracket#1\rrbracket}} 

\newcommand{\fnt}[1]{\bm{\mathsf{ #1}}}

\newcommand{\rzero}[1]{{\color{black}{#1}}}
\newcommand{\rone}[1]{{\color{black}{#1}}}
\newcommand{\rtwo}[1]{{\color{black}{#1}}}

\newtheorem{remark}{Remark}

\title{A positivity preserving strategy for entropy stable discontinuous Galerkin discretizations of the compressible Euler and Navier-Stokes equations} 
\author[1]{Yimin Lin}
\author[1]{Jesse Chan}
\address[1]{Department of Computational and Applied Mathematics, Rice University, 6100 Main St, Houston, TX, 77005}
\author[2]{Ignacio Tomas}
\address[2]{Sandia National Laboratories\footnote{Sandia National Laboratories is a multimission laboratory managed and operated by National Technology \& Engineering Solutions of Sandia, LLC, a wholly-owned subsidiary of Honeywell International Inc., for the U.S. Department of Energy’s National Nuclear Security Administration under contract DE-NA0003525. This document describes objective technical results and analysis. Any subjective views or opinions that might be expressed in the paper do not necessarily represent the views of the U.S. Department of Energy or the United States Government}, P.O. Box 5800, MS 1320, Albuquerque, NM 87185-1320}

\begin{document}

\begin{abstract}
    High-order entropy-stable discontinuous Galerkin methods for the compressible Euler and Navier-Stokes equations require the positivity of thermodynamic quantities in order to guarantee their well-posedness. In this work, we introduce a positivity limiting strategy for entropy-stable discontinuous Galerkin discretizations constructed by blending high order solutions with \rzero{a low order positivity-preserving discretization}. \rtwo{The proposed low order discretization is semi-discretely entropy stable}, and the proposed limiting strategy is positivity preserving for the compressible Euler and Navier-Stokes equations. Numerical experiments confirm the high order accuracy and robustness of the proposed strategy. 
\end{abstract}

\maketitle

\section{Introduction}

Computational fluid dynamics simulations increasingly require higher resolutions for a variety of applications \cite{pi2013cfd}. For certain flows, high order accurate numerical methods are more accurate per degree of freedom compared to low order methods, and provide one avenue towards high accuracy while retaining reasonable efficiency \cite{wang2013high}. This paper focuses on high order discontinuous Galerkin (DG) methods, which are suitable for convection dominated problems and admit relatively simple and efficient implementations due to the locality of many operations \cite{hesthaven2007nodal}. 

Unfortunately, high-order DG methods typically suffer from stability issues in the presence of under-resolved solution features. High-order entropy stable DG (ESDG) discretizations of the compressible Euler and Navier-Stokes equations provide one method of addressing such stability issues while retaining high order accuracy. ESDG discretizations satisfy a semi-discrete balance of entropy and remain robust even in the absence of additional stabilization, filtering, or artificial viscosity \cite{chan2018discretely,chan2019skew,chan2022entropy,carpenter2014entropy,gassner2013skew,gassner2016split,chen2017entropy}. However, while entropy stable schemes require the positivity of thermodynamic quantities such as density and pressure, such schemes do not enforce positivity over the course of a simulation. This paper focuses on techniques to enforce positivity while retaining beneficial properties of high order ESDG schemes.

\rtwo{Positivity-preserving limiters have been widely employed for several decades \cite{lohner2008applied, berthon2008invariant}, and a variety of limited numerical schemes have been applied both to the compressible Euler \cite{kuzmin2010failsafe, calgaro2013positivity} and compressible Navier-Stokes equations \cite{kuzmin2012flux}. While a comprehensive review of limiting strategies is outside of the scope of this paper, it is worth reviewing several relevant limiting techniques to provide context for the strategy introduced in this paper. }

\rtwo{One popular class of limiting strategies are algebraic flux-correction (AFC) schemes (see for example Kuzmin et al.\ \cite{kuzmin2012flux, lohmann2016synchronized, lohmann2017flux, kuzmin2020subcell}). The underlying low-order method is a generalization of the local Lax-Friedrichs (LLF) method to nodal finite element discretizations. Guermond and Popov \cite{guermond2016invariant} showed that this low order scheme ensures preservation of invariant domains if the artificial ``graph viscosity'' coefficient is chosen to be sufficiently large. When combined with the convex-limiting framework (a generalization of algebraic flux correction) this robust low order discretization can be used to construct limited solutions which are second-order accurate and invariant domain preserving (implying positivity) \cite{guermond2018second, guermond2019invariant}. This framework has also been extended to the compressible Navier-Stokes equations \cite{guermond2021second} using an operator splitting approach and appropriate discretization of the parabolic terms.}

For high order DG schemes, Zhang, Shu, and colleagues introduced simple and effective positivity-preserving scaling limiters for systems of conservation laws \cite{zhang2010maximum, zhang2011maximum}. These limiters have been extended to advection-diffusion problems and the compressible Navier-Stokes equations as well \cite{zhang2010positivity, zhang2017positivity, srinivasan2018positivity, upperman2021first}, and are often paired with additional shock capturing techniques, such as TVD limiting or artificial viscosity.  

\rtwo{The limiting strategy used in this paper essentially combines an AFC-type low order scheme and a Zhang-Shu scaling limiter on each element.} This limiting strategy is similar to the use of sub-cell finite volume limiting schemes for high order DG methods \cite{dumbser2014posteriori, sonntag2014shock, vilar2019posteriori, hennemann2021provably, rueda2021subcell}, and most closely resembles the limiting approaches introduced in \cite{hennemann2021provably, rueda2021subcell}. Our approach generalizes these approaches to the compressible Navier-Stokes equations using techniques from \cite{zhang2017positivity}.
We also extend the construction of sparse summation-by-parts operators on quadrilateral elements in \cite{pazner2021sparse} to triangular elements. 

\rtwo{We briefly note some ways in which the limiting strategy introduced in this paper differs from strategies introduced in recent works such as \cite{guermond2021second, pazner2021sparse, hajduk2021monolithic, DZANIC2022111501}. First, we treat the parabolic terms fully explicitly, while \cite{guermond2021second, DZANIC2022111501} treat the parabolic terms using an implicit-explicit operator splitting. As a result, the spatial resolution of the method in this paper is limited by a parabolic CFL condition. However, numerical experiments suggest that this approach can still produce meaningful results when the viscous regime is under-resolved. Secondly, while a local minimum entropy principle is imposed in \cite{guermond2021second, pazner2021sparse, hajduk2021monolithic, DZANIC2022111501}, we impose limiting on the density and pressure. The motivation for this is that the minimum entropy principle is consistent with the compressible Euler equations but not the compressible Navier-Stokes equations. 
}

The outline of the paper is as follows: Section \hyperref[sec:background]{2} reviews the compressible Euler and Navier-Stokes equations and the notations used in the paper. Section \hyperref[sec:ops]{3} presents the discrete operators in our proposed discretizations, and Section \rzero{\hyperref[sec:esdg]{4}} reviews high order nodal entropy stable discontinuous Galerkin discretizations. In Section \hyperref[sec:low]{5}, we introduce the low order positivity preserving discretization. In Section \hyperref[sec:poslim]{6} we present \rzero{the elementwise limiting strategy}. In Section \hyperref[sec:time]{7}, we introduce the time discretization used, and in Section \hyperref[sec:exp]{8}, we provide various numerical experiments in 1D and 2D to verify the convergence and robustness of the proposed limiting strategy. Finally we summarize our work in Section \hyperref[sec:con]{9}. 


\section{Background knowledge}
\label{sec:background}
In this work, we focus on the compressible Euler and Navier-Stokes equations in two space-dimensions. The theoretical contributions of this paper are straightforward to extend to three dimensions. 
\subsection{Governing equations}
The two-dimensional compressible Navier-Stokes equations in conservative form are given by:
\begin{equation}\label{eq:CNS}
    \pd{\vec{u}}{t} + \sum\limits_{i=1}^2\pd{\vec{f}^\text{I}_i}{x_i} = \sum\limits_{i=1}^2 \pd{\vec{f}^\text{V}_i}{x_i},
\end{equation}
where $\vec{u}, \vec{f}^\text{I}_i, \vec{f}^\text{V}_i$ denote the vector of conservative variables, inviscid fluxes, and viscous fluxes respectively. We follow \cite{chan2014dpg} and write the nondimensional compressible Navier-Stokes equation in 2D as:
\begin{equation}\label{eq:CNSnondim}
     \pd{}{t}\begin{bmatrix}\rho \\ \rho u \\ \rho v \\ E\end{bmatrix} + \pd{}{x}\begin{bmatrix}\rho u \\ \rho u^2 + p \\ \rho u v \\ (E+p)u \end{bmatrix} + \pd{}{y}\begin{bmatrix} \rho v \\ \rho v u \\ \rho v^2 + p \\ (E+p)v\end{bmatrix} = \frac{1}{\text{Re}}\pd{}{x}\begin{bmatrix}0 \\ \tau_{xx} \\ \tau_{yx} \\ \tau_{xx}u + \tau_{yx}v + \kappa \pd{T}{x}\end{bmatrix} + \frac{1}{\text{Re}}\pd{}{y}\begin{bmatrix}0 \\ \tau_{xy} \\ \tau_{yy} \\ \tau_{xy}u + \tau_{yy}v + \kappa \pd{T}{y}\end{bmatrix}
\end{equation}
Here, $\rho,u,v,E$ denote the density, velocity in the $x,y$ directions, and total mechanical energy respectively. In this work, we assume an ideal gas closure, such that the pressure $p$ and the temperature $T$ are given by the equations of state: 
\begin{equation*}
    p = \LRp{\gamma - 1}\rho e, \qquad e = c_v T,
\end{equation*}
where $e$ is the specific internal energy, and $c_v, \gamma, \kappa, \mu$ are specific heat capacity, ratio of specific heats, heat conductivity, and dynamic viscosity respectively. $\text{Re}$, the Reynolds number, and $\text{Pr}$, the Prandtl number, are dimensionless quantities. The relation between the total energy and the specific internal energy is given by:
\begin{align*}
    E = \rho e + \frac{1}{2}\rho \LRp{u^2+v^2}
\end{align*}

The nondimensional compressible Navier-Stokes equations are equivalent to the conservative form \eqref{eq:CNS} through a scaling of physical parameters \cite{chan2014dpg}. In this work, we assume all parameters to refer to their nondimensionalized quantities. The compressible Euler equations are a special case of the compressible Navier-Stokes equations with $\text{Re} \rightarrow \infty$. In other words, the compressible Euler equations describe compressible fluids with zero viscosity and thermal conductivity. 

The admissible set of the compressible Euler and Navier-Stokes equations is the set of conservative variables with positive density and specific internal energy. We can write this admissible set as the intersection of superlevel sets of concave functions:
\begin{equation}\label{eq:admissbleset}
    \mathcal{A} \coloneqq \LRc{\vec{u} = \LRp{\rho, \rho u, E} \  | \ \rho(\vec{u}) > 0, \rho e(\vec{u}) > 0} = \bigcap\limits_{\rho_0 > 0, e_0 > 0} \LRc{\vec{u} \ | \ \rho(\vec{u}) \geq \rho_0}  \cap \LRc{\vec{u} \ | \ \rho e(\vec{u}) \geq \rho_0 e_0}.
\end{equation}
Physically meaningful solutions to \eqref{eq:CNS} lie in this admissible set, which is a convex set for both sets of equations \cite{zhang2017positivity}. \rone{From \cite{nishida1968global, chueh1977positively, lions1996existence, guermond2014viscous}, solutions to the compressible Euler equations lie in the admissible set under appropriate regularity assumptions. While we are not aware of similar theoretical results for solutions of the compressible Navier-Stokes equations, we follow works such as \cite{svard2016convergent, guermond2021second} which assume that such solutions are consistent with positivity of thermodynamic quantities.} The focus of this work is to limit discretely entropy-stable discontinuous Galerkin methods so that the limited solution at each time step remains in $\mathcal{A}$.

\subsection{Entropy variables and the symmetrization of viscous fluxes}
Both compressible Euler and Navier-Stokes equations admit a mathematical entropy balance with respect to a convex scalar mathematical entropy
\begin{equation*}
    \eta(\vec{u}) = - \frac{\rho s}{\gamma-1}
\end{equation*}
where $s = \log\LRp{\frac{p}{\rho^\gamma}}$ denotes the physical entropy \cite{hughes1986new}. Entropy variables are then defined as the derivative of the mathematical entropy with respect to the conservative variables. The mappings between the entropy variables $\vec{v}$ and the conservative variables $\vec{u}$ are given by
\begin{align*}
    \vec{v}\LRp{\vec{u}} &= \begin{bmatrix}v_1 & v_2 & v_3 & v_4\end{bmatrix} = \begin{bmatrix}\frac{\rho e (\gamma +1 -s)- E}{\rho e} & \frac{u}{e} & \frac{v}{e} & -\frac{1}{e}\end{bmatrix} \\
    \vec{u}\LRp{\vec{v}} &= \begin{bmatrix}-(\rho e)v_4 & \rho e v_2 & \rho e v_3 & \rho e \LRp{1-\frac{v_2^2+v_3^2}{2v_4}}\end{bmatrix}.
\end{align*}
The internal energy and physical entropy can also be expressed in terms of entropy variables 
\begin{equation*}
    \rho e = \LRp{\frac{\gamma-1}{(-v_4)^\gamma}}^{1/(\gamma-1)} e^{-\frac{s}{\gamma-1}}, \qquad s = \gamma - v_1 + \frac{v_2^2+v_3^2}{2v_4}. 
\end{equation*}
It was shown in \cite{hughes1986new} that the entropy variables symmetrize the viscous fluxes in the following sense
\begin{equation}\label{eq:entropysymm}
    \pd{\vec{f}^\text{V}_1}{x} + \pd{\vec{f}^\text{V}_2}{y} = \nabla \cdot \LRp{\vec{K}\nabla \vec{v}} = \pd{}{x}\LRp{\vec{K}_{11} \pd{\vec{v}}{x} + \vec{K}_{12} \pd{\vec{v}}{y}} + \pd{}{y}\LRp{\vec{K}_{21} \pd{\vec{v}}{x} + \vec{K}_{22} \pd{\vec{v}}{y}}, 
\end{equation}
where $\vec{K}_{ij}$ are blocks of a symmetric positive \rone{semi-}definite matrix $\vec{K}$ 
\begin{equation*}
    \vec{K} = \begin{bmatrix} \vec{K}_{11} & \vec{K}_{12} \\ \vec{K}_{21} & \vec{K}_{22} \end{bmatrix} 
\end{equation*}

\subsection{On notation}

We follow the notation convention introduced in \cite{chan2022entropy}. Vector and matrix quantities are denoted using lower and upper case bold fonts respectively (for example, $\vec{A}$ and $\vec{u}$). Spatially discrete quantities are written in bold sans serif font (for example, $\fnt{x}$). For clarity, continuous real functions evaluated over spatially discrete quantities are taken to mean point-wise evaluations. For example,
\begin{align*}
    \fnt{x} = \begin{bmatrix}\vec{x}_1 \\ \vdots \\ \vec{x}_n\end{bmatrix}, \qquad u: \mathbb{R} \rightarrow \mathbb{R}, \qquad u(\fnt{x}) = \begin{bmatrix} u(\vec{x}_1) \\ \vdots \\ u(\vec{x}_n)\end{bmatrix}
\end{align*}
For systems of conservation laws, there are multiple scalar components. When $\fnt{A} \in \mathbb{R}^{m\times m}, \fnt{u}\in \mathbb{R}^{mn}$, we abuse notation and adopt the convention that $\fnt{A}\fnt{u}$ is the Kronecker product $(\fnt{A}\kron \fnt{I}_n) \fnt{u}$ \cite{chen2017entropy}.

In this paper, we will adopt the lexicographical ordering of nodes and basis functions, so that a multi-index is replaced by a lexicographical single index for clarity of notation. 
We will use a number subscript $\fnt{A}_1,\fnt{A}_2$ or a letter subscript $\fnt{A}_r$,$\fnt{A}_s$ interchangeably to indicate the coordinates of discrete operators. This work will present the theory on the reference element $\widehat{D}$ and ignore the involving geometric terms for clarity of notation. We refer readers to \cite{chan2018discretely, chan2019discretely, crean2017high} for the extension of high order ESDG schemes to mapped elements and curved meshes. The extension of positivity preserving schemes to curved meshes follows the approach in  \cite{pazner2021sparse}, \rone{and is expanded on in more detail in Appendix~\ref{app:curved}}.

\section{Discrete operators}
\label{sec:ops}

We denote the computational domain by $\Omega \subseteq \mathbb{R}^2 $. We discretize the domain using non-overlapping quadrilateral or triangular elements $D^k$. We assume for now that each physical element is the images of the reference element $\widehat{D}$ through an affine mapping $\vec{\Phi}^k(\vec{r},\vec{s})$, such that geometric change-of-variable factors are constant on each element. The extension to curvilinear meshes \rtwo{is briefly described in Appendix~\ref{app:curved}.}

\subsection{Multidimensional Summation-By-Parts operators}

The construction of both the high and low order numerical schemes in this work relies on summation-by-parts (SBP) operators \cite{svard2014review, fernandez2014review}. In this work, we focus on diagonal-norm SBP operators, which can be interpreted as differentiation matrices weighted by diagonal mass matrices. Each SBP operator is induced by an appropriate volume quadrature rule $\LRp{\vec{r},\vec{w}}$. We assume these quadrature rules contain identically distributed surface points on each face, and that these surface points correspond to a separate surface quadrature rule $\LRp{\vec{r}^f,\vec{w}^f}$. We denote the number of collocated nodes by $N_{\rm p}$ and the number of surface quadrature points by $N_{\rm p}^{\rm f}$. Moreover, we require positivity of both volume and surface quadrature weights and assume that the volume and surface quadrature rules are exact for polynomials of degree $2N-1$ and $2N$ respectively. 

Figure \ref{fig:SBPnodes} illustrates the SBP quadrature points on a tensor product and a simplicial reference element. A degree $N$ SBP quadrature on tensor product elements is simply the tensor product of $(N+1)$-point Gauss-Lobatto quadratures, and the surface quadrature rule on simplicial elements is the $(N+1)$-point Gauss-Lobatto rule. 
\begin{figure}[htp] 
    \centering
    \subfloat[SBP nodes on the reference tensor product element, $N = 4$]{%
        \includegraphics[width=0.2\textwidth]{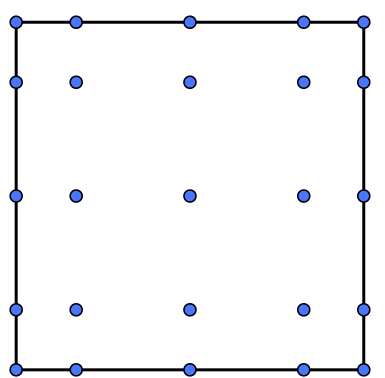}
        }%
    \hspace{2cm}
    \subfloat[SBP nodes on the reference simplex element, $N = 3$]{%
        \includegraphics[width=0.32\textwidth]{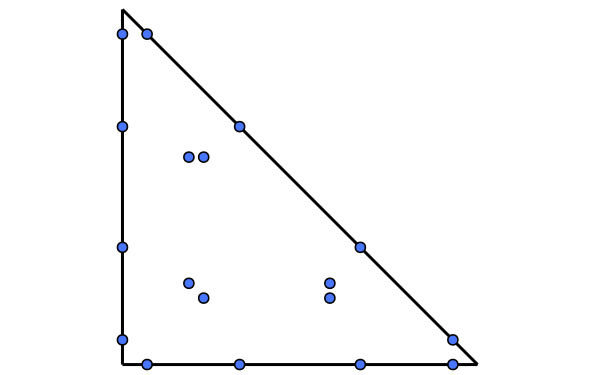}%
        }%
    \caption{SBP quadrature rules}\label{fig:SBPnodes}
\end{figure}

We can now introduce the relevant SBP matrix operators. First, we define the lumped mass matrix in terms of SBP quadrature weights:
\begin{equation*}
    \fnt{M} = \begin{bmatrix} \vec{w}_1 & & \\ & \ddots & \\ & & \vec{w}_{N_{\rm p}}\end{bmatrix}.
\end{equation*}
Next, we introduce SBP differentiation matrices by $\fnt{Q}_r, \fnt{Q}_s$ and nodal differentiation matrices $\fnt{D}_r = \fnt{M}^{-1}\fnt{Q}_r$,  $\fnt{D}_s = \fnt{M}^{-1}\fnt{Q}_s$. These differentiation matrices satisfy high order accuracy conditions: let $\fnt{u}$ denote the vector containing nodal values $u(\fnt{x}_i)$ for some degree $N$ polynomial $u(x)$. Then, $\fnt{D}_r, \fnt{D}_s$ satisfy
\begin{equation*}
     \LRp{\fnt{D}_r \fnt{u}}_{i} = \LRl{\pd{u}{r}}_{r = r_i}, \quad \LRp{\fnt{D}_s \fnt{u}}_{i} = \LRl{\pd{u}{s}}_{s = s_i}.
\end{equation*}
We note that these operators can be constructed directly using a Lagrange polynomial basis on quadrilateral elements. On triangular elements, we construct nodal SBP operators using Theorem 3.1 of \cite{wu2021high}.

A key property of SBP operators involves the relationship between the differentiation matrices and surface matrices. We first introduce $\fnt{E}$ as the extrapolation (or face extraction) matrix from the volume quadrature points to the surface quadrature points. Since we assume the SBP quadrature includes the surface quadrature points, the extrapolation matrix is a matrix of size $N_{\rm p}^{\rm f}\times N_{\rm p}$ with entries either $0$ or $1$. Next, we introduce boundary integration matrices of size $N_{\rm p}^{\rm f}\times N_{\rm p}^{\rm f}$
\begin{equation*}
    \fnt{B}_r = \begin{bmatrix} \vec{w}_1^{\rm f} \widehat{\vec{n}}_{r,1} & & \\ & \ddots & \\ & & \vec{w}_{N_{\rm p}^{\rm f}}\widehat{\vec{n}}_{r,N_{\rm p}^{\rm f}}\end{bmatrix},\qquad 
    \fnt{B}_s = \begin{bmatrix} \vec{w}_1^{\rm f} \widehat{\vec{n}}_{s,1} & & \\ & \ddots & \\ & & \vec{w}_{N_{\rm p}^{\rm f}}\widehat{\vec{n}}_{s,N_{\rm p}^{\rm f}}\end{bmatrix},
\end{equation*}
where $\widehat{\vec{n}}_r, \widehat{\vec{n}}_s$ are components of outward normal vectors of the reference element. Finally, we note that SBP operators $\fnt{Q}$ (e.g. the differentiation matrices weighted by the lumped mass matrix) satisfy the SBP property:
\begin{align}
    \fnt{Q}_r + \fnt{Q}_r^T = \fnt{E}^T \fnt{B}_r \fnt{E}, \qquad \fnt{Q}_s + \fnt{Q}_s^T = \fnt{E}^T \fnt{B}_s \fnt{E}.\label{eq:SBP}
\end{align}
The SBP property replicates integration by parts at a discrete level in the following way:
\begin{align}
    \fnt{u}^T(\fnt{Q}_r+\fnt{Q}_r^T)\fnt{v} = \fnt{u}^T \fnt{E}^T\fnt{B}_r\fnt{E} \fnt{v} \quad \Longleftrightarrow \quad \int_{\widehat{D}}{{u}\pd{{v}}{r}} + \int_{\widehat{D}}{\pd{{u}}{r}, {v}} = \int_{\partial \widehat{D}}{{u},{v} \widehat{{n}}_r}
\end{align}

\subsection{Sparse low order Summation-By-Parts operators}

Given a set of SBP quadrature points, we also wish to construct sparse low-order summation-by-parts operators $\fnt{Q}^{\low}$. Since we will construct a low-order scheme by applying an algebraic dissipation operator based on the sparsity pattern of our discrete operators, we seek operators which are sparse to avoid introducing unnecessary dissipation in the low-order algebraic scheme \rtwo{\hbox{\cite{kuzmin2020subcell,lohmann2017flux,hajduk2021monolithic,pazner2021sparse}}}. We require only that the operators satisfy an SBP property and a conservation property:
\begin{align}
    &\fnt{Q}_r^{\low} + \LRp{\fnt{Q}_r^{\low}}^T = \fnt{E}^T\fnt{B}_r \fnt{E},\qquad \fnt{Q}_r^{\low}\bm{1} = 0, \label{eq:SBPlowr}\\
    &\fnt{Q}_s^{\low} + \LRp{\fnt{Q}_s^{\low}}^T = \fnt{E}^T\fnt{B}_s \fnt{E},\qquad \fnt{Q}_s^{\low}\bm{1} = 0. \label{eq:SBPlows}
\end{align}

For tensor product elements, we follow \cite{pazner2021sparse} and construct sparse low order SBP operators are derived by integrating the piecewise linear basis $\vec{p}_i$ constructed on the LGL nodes
\begin{align*}
    \LRp{\fnt{Q}^{\low}_{\oned}}_{ij} = \int_{\widehat{D}} \vec{p}_i \pd{\vec{p}_j}{r} \diff{r}, \qquad \fnt{Q}^{\low}_{\oned} = \begin{bmatrix}-\frac{1}{2} & \frac{1}{2} & & \\-\frac{1}{2} & 0 & \frac{1}{2} & \\ & -\frac{1}{2} & 0 & \frac{1}{2} & \\ & & & \ddots\end{bmatrix}, \qquad \fnt{Q}^{\low}_r = \fnt{I}_{N_{\rm P}} \kron \fnt{Q}_{\oned}^{\low}, \qquad \fnt{Q}^{\low}_s = \fnt{Q}^{\low}_{\oned} \kron \fnt{I}_{N_{\rm P}}
\end{align*}

For simplicial elements, we follow \cite{wu2021entropy} and construct sparse operators based on a user-provided stencil (or a connectivity graph) built upon the SBP quadrature points. We motivate this procedure as follows. First, note that any SBP operator can be written as the sum of its skew-symmetric part and the boundary integration matrix:
\begin{align}
     &\fnt{Q}_r = \frac{\fnt{Q}_r - \fnt{Q}_r^T}{2} + \frac{1}{2}\fnt{E}^T \fnt{B}_r \fnt{E},\label{eq:SBPr}\\
    &\fnt{Q}_s = \frac{\fnt{Q}_s - \fnt{Q}_s^T}{2} +  \frac{1}{2}\fnt{E}^T \fnt{B}_s \fnt{E}.\label{eq:SBPs}   
\end{align}
Because the boundary integration matrices are fully determined by the surface quadrature rule, an SBP operator can be determined by specifying its skew-symmetric part. Thus, constructing a sparse SBP operator reduces to constructing a sparse skew-symmetric matrix which also satisfies the conservation properties (\ref{eq:SBPlowr}) and (\ref{eq:SBPlows}). 

To determine the sparsity pattern of the skew-symmetric part of the SBP matrix, we restrict the low order SBP operator to have the same sparsity pattern as the adjacency matrix of the graph. In this work, we define the adjacency matrix $\fnt{A}$ through a simple formula
\begin{align}
    \LRp{\fnt{A}}_{ij} = 
    \begin{cases}
    1 &\text{if } \nor{\vec{r}_i-\vec{r}_j}_2 \leq \alpha \max\LRc{\LRp{\frac{\vec{w}_i}{\pi}}^{\beta}, \LRp{\frac{\vec{w}_j}{\pi}}^\beta}\\
    0 &\text{otherwise}
    \end{cases}
\label{eq:adjacency}
\end{align}
where $\alpha,\beta \in \mathbb{R}$ are parameters. Note that when $\alpha = 1, \beta = \frac{1}{2}$, two quadrature points are considered ``adjacent'' if they both lie in the circle of area $\max\LRc{\vec{w}_i,\vec{w}_j}$ centered at either quadrature point. In one-space dimension this is equivalent to the notion of "three-point stencil': for each time step, the evolution of the solution at the current node depends only on its own value and the value of its immediate neighbors.

Finally, we can define the sparse low order SBP operator by assuming that nonzero entries of the skew-symmetric part $\fnt{S}_i = \frac{1}{2}\LRp{\fnt{Q}_i-\fnt{Q}_i^T}$ are of the form $\fnt{S}_{i,jk} = \vec{\psi}_{i,k} - \vec{\psi}_{i,j}$, where $\vec{\psi}$ is some ``potential'' vector \cite{trask2020conservative} and the subscript $i$ denotes one of the reference coordinates $r, s$. We compute the vector $\vec{\psi}_i$ by solving the constrained linear system:
\begin{align*}
    &\fnt{Q}_i^{\low} \bm{1} = 0 \\
    &\textrm{s.t.} \quad \fnt{S}_{i, jk} = \begin{cases}
    0 &\text{if } \fnt{L}_{jk} = 0\\
    \vec{\psi}_{i,k} - \vec{\psi}_{i,j} &\text{otherwise}
    \end{cases} .\\
    &\qquad \quad\vec{\psi}_i^T \bm{1} = 0
\end{align*}
where $\fnt{Q}_i^{\low} = \fnt{S}_i + \frac{1}{2} \fnt{E}^T \fnt{B}_i \fnt{E}$. 
It was observed in \cite{trask2020conservative} that this system reduces to
\begin{align*}
    \fnt{L} \vec{\psi} = \frac{1}{2} \fnt{E}^T \fnt{B}_i\fnt{E} \bm{1} \qquad\textrm{    s.t.    }\qquad \vec{\psi}_i^T \bm{1} = 0,
\end{align*}
where $\fnt{L}$ is the graph Laplacian
\[
\LRp{\fnt{L}}_{ij} = \begin{cases}
    \text{deg}\LRp{\vec{r}_i} &\text{if } i = j \\
    -1 &\text{if } \fnt{A}_{ij} \neq 0 \\
    0 &\text{otherwise}
    \end{cases}.
\]    
The constraint $\vec{\psi}_i^T \bm{1} = 0$ ensures the existence and uniqueness of the solution. We refer interested readers to \cite{wu2021entropy} for a more detailed description of this process.

\section{High order entropy-stable discontinuous Galerkin discretizations}
\label{sec:esdg}
In this section, we will review nodal entropy stable DG methods for the compressible Euler and Navier-Stokes equations. We first review entropy conservative numerical fluxes as introduced by Tadmor \cite{tadmor1987numerical,tadmor2003entropy}, which are fundamental to the construction of entropy stable DG schemes. Such fluxes are symmetric bivariate functions $\vec{f}_{S}(\vec{u}_L,\vec{u}_R)$ that are consistent with respect to a given flux $\vec{f}(\vec{u})$, i.e. $\vec{f}_{S}(\vec{u},\vec{u}) = \vec{f}(\vec{u})$. In addition, they satisfy an entropy conservation property 
\[
(\vec{v}_L-\vec{v}_R)^T \vec{f}_{S}(\vec{u}_L,\vec{u}_R) = \psi(\vec{u}_L)-\psi(\vec{u}_R),
\]
which relates the flux to the entropy variables and entropy potential $\psi(\vec{u})$. In this paper, we utilize numerical fluxes derived by Chandrashekar which are both entropy conservative and kinetic energy preserving \cite{chandrashekar2013kinetic}. 

The derivative of the inviscid flux $\vec{f}_i^{\rm I}$ can then be reformulated using the entropy conservative flux. The reformulation is commonly referred to as ``flux differencing'' and can be interpreted as a high-order subcell-based finite volume formulation. Gassner, Winters, Kopriva and Chan provide continuous interpretations of the technique \cite{gassner2018br1, chan2018discretely}
\begin{align*}
    \pd{\vec{f}_i^{\rm I}(\vec{u}(\vec{x}))}{x_i} = 2\LRl{ \pd{\vec{f}_{i,S}(\vec{u}(\vec{x}),\vec{u}(\vec{y}))}{x_i}}_{\vec{y}=\vec{x}}
\end{align*}
Then, we can discretize the variational form of the derivative of the inviscid flux using the high order weighted differentiation matrix and the the row sum of a Hadamard product
\begin{align}
    \int_{\widehat{D}} \pd{\vec{f}_k^{\rm I}(\vec{u})}{x_k} 
    \vv{\vec{l}}  \quad \xRightarrow{\rm Discretize} \quad 2\LRp{\fnt{Q}_k \circ \fnt{F}_{k}}\bm{1}, \quad \LRp{\fnt{F}_{k}}_{ij} = \vec{f}_{k,S}\LRp{\vec{u}_i,\vec{u}_j},  \label{eq:Fk}
\end{align}
where $\vv{\vec{l}}$ denotes the vector of Lagrange polynomial basis functions.

An entropy conservative DG formulation for the compressible Euler equation can then be written as follows:
\begin{align}\label{eq:ESDGeuler}
    \fnt{M}\rtwo{\td{\fnt{u}}{t}} + \sum\limits_{k=1}^2 2 \LRp{\fnt{Q}_k \circ \fnt{F}_{k}} \bm{1} + \sum\limits_{k=1}^2 \fnt{E}^T\fnt{B}_k\LRp{\fnt{f}_k^{\rm{I},*}-\fnt{f}_k^{\rm I}} = 0.
\end{align}
We can rewrite the formulation in a skew-symmetric form by using the SBP property
\begin{align}\label{eq:ESDGskeweuler}
        \fnt{M}\rtwo{\td{\fnt{u}}{t}} + \sum\limits_{k=1}^2 \LRp{\LRp{\fnt{Q}_k - \fnt{Q}_k^T} \circ \fnt{F}_{k}} \bm{1} + \sum\limits_{k=1}^2 \fnt{E}^T\fnt{B}_k\fnt{f}_k^{\rm{I},*} = 0,
\end{align}
where we denote the interface numerical flux by $\fnt{f}^{\rm{I},*}_k$. \rzero{For this paper, we use local Lax-Friedrichs interface fluxes in our high order ESDG formulation. When the maximum wave speed estimate is suitably chosen, the Lax-Friedrichs fluxes are entropy-stable, implying that (\ref{eq:ESDGskeweuler}) will also be entropy-stable. }

To discretize the symmetrized viscous fluxes, we first rewrite the viscous terms as a first-order system of partial differential equations using derivatives of the entropy variables and auxiliary variables $\bm{\sigma}_i$. Then we can derive an LDG-type formulation for the first-order system \cite{chan2022entropy}, 
\begin{alignat}{4}
    \vec{\Theta}_i &= \pd{\vec{v}}{x_i} &&\fnt{M}\fnt{\Theta}_i &&= \frac{\fnt{Q}_i - \fnt{Q}_i^T}{2}\fnt{v} + \fnt{E}^T\fnt{B}_i\rzero{\fnt{v}^+} \nonumber\\
    \vec{\sigma}_i &= \vec{K}_{i1}\vec{\Theta}_1 + \vec{K}_{i2}\vec{\Theta}_2 \quad \xRightarrow{\rm Discretize} \quad &&\fnt{\sigma}_i &&= \fnt{K}_{i1}\fnt{\Theta}_1 + \fnt{K}_{i2}\fnt{\Theta}_2 \label{eq:ESDGCNScont}\\
    \pd{\vec{f}^{\rm V}_1}{x_1} + \pd{\vec{f}^{\rm V}_2}{x_2} &= \pd{\vec{\sigma}_1}{x_1}+\pd{\vec{\sigma}_2}{x_2} &&\fnt{M} \fnt{H} &&= \sum\limits_{k=1}^2 \LRp{\LRp{\fnt{Q}_k - \fnt{Q}_k^T} \circ \fnt{F}^{\sigma}_k}\bm{1} + \sum\limits_{k=1}^2 \fnt{E}^T\fnt{B}_k \fnt{\sigma}^*_k \nonumber \\
    & &&\LRp{\fnt{F}^\sigma_k}_{ij} &&= \frac{\LRp{\fnt{\sigma}_k}_i+\LRp{\fnt{\sigma}_k}_j}{2},\nonumber \\
    & &&\LRp{\fnt{v}^*}_i &&= \frac{\vec{v}_i + \vec{v}_i^+}{2}, \quad \LRp{\fnt{\sigma}_k^*} = \frac{\LRp{\fnt{\sigma}_k}_i + \LRp{\fnt{\sigma}_k}_i^+}{2},\nonumber
\end{alignat}
where $\fnt{M}\fnt{H}$ is the discretization of $\int_{\widehat{D}} \LRp{\pd{\vec{\sigma}_1}{x_1}+\pd{\vec{\sigma}_2}{x_2}} \vv{\vec{l}}$. Note that we have written the product of the SBP operator and the flux vector in a flux differencing formulation involving the Hadamard product of the SBP operator and a central flux matrix \cite{chan2022entropy}. This formulation will become useful in Section~\ref{sec:low}.

Under appropriate choices of interface flux and interface penalization, formulations \eqref{eq:ESDGeuler}, \eqref{eq:ESDGskeweuler} and \eqref{eq:ESDGCNScont} satisfy semi-discrete entropy balances. For details of the proof, readers may refer to \cite{chan2022entropy}. For completeness, we state the semi-discrete entropy balance for the compressible Navier-Stokes equation for reference:

\begin{theorem}
    Assume continuity in time and positivity of the density and internal energy. If the domain is periodic, then
    \begin{equation*}
        \td{}{t}\int_{\Omega} \eta\LRp{\vec{u}} \approx \bm{1}^T \rtwo{\fnt{M}}\td{\eta\LRp{\fnt{u}}}{t} = \fnt{v}^T \fnt{M}\rtwo{\td{\fnt{u}}{t}} \leq -\sum\limits_{i,j=1}^2 \fnt{\Theta}_i^T \fnt{M} \fnt{K}_{ij} \fnt{\Theta}_j \leq 0,
    \end{equation*}
    which corresponds to the continuous entropy balance
    \begin{equation*}
        \int_\Omega \pd{\eta(\vec{u})}{t} = \int_{\Omega} \LRp{\pd{\vec{u}}{t}}^T \vec{v}\LRp{\vec{u}} \leq \int_\Omega \sum\limits_{i,j=1}^2 \LRp{\pd{\vec{v}}{x_i}}^T \LRp{\vec{K}_{ij} \pd{\vec{v}}{x_j}}
    \end{equation*}
\end{theorem}

This discrete entropy balance extends also to certain non-periodic boundary conditions (for example, adiabatic or reflective walls) \cite{chan2022entropy}. 


\section{A positivity preserving low order method for the compressible Navier-Stokes equations}
\label{sec:low}

The nonlinear entropy stability described in the previous section holds only when the density and internal energy are positive. In this section, we will introduce a sparse low-order positivity preserving discretizations. The formulation is based on a sparse low order DG discretization which, when augmented with an artificial dissipation term (usually referred to as the ``graph viscosity'' \cite{guermond2018second}), preserves the positivity of density and pressure. The sparse low-order positivity preserving DG discretization can be written in matrix form as \begin{gather}
    \fnt{M}\rtwo{\td{\fnt{u}}{t}} + \sum\limits_{k=1}^2 \LRp{\LRp{\fnt{Q}_k^{\low} - \LRp{\fnt{Q}_k^{\low}}^T} \circ \fnt{F}_k}\bm{1}  - \LRp{\fnt{\Lambda}\circ \fnt{D}}\bm{1} + \sum\limits_{k=1}^2 \fnt{E}^T\LRs{\fnt{B}_k \fnt{f}_k^* - \fnt{\lambda}_k\jump{\fnt{u}}} = 0 \label{eq:lowpp}
    \\
            \LRp{\fnt{F}_k}_{ij} = 
    \frac{1}{2}\LRp{\vec{f}_{k,i}^\text{I} + \rzero{\vec{f}_{k,j}^\text{I}} - \LRp{\vec{\sigma}_k}_i - \LRp{\vec{\sigma}_k}_j}
    , \quad
    \fnt{f}_k^* = 
    \frac{1}{2}\LRp{\vec{f}_{k,i}^\text{I} + \vec{f}_{k,i}^{\text{I},+} - \LRp{\fnt{\sigma}_k}_i - \LRp{\fnt{\sigma}_k}_i^+}\nonumber
    \\
    \fnt{\Lambda}_{ij} = \lambda_{ij}, \quad \LRp{\fnt{D}_{ij}} = \vec{u}_j -\vec{u}_i. \nonumber 
\end{gather}
Note that $\vec{\sigma} = 0$ recovers the compressible Euler equations. 

Formulation \rone{\eqref{eq:lowpp}} starts from the skew-symmetric form of the discontinuous Galerkin collocation spectral element discretizations \cite{gassner2013skew}. To construct the low order positivity-preserving scheme, we replace the high-order SBP operators $\fnt{Q}_k$ with the sparse low-order operators $\fnt{Q}_k^{\low}$ introduced in Section \ref{sec:ops}. As suggested in \rtwo{\hbox{\cite{kuzmin2020subcell,lohmann2017flux,hajduk2021monolithic,pazner2021sparse}}}, sparsifying the SBP operator prevents the accuracy of the proposed low order discretization from decreasing as the order of approximation increases \cite{pazner2021sparse}. 

The key to achieving positivity is the addition of the graph viscosity term $\LRp{\fnt{\Lambda}\circ \fnt{D}}\bm{1}$ and the penalization $\fnt{\lambda}_k\jump{\fnt{u}}$ at interfaces. The graph viscosity coefficients are chosen as
\begin{align}
    \lambda_{ij} &= \max\LRc{\beta\LRp{\vec{u}_i,\vec{\sigma}_i,\frac{\vec{n}_{ij}}{\nor{\vec{n}_{ij}}}}, \beta\LRp{\vec{u}_j,\vec{\sigma}_j,\frac{\vec{n}_{ij}}{\nor{\vec{n}_{ij}}}}}\nor{\vec{n}_{ij}} \nonumber\\
    \LRp{\fnt{\lambda}_k}_i &= \frac{\LRb{\LRp{\fnt{B}_k}_{ii}}}{2}
    \max\LRc{\beta\LRp{\vec{u}_i,\vec{\sigma}_i, \widehat{\fnt{n}}_i}, \beta\LRp{\vec{u}_i^+,\vec{\sigma}_i^+, \widehat{\fnt{n}}_i}}
     \qquad \fnt{n}_{ij} =\frac{1}{2}
    \begin{bmatrix}\LRp{\fnt{Q}_r^{\low} - \LRp{\fnt{Q}_r^{\low}}^T}_{ij} \\ \LRp{\fnt{Q}_s^{\low} - \LRp{\fnt{Q}_s^{\low}}^T}_{ij}\end{bmatrix}
    \label{eq:visccoeff}
\end{align}
We note that the graph viscosity coefficient $\lambda_{ij}$ is nonzero if and only if ${\LRp{\fnt{Q}_r^{\low} - \LRp{\fnt{Q}_r^{\low}}^T}}_{ij}$ or ${\LRp{\fnt{Q}_s^{\low} - \LRp{\fnt{Q}_s^{\low}}^T}}_{ij}$ is nonzero (equivalently, the graph viscosity term has the same sparsity pattern as the graph-Laplacian $\fnt{L}$, excluding the diagonal entries). 
Moreover, if the coefficient $\beta$ is defined appropriately, we can ensure positivity of the fully discrete version of (\ref{eq:lowpp}). In particular, we define 
\begin{align}
\beta\LRp{\vec{u},\vec{\sigma},\vec{n}} &= \epsilon_0 + \LRb{\vec{n}\cdot \vec{u}} + \frac{1}{2\rho^2 e}\LRp{\sqrt{\rho^2 \LRp{\vec{q}\cdot\vec{n}}^2 + 2\rho^2e\nor{\vec{n}\cdot \vec{\tau} - p \vec{n}}^2}+\rho \LRb{\vec{q}\cdot \vec{n}}},\label{eq:zhangbeta}\\
    \vec{n} \cdot \vec{\tau} &= \begin{bmatrix}\vec{n}_1 \LRp{\vec{\sigma}_1}_2 + \vec{n}_2 \LRp{\vec{\sigma}_1}_3 \\ \vec{n}_1 \LRp{\vec{\sigma}_2}_2 + \vec{u}_2 \LRp{\vec{\sigma}_2}_3\end{bmatrix}, \qquad \vec{q} = \begin{bmatrix}\vec{u}_1 \LRp{\vec{\sigma}_1}_2 + \vec{u}_2 \LRp{\vec{\sigma}_1}_3 - \LRp{\vec{\sigma}_1}_4 \\ \vec{u}_1 \LRp{\vec{\sigma}_2}_2 + \vec{u}_2 \LRp{\vec{\sigma}_2}_3 - \LRp{\vec{\sigma}_2}_4\end{bmatrix}\nonumber
\end{align} where $\epsilon_0$ is a positive number arbitrarily close to $0$\rzero{, and we assume the normal $\vec{n}$ is an unit vector}. 
This definition is motivated by the observation in \cite{zhang2017positivity} that positivity of the quantity 
\begin{equation}
\beta\vec{u} \pm \LRp{\vec{f}^{\rm I}\LRp{\vec{u}}-\vec{\sigma}}\cdot \vec{n}
\label{eq:zhang_pos}
\end{equation}
is equivalent to the following condition:
\begin{align}
\beta > \LRb{\vec{n}\cdot \vec{u}} + \frac{1}{2\rho^2 e}\LRp{\sqrt{\rho^2 \LRp{\vec{q}\cdot\vec{n}}^2 + 2\rho^2e\nor{\vec{n}\cdot \vec{\tau} - p \vec{n}}}+\rho \LRb{\vec{q}\cdot \vec{n}}}, \label{eq:charbeta}
\end{align}
where $\vec{n}$ is a unit normal vector. This observation can be used to show that, under forward Euler time-stepping and an appropriate CFL condition, the low order scheme preserves positivity of the density and pressure. 

Using the SBP properties \eqref{eq:SBPlowr} and \eqref{eq:SBPlows}, we can rewrite formulation \eqref{eq:lowpp} (under a forward Euler time discretization) for each index $i$ as 
\begin{align}
    \fnt{m}_i\frac{\fnt{u}_i^{{\rm L},n+1}- \fnt{u}_i}{\tau} &+ \sum\limits_{j \in \mathcal{I}(i)}\sum\limits_{k=1}^2 \LRp{\fnt{Q}_k^{\low} - \LRp{\fnt{Q}_k^{\low}}^T}_{ij}\frac{\vec{f}_k\LRp{\fnt{u}_j} \rzero{-\LRp{\fnt{\sigma}_k}_j}-\vec{f}_k(\fnt{u}_i)\rzero{+\LRp{\fnt{\sigma}_k}_i}}{2} - \lambda_{ij}\LRp{\fnt{u}_j + \fnt{u}_i} \nonumber\\
    &+ \sum\limits_{j \in \mathcal{B}(i)} \sum\limits_{k=1}^2 \LRp{\fnt{E}^T\fnt{B}_k\fnt{E}}_{ii} \frac{\rzero{\vec{f}_k\LRp{\fnt{u}_i^+}-\LRp{\fnt{\sigma}_k}_i^+ - \vec{f}_k\LRp{\fnt{u}_i}+\LRp{\fnt{\sigma}_k}_i}}{2} - \LRp{\fnt{\lambda}_k}_i\LRp{\fnt{u}_j + \fnt{u}_i} \nonumber \\
    &+\LRp{\sum\limits_{j \in \mathcal{I}(i)} 2\lambda_{ij} + \sum\limits_{j \in \mathcal{B}(i)} \sum\limits_{k=1}^2 2\LRp{\fnt{\lambda}_k}_i} \fnt{u}_i = 0, \label{eq:lpprewrite}
\end{align}
where we write $\fnt{u}^n_i$ as $\fnt{u}_i$ for simplicity of notation. We define $\mathcal{I}(i)$, the neighboring nodes of node $i$, as the set of indices $j$ where $\LRp{\fnt{Q}_r^{\low} - \LRp{\fnt{Q}_r^{\low}}^T}_{ij}$ or $\LRp{\fnt{Q}_s^{\low} - \LRp{\fnt{Q}_s^{\low}}^T}_{ij}$ is nonzero. $\mathcal{B}(i)$ is the set of indices for nodes exterior to the surface quadrature points on a given element, whose states are usually denoted by $\fnt{u}^+$. 

We can rearrange \eqref{eq:lpprewrite} to express the low order solution as a convex combination of the previous state $\fnt{u}^n$ and the intermediate states $\overline{\fnt{u}}$:
\begin{align}
    \fnt{u}^{{\rm L},n+1}_i = \LRp{1-\frac{2\tau\lambda_i}{\fnt{m}_i}}\fnt{u}_i  + \frac{2\tau\lambda_{ij}}{\fnt{m}_i}\sum\limits_{j \in \mathcal{I}(i)} \overline{\fnt{u}}_{\fnt{u}_i,\fnt{u}_j,\fnt{\sigma}_i,\fnt{\sigma}_j,\fnt{n}_{ij}} + \sum\limits_{k=1}^2 \frac{2\tau\LRp{\fnt{\lambda}_k}_i}{\fnt{m}_i} \sum\limits_{j \in \mathcal{B}(i)} \overline{\fnt{u}}_{\fnt{u}_i,\fnt{u}_i^+,\fnt{\sigma}_i,\fnt{\sigma}_i^+,\fnt{n}_i}, \label{eq:lowppintermediate}
\end{align}
where we define the intermediate states\rtwo{, introduced in the context of Godunov-type methods for hyperbolic systems \hbox{\cite{harten1983upstream}},} as 
\begin{align}
    \overline{\vec{u}}_{\vec{u}_L,\vec{u}_R,\vec{\sigma}_L,\vec{\sigma}_R,\vec{n}} &=
      \frac{1}{2}\LRp{\vec{u}_L+\vec{u}_R} - \frac{1}{2\beta_{L,R}}\rzero{\frac{\vec{n}}{\nor{\vec{n}}}}\cdot\LRp{\vec{f}^{\rm I}_R - \vec{\sigma}_R- \vec{f}^{\rm I}_L  +\vec{\sigma}_L}, \label{eq:intermediate} \\
      \qquad \beta_{L,R} &= \max\LRc{\beta\LRp{\vec{u}_L, \vec{\sigma}_L, \frac{\vec{n}}{\nor{\vec{n}} }}, \beta\LRp{\vec{u}_R, \vec{\sigma}_R, \frac{\vec{n}}{\nor{\vec{n}} }}} \label{eq:betaLR}
\end{align}
We can show the intermediate states are admissible with the graph viscosity coefficients defined in \eqref{eq:visccoeff}:
\begin{lemma}
For admissible states $\vec{u}_L \in \mathcal{A}$, $\vec{u}_R \in \mathcal{A}$, the intermediate states $\overline{\vec{u}}_{\vec{u}_L,\vec{u}_R,\vec{\sigma}_L,\vec{\sigma}_R,\vec{n}}$ are admissible. \label{lemma}
\end{lemma}

\begin{proof}

We can rewrite the intermediate state as
\begin{align}
    \overline{\vec{u}}_{\vec{u}_L,\vec{u}_R,\vec{\sigma}_L,\vec{\sigma}_R,\vec{n}} &= \frac{1}{2\beta_{L,R}} \LRs{\beta_{L,R} \vec{u}_R+\LRp{\vec{f}^I_R - \vec{\sigma}_R}\cdot \rzero{\frac{\vec{n}}{\nor{\vec{n}}}}} + \frac{1}{2\beta_{L,R}} \LRs{\beta_{L,R} \vec{u}_L-\LRp{\vec{f}^I_L - \vec{\sigma}_L}\cdot \rzero{\frac{\vec{n}}{\nor{\vec{n}}}} }.
    \label{eq:intermediaterewrite}
\end{align}
By \eqref{eq:betaLR}, $\beta_{L,R} \geq \beta_L, \beta_{L,R} \geq \beta_R$, so the intermediate state $    \overline{\vec{u}}_{\vec{u}_L,\vec{u}_R,\vec{\sigma}_L,\vec{\sigma}_R,\vec{n}}$ is admissible by  \eqref{eq:charbeta} \cite{zhang2017positivity}.
\end{proof}

Then the conservation and positivity preserving properties of the formulation \eqref{eq:lowpp} directly follows from Lemma \ref{lemma}.

\begin{theorem}
    \label{thm:lowpos}
    Assume the domain is periodic. The formulation \eqref{eq:lowpp} is conservative 
    \begin{align}
        \bm{1}^T\fnt{M}\fnt{u}^{k+1} = \bm{1}^T\fnt{M}\fnt{u}^{k}. \qquad \label{eq:conservation}
    \end{align}
    In addition, if $\fnt{u}_i^n \in \mathcal{A}$ and a timestep condition is satisfied
    \begin{align}
        \tau \leq \frac{\fnt{m}_i}{2\lambda_i}, \qquad \lambda_i = \sum\limits_{j \in \mathcal{I}(i)} \lambda_{ij} + \sum\limits_{j \in \mathcal{B}(i)} \sum\limits_{k=1}^2 \LRp{\fnt{\lambda}_k}_i, \label{eq:CFL}
    \end{align}
    then the formulation is positivity preserving under forward Euler time-stepping:
    \begin{align}
        \fnt{u}^{{\low},n+1}_i \in \mathcal{A}.
    \end{align}
\end{theorem}

\begin{proof}
To show conservation, we proceed algebraically and observe that
    $\LRp{\fnt{Q}_k^{\low}-\LRp{\fnt{Q}_k^{\low}}^T}\circ \fnt{F}_k$ and $\fnt{\Lambda}\circ \fnt{D}$ are skew-symmetric. Thus,
    \begin{align}
        \bm{1}^T \LRs{\sum\limits_{k=1}^2 \LRp{\LRp{\fnt{Q}_k^{\low} - \LRp{\fnt{Q}_k^{\low}}^T} \circ \fnt{F}_k}\bm{1}  - \LRp{\fnt{\Lambda}\circ \fnt{D}}\bm{1} } = 0. \label{eq:proofcons}
    \end{align}
    Additionally, assuming the domain is periodic yields 
    \begin{align}
        \bm{1}^T\fnt{E}^T\sum\limits_{k=1}^2\LRp{\fnt{B}_k \fnt{f}_k^* - \fnt{\lambda}_k\jump{\fnt{u}} + \fnt{B}_k^+ \fnt{f}_k^* - \fnt{\lambda}_k \jump{\fnt{u}^+}} = 0. \label{eq:proofcons2}
    \end{align}
    where we have used that $\fnt{B}_k = -\fnt{B}_k^+$ and that the jump changes sign contributions on a neighboring element across a face. The conservation property follows. 
    
    To show positivity, we note that \rzero{\eqref{eq:lowppintermediate}} implies that the low order update is an convex combination of admissible states if a timestep condition is satisfied: $\tau \leq \frac{\fnt{m}_i}{2\lambda_i}$. Since the admissible set $\mathcal{A}$ is convex, the low order update $\fnt{u}_i^{{\low},n+1}$ is admissible.
\end{proof}

\begin{remark}
\rone{
We note that the timestep condition \eqref{eq:CFL} is comparable to other conditions derived in the literature \cite{zhang2010positivity, zhang2017positivity, pazner2021sparse}, and scales as $O(h / N^2)$. However, it has been observed that this positivity-preserving time-step is around 2-3 times smaller than the maximum time-step for a high order DG method \cite{zhang2010positivity}. This has been addressed using heuristic approaches, e.g., using a less restrictive CFL and backtracking to the positivity-preserving CFL condition when bounds violations are detected. However, we have not utilized such approaches in this work.
}
\end{remark}

\begin{remark}
\rtwo{
We emphasize that we only consider explicit time-stepping in this work, and the time step size will be limited by a parabolic CFL condition when solving the compressible Navier-Stokes equations. 
Thus, while the time-step restriction \eqref{eq:CFL} implies positivity, it is not sufficient to imply stability of the time-stepping iteration \cite{zhang2017positivity}. 

While the proposed method is restricted by the parabolic CFL condition for very fine resolutions, it is applicable to the simulation convection-dominated flows at lower resolutions. The stable time-step for the compressible Euler equations scales as $O(h)$, while the stable time-step for the parabolic part of the compressible Navier-Stokes equations scales as $O(h^2 \text{Re})$, where $\text{Re}$ is the Reynolds number. Thus, if $h = O(1 / \text{Re})$, the parabolic CFL condition is not overly restrictive compared to the hyperbolic CFL condition.
}
\end{remark}

\section{An entropy stable and positivity preserving limiting strategy}
\label{sec:poslim}

\rzero{In this section, we will discuss procedures for blending the positivity-preserving low-order discretization with the high-order entropy-stable discretization. The limiting strategy we will focus on in this paper was first proposed in \cite{zhang2010maximum}, and we refer this limiting strategy as elementwise (Zhang-Shu type) limiting. Readers may refer to Appendix \ref{app:convexlimit} for another limiting strategy called convex limiting inspired by flux corrected transport \cite{boris1973flux}. 

Our objective is to construct a limited solution $\fnt{u}^{n+1}_i$ from low and high order solution updates $\fnt{u}^{{\low},n+1}_i, \fnt{u}^{{\high},n+1}_i$ such that $\fnt{u}^{n+1}_i$ lies in the generalized admissible set
\begin{align}
    \fnt{u}^{n+1}_i \in \mathcal{A}\LRp{\rho^{\min}_i, \rho e^{\min}_i} = \LRc{\fnt{u}_i\ | \ \rho\LRp{\fnt{u}_i} \geq \rho^{\min}_i > 0, \rho e \LRp{\fnt{u}_i} \geq \rho e^{\min}_i > 0}. \label{eq:genralizedadmisset}
\end{align}
Here, we extend the notion of admissible set \eqref{eq:admissbleset} to satisfy positive lower bounds on the density and specific internal energy. 
We emphasize that these lower bounds are defined per node and are time dependent.

To simplify notation, we write the low order and high-order approximations \eqref{eq:ESDGskeweuler}, \eqref{eq:ESDGCNScont} and \eqref{eq:lowpp} over node $i$ in the residual form:
\begin{align}
    \fnt{m}_i \frac{\fnt{u}^{{\low},n+1}_i-\fnt{u}_i}{\tau} + \fnt{r}^{\low}_i &= 0 \label{eq:lpprf}\\
    \fnt{m}_i \frac{\fnt{u}^{{\high},n+1}_i-\fnt{u}_i}{\tau} + \fnt{r}^{\high}_i &= 0, \label{eq:esdgrf}
\end{align}
where we define the low order and high-order residuals at node $i$ as
\begin{align}
    \fnt{r}_i^{\low} &= \sum\limits_{j\in \mathcal{I}\LRp{i}}\sum\limits_{k=1}^2 \frac{1}{2}\LRp{\fnt{Q}_k^{\low} - \LRp{\fnt{Q}_k^{\low}}^T}_{ij} \LRs{\vec{f}_k\LRp{\fnt{u}_i}+\vec{f}_k\LRp{\fnt{u}_j} - \LRp{\fnt{\sigma}_k}_i - \LRp{\fnt{\sigma}_k}_j} - \lambda_{ij}\LRp{\fnt{u}_j - \fnt{u}_i} \nonumber \\
    &+ \sum\limits_{j \in \mathcal{B}\LRp{i}}\sum\limits_{k=1}^2 \frac{1}{2}\LRp{\fnt{E}^T\fnt{B}_k\fnt{E}}_{ii} \LRs{\vec{f}_k\LRp{\fnt{u}_i} + \vec{f}_k\LRp{\fnt{u}_i^+} - \LRp{\fnt{\sigma}_k}_i - \LRp{\fnt{\sigma}_k}_i^+} - \LRp{\fnt{\lambda}_k}_i \LRp{\fnt{u}_i^+ - \fnt{u}_i} \label{rli} \\
    \fnt{r}_i^{\high} &= \sum\limits_{k=1}^2 \LRp{\fnt{Q}_k - \fnt{Q}_k^{T}}_{ij} \LRs{\vec{f}_{k,S}\LRp{\fnt{u}_i,\fnt{u}_j}- \frac{\LRp{\fnt{\sigma}_k}_i + \LRp{\fnt{\sigma}_k}_j}{2}}\nonumber\\ 
    &+  \sum\limits_{k=1}^2\LRp{\fnt{E}^T\fnt{B}_k\fnt{E}}_{ii} \LRs{\vec{f}_{k,S}\LRp{\fnt{u}_i, \fnt{u}_i^+} - \frac{\LRp{\fnt{\sigma}_k}_i + \LRp{\fnt{\sigma}_k}_i^+}{2}} - \frac{\vec{w}_i^{ f}\LRp{\fnt{\lambda}_{\max,k}}_i}{2} \LRp{\fnt{u}_i^+ - \fnt{u}_i}. \label{rhi}
\end{align}
The low and high order updates satisfy the relation
\begin{align}
    \fnt{m}_i \fnt{u}_i^{{\high},n+1} = \fnt{m}_i \fnt{u}_i^{{\low},n+1} + \tau \LRp{\fnt{r}_i^{\low} - \fnt{r}_i^{\high}} \label{eq:lowhighres} 
\end{align}
Following Zhang and Shu's previous work on positivity-preserving limiters for high order DG methods \cite{zhang2010maximum, zhang2010positivity}, the limited solution $u_i^{n+1}$ can be written as
\begin{align}
     \fnt{m}_i \fnt{u}_i^{n+1} = \fnt{m}_i \fnt{u}_i^{{\low},n+1} + \tau l^e\LRp{\fnt{r}_i^{\low} - \fnt{r}_i^{\high}}, \label{eq:limitedzhangshu} 
\end{align}
where $l^e \in [0, 1]$ is the limiting parameter. These limiting parameters are defined as the maximal value of the blending parameter $l$ such that nodal values of the solution satisfy the following constraints on each element $D^k$
\begin{align}
    l^e = \max\LRc{ l \in [0,1]\ :  \fnt{u}^{{\low},n+1}_i + l \frac{\tau}{\fnt{m}_i} \LRp{\fnt{r}_i^{\low}- \fnt{r}_i^{\high}} \in \mathcal{A}\LRp{\rho^{\min}_i, \rho e^{\min}_i} \text{ for all } i \in D^k}. \label{eq:zhangshul}
\end{align}
We refer to $l^e$ as the elementwise (Zhang-Shu type) limiting/blending parameters.

\subsection{Solving for the blending parameter}

We can solve for the value of the limiting/blending parameter $l^e$ explicitly. In particular, its value is the maximum over quadratic constraints of form:
\begin{align}
    \fnt{u}^{\low} + l \fnt{P} &\in \mathcal{A}\LRp{\rho^{\min}_i, \rho e^{\min}_i}.\label{eq:quadconstraint}
\end{align}
Solving the constraint for the internal energy reduces to solving a quadratic equation. We state the explicit formula of $l$ here and refer readers to \cite{guermond2019invariant} for details of the proof. Denote $\fnt{u}^{\low} = \LRp{\rho^{\low}, \vec{m}^{\low}, E^{\low}}$ and $\fnt{P} = \LRp{\rho^{\rm P}, \vec{m}^{\rm P}, E^{\rm P}}$; then, the blending parameter $l$ is given by
\begin{align}
    l^\rho &= \begin{cases}
    1 &\text{ if }\rho^{\low} + \rho^{\rm P} \geq 0 ,\\
    \max\LRp{\frac{-\rho^{\low}+\rho^{\min}}{\rho^{\rm P}}, 0} &\text{ otherwise},
    \end{cases} \label{eq:lparamrho}\\
    l^{\rho e} &= \begin{cases} 1 &\text{ if } \LRc{l \geq 0 \ | \  al^2 + bl + c = 0} = \varnothing, \\
    \min\LRc{l \geq 0 \ | \  al^2 + bl + c = 0} &\text{ otherwise,}
    \end{cases} \label{eq:lparamrhoe}\\
    l &= \min\LRp{l^\rho, l^{\rho e}},
 \label{eq:lparam}
\end{align}
where $a, b, c$ are defined via
\begin{gather*}
a = E^{\rm P}\rho^{\rm P} - \frac{1}{2}\vec{m}^{\rm P} \cdot \vec{m}^{\rm P}, \qquad 
b = E^{\low}\rho^{\rm P} + \rho^{\low} E^{\rm P} - \vec{m}^{\low} \cdot \vec{m}^{\rm P} - \rho^{\rm P} \rho e^{\min},\\
c = E^{\low}\rho^{\low} - \frac{1}{2}\vec{m}^{\low} \cdot \vec{m}^{\low} - \rho^{\rm L} \rho e^{\min}.
\end{gather*}

An admissible set of limiting parameters exists as the solution to \eqref{eq:zhangshul} if the low order solution satisfies $\fnt{u}^{\low}_i \in \mathcal{A}\LRp{\rho^{\min}_i, \rho e^{\min}_i}$. 
We consider a lower bound that depends on the low order solution $\fnt{u}^{\low}$ and a relaxation factor $\zeta$:
\begin{align}
    \rho^{\min}_i = \zeta\rho\LRp{\fnt{u}_i^{\low}},\qquad \rho e^{\min}_i = \zeta\rho e \LRp{\fnt{u}_i^{\low}}, \qquad \zeta \in (0,1].\label{eq:generalizedposbound}
\end{align}
We note that (\ref{eq:generalizedposbound}) enforces a stronger condition than the minimal positivity constraints 
\begin{align}
     \rho^{\min}_i = \epsilon_0 > 0,\qquad \rho e^{\min}_i = \epsilon_0 > 0,
     \label{eq:posbound}
\end{align}
where $\epsilon_0$ is a small threshold. Numerical experiments suggest that the stronger bounds (\ref{eq:generalizedposbound}) avoid various issues (e.g., the time-step size converging to zero) without compromising accuracy. 
}

\subsection{Entropy stability of the low order and limited scheme}

We can, in addition, show the low-order positivity preserving method in Section \ref{sec:low} is discretely entropy-stable if we modify the amount of dissipation $\alpha$:
\begin{theorem}
\label{thm:lowes}
Let $\lambda_{\max}(\bm{u}_L, \bm{u}_R, \bm{n})$ be an upper bound on the maximum wavespeed. If the graph viscosity coefficients in \eqref{eq:visccoeff} are defined as
    \begin{align}
         \lambda_{ij} = \max\LRc{\lambda_{\max}\LRp{\vec{u}_i,\vec{u}_j,\frac{\vec{n}_{ij}}{\nor{\vec{n}_{ij}}}}, \beta\LRp{\vec{u}_i,\vec{\sigma}_i,\frac{\vec{n}_{ij}}{\nor{\vec{n}_{ij}}}}, \beta\LRp{\vec{u}_j,\vec{\sigma}_j,\frac{\vec{n}_{ij}}{\nor{\vec{n}_{ij}}}}}\nor{\vec{n}_{ij}},  \label{eq:modifiedCNScoeff}
    \end{align}
    then the positivity preserving low order method \eqref{eq:lowpp} also satisfies a \rzero{semi-}discrete entropy balance \rtwo{over each element $D$}
    \begin{align}
        \rtwo{\td{}{t}\int_{D} \eta\LRp{\fnt{u}} \approx \fnt{v}^T\fnt{M}\td{\fnt{u}}{t}}  \leq \sum\limits_{k=1}^2 \LRs{\bm{1}^T\fnt{E}^T \fnt{B}_k \fnt{E} \vec{\psi}_k} - \fnt{v}^T \fnt{E}^T \fnt{W}_f \widehat{\fnt{f}}^* -\sum\limits_{i,j=1}^2 \fnt{\Theta}_i^T \fnt{M} \fnt{K}_{ij} \fnt{\Theta}_j,  \label{eq:entropyinequalityCNS}
    \end{align}
    where $\fnt{\Theta} = 0$ for the compressible Euler equation.
\end{theorem}

\begin{proof}
We begin by interpreting the low order scheme \rone{\eqref{eq:lowpp}} as a low order subcell-based finite volume method using local Lax-Friedrichs type fluxes. In particular, an equivalent form of the formulation \eqref{eq:lowpp} is 
\begin{align}
    &\fnt{M}\rtwo{\td{\fnt{u}}{t}} + 2\LRp{\nor{\fnt{n}_{ij}}\circ \widehat{\fnt{F}}} \bm{1} + \fnt{E}^T \fnt{W}_f\widehat{\fnt{f}}^* = 0, \label{eq:subcellFV}\\
    &\widehat{\fnt{F}}_{ij} = \widehat{\vec{f}}\LRp{\fnt{u}_i, \fnt{u}_j,\frac{\fnt{n}_{ij}}{\nor{\fnt{n}_{ij}}}}, \quad \widehat{\fnt{f}}^*_i = \widehat{\vec{f}}\LRp{\fnt{u}_i, \fnt{u}_i^+, \widehat{\fnt{n}}_i}, \nonumber
\end{align}
where $\nor{\fnt{n}_{ij}}$ denotes the matrix whose $\LRp{i,j}$ entry is $\nor{\fnt{n}_{ij}}$ defined in \eqref{eq:visccoeff}, \rtwo{and $\fnt{W}_f$ denotes the diagonal matrix of face quadrature weights}. We define the local Lax-Friedrichs \rzero{type} flux as
\begin{align}
    \widehat{\vec{f}}\LRp{\vec{u}_L,\vec{u}_R,\vec{n}} = \vec{n} \cdot \frac{\vec{f}\LRp{\vec{u}_L} + \vec{f}\LRp{\vec{u}_R}}{2} - \frac{\alpha\LRp{\vec{u}_L,\vec{u}_R,\vec{\sigma}_L,\vec{\sigma}_R,\vec{n}}}{2}\LRp{\vec{u}_R - \vec{u}_L} \label{eq:LLF}
\end{align}

We next test the formulation \eqref{eq:subcellFV} with entropy variables $\fnt{v}$ evluated at nodes. We can rewrite the volume contribution as \cite{chan2018discretely}
    \begin{align}
        -\fnt{v}^T \LRp{2\nor{\fnt{n}_{ij}} \circ \widehat{\fnt{F}}} \bm{1} &= -\fnt{v}^T \LRp{\nor{\fnt{n}_{ij}} \circ \widehat{\fnt{F}} }\bm{1} + \fnt{v}^T \LRp{\nor{\fnt{n}_{ij}} \circ \widehat{\fnt{F}}^T }\bm{1} \label{eq:proofES1}\\
        &= \sum\limits_{i,j} \nor{\fnt{n}_{ij}}\LRp{\fnt{v}_i - \fnt{v}_j}^T \widehat{\vec{f}}\LRp{\fnt{u}_i,\fnt{u}_j,\frac{\fnt{n}_{ij}}{\nor{\fnt{n}_{ij}}}} \label{eq:proofES2}\\
        &\leq \sum\limits_{i,j}\nor{\fnt{n}_{ij}} \frac{\fnt{n}_{ij}}{\nor{\fnt{n}_{ij}}} \cdot \LRp{\vec{\psi}_i-\vec{\psi}_j} \label{eq:proofES3}\\
        &= \sum\limits_{k=1}^2 \bm{1}^T \LRp{\frac{\fnt{E}^T\fnt{B}_k\fnt{E}}{2} - \fnt{Q}^T_k}\vec{\psi}_k - \vec{\psi}_k^T \LRp{\fnt{Q}_k - \frac{\fnt{E}^T\fnt{B}_k\fnt{E}}{2}} \bm{1} = \sum\limits_{k=1}^2 \LRs{\bm{1}^T\fnt{E}^T \fnt{B}_k \fnt{E} \vec{\psi}_k},\label{eq:proofES4}
    \end{align}
    where we used the skew-symmetry of $\widehat{\fnt{F}}$ and the SBP property of $\fnt{Q}$. From \eqref{eq:proofES2} to \eqref{eq:proofES3}, we have used the entropy stability of the Lax-Friedrichs flux:
    \begin{align}
        \LRp{\fnt{v}_i -\fnt{v}_j}^T \widehat{\vec{f}} \leq \vec{\psi}_i-\vec{\psi}_j. \label{eq:nodewiseES}
    \end{align}
    The surface contribution $- \fnt{v}^T \fnt{E}^T \fnt{W}_f \widehat{\fnt{f}}^*$ directly follows from rearrangement, and the viscous contribution $-\sum\limits_{i,j=1}^2 \fnt{\Theta}_i^T \fnt{M} \fnt{K}_{ij} \fnt{\Theta}_j $ follows from \cite{chan2022entropy}.
\end{proof}

We note that for the compressible Euler equations, the viscosity coefficient $\beta$ defined in \eqref{eq:zhangbeta} is never larger than the maximum wavespeed. Therefore, the modified viscosity coefficient $\alpha$ defined in \eqref{eq:modifiedCNScoeff} reduces to the standard maximum wavespeed estimate for compressible Euler. As a result of Theorem \ref{thm:limitedES}, \rzero{the limited solution satisfies a semi-discrete entropy balance.}

\begin{theorem}
    \label{thm:limitedES}
    Assume the domain is periodic. Then, the semi-discrete limited scheme is conservative and satisfies a \rzero{semi-}discrete entropy balance for the compressible Navier-Stokes equations,
    \begin{align}
        \rtwo{\td{}{t}\int_{D} \eta\LRp{\fnt{u}}} \leq -\sum\limits_{i,j=1}^2 \fnt{\Theta}_i^T \fnt{M} \fnt{K}_{ij} \fnt{\Theta}_j \label{eq:ESdisspCNS}
    \end{align}
\end{theorem}

\begin{proof}
    Over each element, the elementwise limited solution can be written as a linear combination of low and high order solutions:
    \begin{align}
        \fnt{u}_i^{n+1} = \LRp{1-l} \fnt{u}^{{\low},n+1} + l \fnt{u}^{{\high},n+1}, \label{eq:limitedlncomb}
    \end{align}
    where $l$ is the elementwise limiting parameter. The result follows from the fact that both high and low-order solutions \rzero{are conservative and} satisfy a \rzero{semi-}discrete entropy balance \cite{chen2017entropy,carpenter2014entropy}.
\end{proof}

\rtwo{We note that while the proposed positivity-preserving limiter is fully discrete, Theorem~\ref{thm:limitedES} is a semi-discrete result. This is because (to the authors knowledge) it is not possible to prove a fully discrete entropy inequality for the high order method for either forward Euler or Runge-Kutta time-steppers. However, we note that it is possible to enforce a fully discrete entropy inequality using space-time formulations \cite{friedrich2019entropy} or explicit relaxation Runge-Kutta methods \cite{ranocha2020relaxation}.}

\subsection{Incorporating shock capturing}

\rzero{
Finally, our limiting strategy can be straightforwardly adapted as a shock capturing strategy. We can replace the limiting parameter as a blending function $\xi$. This shock capturing approach can be understood as blending the high order approximation with the low order positivity preserving approximation we proposed. Then we can write the shock-captured solution as:
\begin{align}
    \fnt{u}^{n+1} = \fnt{u}^{{\low},n+1} + \tau \xi \LRp{\fnt{r}_i^{\low} - \fnt{r}_i^{\high}}. \label{eq:shockcapsol}
\end{align}
The blending function $\xi$ depends on Persson and Peraire's modal shock indicator \cite{persson2006sub}. In this work's numerical experiment, we utilize the same the blending function and shock capturing parameters as in Hennemann et al.\ \cite{hennemann2021provably}.

}

\section{Time discretization}\label{sec:time}

Until now, we have assumed a first order forward Euler time discretization. We extend to higher order in time using Strong Stability Preserving (SSP) explicit Runge-Kutta schemes. We present here the SSPRK(3,3) method for reference.
\begin{alignat*}{3}
    w^{\LRp{1}} &= u^n + \tau L(t^n, u^n), \qquad &&z^{\LRp{1}} = w^{\LRp{1}} + \tau L(t^n + \tau, w^{\LRp{1}}) \\
    w^{\LRp{2}} &= \frac{3}{4} u^n + \frac{1}{4} z^{\LRp{1}}, \qquad &&z^{\LRp{2}} = w^{\LRp{2}} + \tau L(t^n+\frac{1}{2}\tau, w^{\LRp{2}}), \\
    u^{n+1} &= \frac{1}{3} u^n + \frac{2}{3} z^{\LRp{2}}
\end{alignat*}
where $\tau$ is the timestep size, and $L(t,u)$ is the evaluation of the time derivative. The limiting framework we propose ensures the solution remains in the convex admissible set \eqref{eq:admissbleset} for a single forward Euler timestep. SSP schemes are convex combinations of first-order forward Euler steps, and since convex combinations of the solution remain in a convex set, SSPRK time-steppers preserve positivity \cite{shu1988efficient}. 

\section{Numerical experiments}
\label{sec:exp}
In this section, we present various numerical experiments to verify the convergence and robustness of the proposed limiting strategy \footnote{The codes used for the experiments are available at \\  https://github.com/yiminllin/ESDG-PosLimit/tree/main/examples/IDP}. All of the simulations advance in time using the third-order SSP Runge Kutta method introduced in Section \ref{sec:time}. The timestep size is determined from the timestep condition in \eqref{eq:CFL}: 
\begin{align}\label{eq:deltat}
    \Delta t = \text{CFL} \min_i \frac{\fnt{m}_i}{2\lambda_i},
\end{align}
where CFL is a user specified parameter. For all numerical experiments, \rzero{if not specified,} we set the positivity threshold as $\epsilon_0 = 10^{-14}$ in \eqref{eq:zhangbeta} \rzero{and \eqref{eq:posbound}}. 

For all numerical experiments on simplicial elements, we construct sparse low order SBP operators by building the adjacency matrix and graph Laplacian in equation \eqref{eq:adjacency} with $\beta = 1$ and $\alpha = 4, 2.5, 3.5, 3.5$ for polynomial degrees $N = 1,2,3,4$ respectively. We note that the selection of parameters are not optimal, and numerical results are not sensitive to the choice of the parameters. 

All entropy stable schemes in this work utilize the entropy conservative numerical flux introduced by Chandreshekar \cite{chandrashekar2013kinetic}. We evaluate the logarithmic mean with Ismail and Roe's numerically stable expansion \cite{ismail2009affordable}. 

All limiting strategies use the modified viscosity coefficients introduced in Theorem \ref{thm:lowes}. We estimate the maximum wavespeed associated with the 1D Riemann problems by the Davis estimate \cite{davis1988simplified}\footnote{\rone{We note that the Davis estimate does not provide a robust upper bound of the maximum wavespeed. Appendix B of \cite{guermond2016fast} describes an example where the Rusanov estimate underestimates the maximum wavespeed, which could potentially compromise the positivity-preservation and semi-discrete entropy stability of the low order discretization \eqref{eq:lowpp}. While we have not observed issues using Rusanov estimate in our numerical experiments, we hope to explore more robust estimates of the maximum wavespeed \cite{guermond2016fast} in future works.}}
\begin{align}
    \lambda_{\max} \LRp{\vec{u}_L,\vec{u}_R,\vec{n}} = \max\LRp{\left|\vec{u}_L \cdot \vec{n}\right| + \sqrt{\gamma \frac{p_L}{\rho_L}}, \left|\vec{u}_R \cdot \vec{n}\right| + \sqrt{\gamma \frac{p_R}{\rho_R}}}. \label{eq:maxwvspd}
\end{align}

\rzero{
In all numerical experiments, we study the generalized positivity bound proposed in \eqref{eq:generalizedposbound}. All reported values of $\zeta$ in the numerical experiments refer to the relaxation factor in the generalized positivity bound.

Finally, we note that to enforce invsicid slip wall boundary conditions, we utilize the exact solution of the Riemann problem derived by Vegt and Ven \cite{van2002slip}, which was shown to be entropy stable \cite{hindenlang2020stability}. We point out that this imposition of boundary condition is not provably positivity preserving, but it greatly improves the robustness near wall compared to other weak impositions of reflective wall boundary conditions \cite{chen2017entropy, chan2022entropy}.
}

\subsection{Convergence tests}

The high order accuracy of entropy stable DG discretizations has been demonstrated in various works \cite{carpenter2014entropy, chan2018discretely}, and the convergence of low order positivity preserving methods for the compressible Euler equations was explored in detail in \cite{guermond2016invariant, pazner2021sparse}. Thus, in this section, we focus on studying the behaviour of the elementwise (\rzero{Zhang-Shu} type) limiting on test cases where the entropy stable discretization fails due to the loss of positivity. 

Most test cases in this section admit analytical solutions, and we evaluate the relative $L^p$ errors in the conservative variables using quadrature:
\begin{align}
    \frac{\LRs{\fnt{1}^T\fnt{M}\LRp{\vec{\rho}^n -\vec{\rho}}^p}^{1/p}}{\LRs{\fnt{1}^T\fnt{M}\vec{\rho}^p}^{1/p}} +     \frac{\LRs{\fnt{1}^T\fnt{M}\LRp{\vec{\rho u}^n -\vec{\rho u}}^p}^{1/p}}{\LRs{\fnt{1}^T\fnt{M}\LRp{\vec{\rho u}}^p}^{1/p}} +     \frac{\LRs{\fnt{1}^T\fnt{M}\LRp{\vec{\rho v}^n -\vec{\rho v}}^p}^{1/p}}{\LRs{\fnt{1}^T\fnt{M}\LRp{\vec{\rho v}}^p}^{1/p}} +     \frac{\LRs{\fnt{1}^T\fnt{M}\LRp{\vec{E}^n -\vec{E}}^p}^{1/p}}{\LRs{\fnt{1}^T\fnt{M}\vec{E}^p}^{1/p}} , \label{eq:error}
\end{align}
where the numerical solutions and exact solutions evaluated at quadrature nodes are denoted by $\vec{u}^n$ and $\vec{u}$ respectively.

\subsubsection{Leblanc shocktube}
\label{sec:leblanc}

We first consider the Leblanc shocktube problem for the compressible Euler equations. This is a challenging Riemann problem and entropy stable discretizations without limiting fail due to negative density and pressure. The domain is $[0,1]$, and the initial condition is
\begin{align}
    \vec{u}_0 \LRp{x} = \begin{cases}
      \vec{u}_L, \quad &x < x_0 \\
      \vec{u}_R, \quad &\text{otherwise}
    \end{cases}, \quad \vec{u}_L = \begin{bmatrix}
    \rho_L \\
    u_L \\
    p_L \\
    \end{bmatrix} = \begin{bmatrix}
    1.0 \\
    0.0 \\
    \LRp{\gamma-1}0.1 \\
    \end{bmatrix},\quad \vec{u}_R = \begin{bmatrix}
    \rho_R \\
    u_R \\
    p_R \\
    \end{bmatrix} = \begin{bmatrix}
    10^{-3} \\
    0.0 \\
    \LRp{\gamma-1}10^{-10} \\
    \end{bmatrix}, \label{eq:leblanc}
\end{align}
where $x_0 = 0.33, \gamma=\frac{5}{3}$. We set the exterior values at endpoints of the domain $x= 0, 1$ to be $\vec{u}_L, \vec{u}_R$ to enforce inhomogeneous Dirichlet boundary condition. The test case has an exact solution of form \cite{guermond2016invariant}
\begin{alignat}{4}
    & &&\vec{u} \LRp{x}  = \begin{cases}
      \vec{u}_L, \qquad &\xi \leq -\frac{1}{3} \\
      \begin{pmatrix}
      \rho^{**} & v^{**} & p^{**}
      \end{pmatrix}
      &-\frac{1}{3} < \xi \leq \lambda_1\\   
      \begin{pmatrix}
      \rho^{*}_L & \ v^{*} & p^{*}
      \end{pmatrix}
      &\lambda_1 < \xi \leq v^*\\  
      \begin{pmatrix}
      \rho^{*}_R & \ v^{*} & p^{*}
      \end{pmatrix}
      &v^* < \xi \leq \lambda_3\\
      \vec{u}_R, \qquad &\lambda_3 < \xi
    \end{cases}, \label{eq:lablancinit}\\
    \rho^{**} &= \LRp{0.75-0.75\xi}^3,\quad &&v^{**} = 0.75\LRp{\frac{1}{3}+\xi},\quad &&p^{**} = \frac{1}{15}\LRp{0.75-0.75\xi}^5, \nonumber \\
    \rho_L^* &= 5.40793353493162\times 10^{-2}, \quad &&\rho_R^* = 3.99999806043000 \times 10^{-3} ,\quad &&p^* = 0.515577927650970\times 10^{-3}, \nonumber \\
    v^* &= 0.621838671391735,\quad
    &&\lambda_1 = 0.495784895188979, \quad &&\lambda_3 = 0.829118362533470,\nonumber
\end{alignat}
where $\xi = \frac{x-x_0}{t}$. We discretize the domain by uniform intervals, set $\text{CFL} = 0.5$, and run the simulations until $T = 2/3$. We calculate the $L^1$ errors of different strategies for polynomial degrees $N = 2, 5$ and meshes with $K$ uniform elements.  

\rzero{
We compare the $L^1$ error and the convergence rate of the low order solution and the solutions using limiting with $\zeta = 0.1$ and $0.5$. As Table \ref{tab:Leblancconv} shows, all strategies are first order accurate, which is optimal for this test case. Figure \ref{fig:Leblanc.1} compares two different limited solutions with $\zeta = 0.1$ and $600$ degrees of freedom. 

Figure \ref{fig:Leblanc.1} and Table \ref{tab:Leblancconv} suggest that higher order approximations appear to produce less oscillatory solutions and generates more accurate results compared with lower order solutions. However, Figure \ref{fig:Leblanclow} compares low order solutions for a mesh with $600$ degrees of freedom. The results are nearly identical for polynomial degrees $N = 2$ and $N = 5$. The results of Table \ref{tab:Leblanlow} verify that the quality of the low-order solutions using sparsified SBP operators does not degrade as we increase the polynomial order \cite{pazner2021sparse}.}\footnote{\rone{We note that the low order solution does exhibit a staircasing-like effect near element boundaries. This has been observed in other discretizations which utilize the ``bar state'' reformulation \eqref{eq:lpprewrite} (personal communication with Guermond, Popov, and Maier, March 2022). We speculate that this may have to do with the magnitudes of the diagonal mass matrix entries varying over a high order element, as they decrease near element boundaries.} 
}

\begin{table}[!htb]
\qquad \qquad
    \begin{subtable}{.5\linewidth}
      \centering
\begin{tabular}{|c|c|c|c|c|} 
 \hline
  & \multicolumn{2}{|c|}{Low order, $N=2$} & \multicolumn{2}{|c|}{Low order, $N=5$} \\
 \hline
 K & $L^1$ error & Rate & $L^1$ error & Rate \\  
 \hline
 50  & $2.115\times 10^{-1}$ &       & $1.705\times 10^{-1}$ &      \\ 
 100 & $1.664\times 10^{-1}$ & 0.35  & $1.116\times 10^{-1}$ & 0.61 \\
 200 & $1.117\times 10^{-1}$ & 0.57  & $7.382\times 10^{-2}$ & 0.60 \\
 400 & $7.275\times 10^{-2}$ & 0.62  & $4.627\times 10^{-2}$ & 0.67 \\
 800 & $4.610\times 10^{-2}$ & 0.66  & $2.868\times 10^{-2}$ & 0.69 \\ 
 \hline
\end{tabular}
        \caption{Low order method}
        \label{tab:Leblanlow}
    \end{subtable}%
    \newline
\begin{subtable}{.5\linewidth}
      \centering
      \rzero{
\begin{tabular}{|c|c|c|c|c|} 
 \hline
 & \multicolumn{2}{|c|}{Limited, $N=2$} & \multicolumn{2}{|c|}{Limited, $N=5$}\\
 \hline
 K   & $L^1$ error & Rate & $L^1$ error & Rate \\  
 \hline
 50  & $8.058\times 10^{-2}$ &      & $5.070\times 10^{-2}$ & \\ 
 100 & $3.506\times 10^{-2}$ & 1.20 & $1.236\times 10^{-2}$ & 2.04\\
 200 & $1.351\times 10^{-2}$ & 1.37 & $3.660\times 10^{-3}$ & 1.76\\
 400 & $6.193\times 10^{-3}$ & 1.13 & $1.227\times 10^{-3}$ & 1.58\\
 800 & $2.953\times 10^{-3}$ & 1.07 & $6.333\times 10^{-4}$ & 0.95\\ 
 \hline
\end{tabular}
        \caption{Elementwise (Zhang-Shu type) limiting with $\zeta = 0.1$}
        }
        \label{tab:Leblancli1}
    \end{subtable}%
    \begin{subtable}{.5\linewidth}
      \centering
      \rzero{
\begin{tabular}{|c|c|c|c|c|} 
 \hline
 & \multicolumn{2}{|c|}{Limited, $N=2$} & \multicolumn{2}{|c|}{Limited, $N=5$}\\
 \hline
 K   & $L^1$ error & Rate & $L^1$ error & Rate \\  
 \hline
 50  & $8.681\times 10^{-2}$ &      & $5.956\times 10^{-2}$ & \\ 
 100 & $3.658\times 10^{-2}$ & 1.25 & $1.436\times 10^{-2}$ & 2.05\\
 200 & $1.329\times 10^{-2}$ & 1.46 & $3.630\times10^{-3}$  & 1.98\\
 400 & $6.015\times 10^{-3}$ & 1.14 & $1.129\times 10^{-3}$ & 1.69\\
 800 & $2.910\times 10^{-3}$ & 1.05 & $5.889\times 10^{-4}$ & 0.94\\ 
 \hline
\end{tabular}
        \caption{Elementwise (Zhang-Shu type) limiting with $\zeta = 0.5$}
        }
        \label{tab:Leblancli2}
    \end{subtable}
    \caption{Leblanc shocktube convergence tables}
    \label{tab:Leblancconv}
\end{table}

\begin{figure}[!htb]
\centering
\begin{subfigure}{.5\textwidth}
  \centering
  \includegraphics[width=.9\linewidth]{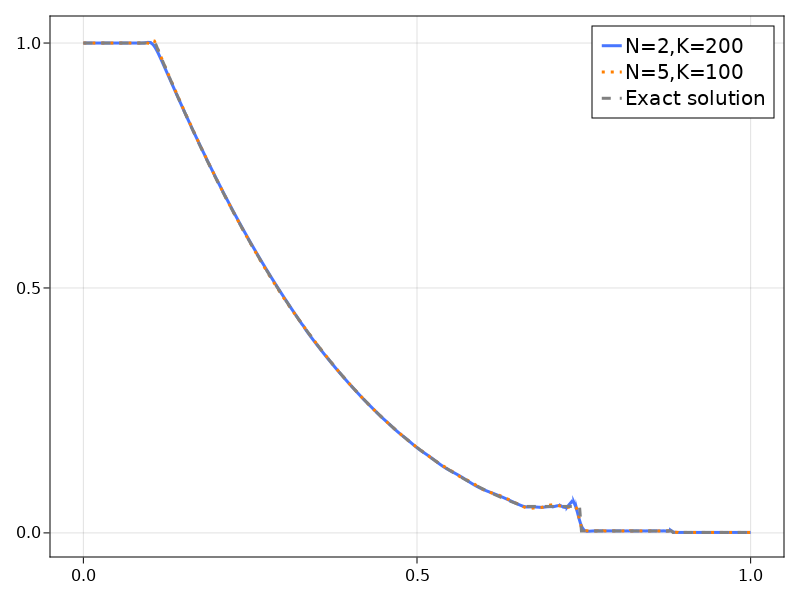}
  \caption{Solutions}
\end{subfigure}%
\begin{subfigure}{.5\textwidth}
  \centering
  \includegraphics[width=.9\linewidth]{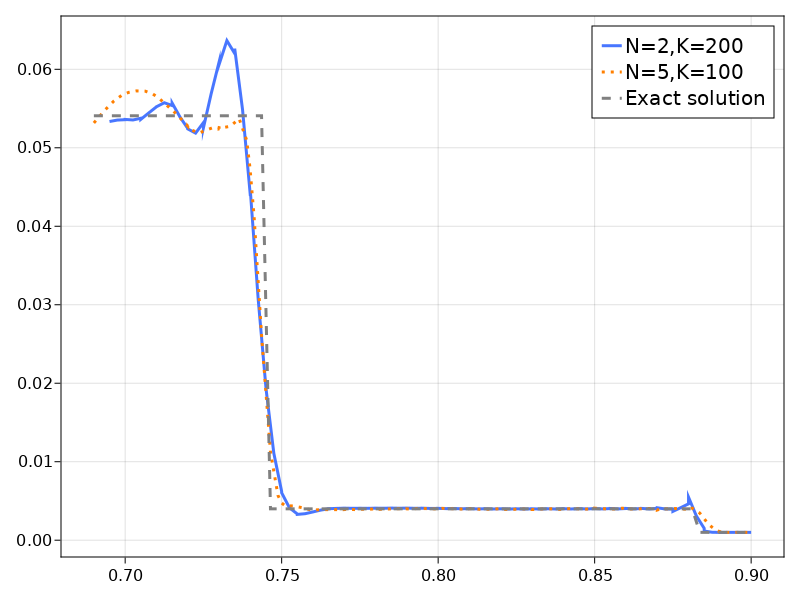}
  \caption{Zoom in view}
\end{subfigure}
\caption{Leblanc shocktube, Elementwise (Zhang-Shu type) limiting with $\zeta = 0.1$}
\label{fig:Leblanc.1}
\end{figure}

\begin{figure}[!htb]
\centering
\begin{subfigure}{.5\textwidth}
  \centering
  \includegraphics[width=.9\linewidth]{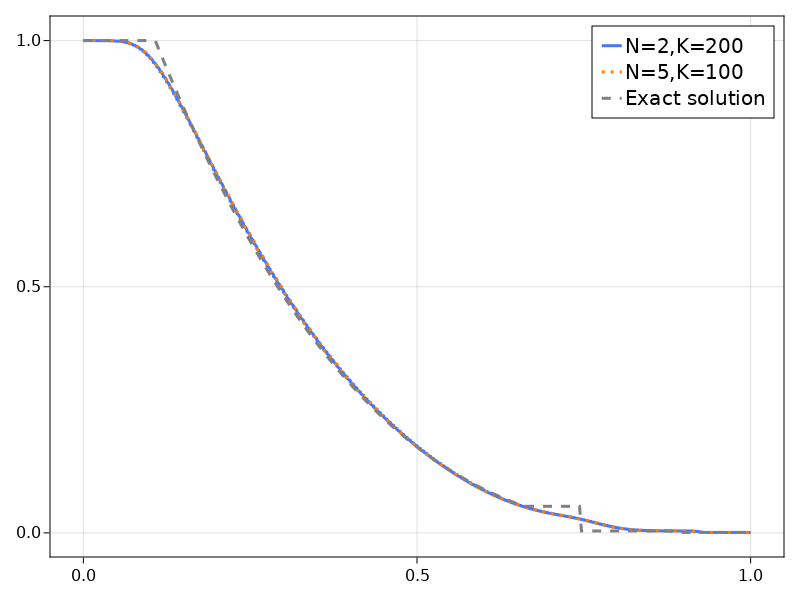}
  \caption{Solutions}
\end{subfigure}%
\begin{subfigure}{.5\textwidth}
  \centering
  \includegraphics[width=.9\linewidth]{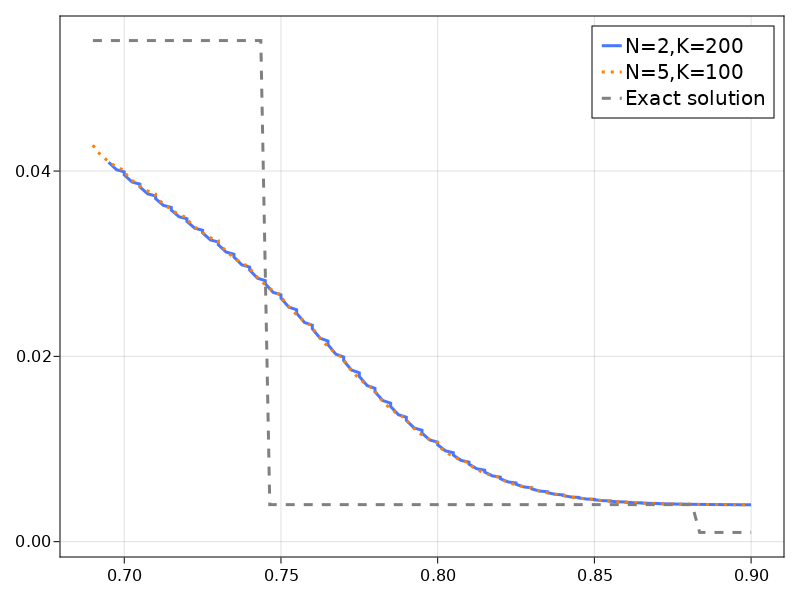}
  \caption{Zoom in view}
\end{subfigure}
\caption{Leblanc shocktube, low order solutions}
\label{fig:Leblanclow}
\end{figure}

\subsubsection{Viscous shockwave}
\label{sec:viscshockwave}

We next consider a viscous \rzero{shockwave} problem for the compressible Navier-Stokes equations \cite{guermond2021second}. The test case starts from a steady state solution of the compressible Navier-Stokes equation, and the solution is translated with constant velocity $u_\infty$. The domain is $[-1,1.5]$, and the analytical solution is defined as
\begin{align}
    \vec{u} \LRp{x,t} &= \begin{bmatrix}
        \rho\LRp{\xi} \\
        \rho\LRp{\xi}\LRp{u_\infty + u \LRp{\xi}}  \\
        \rho\LRp{\xi} \LRp{e \LRp{\xi} + \frac{1}{2} \LRp{u_\infty+u\LRp{\xi}}^2}
    \end{bmatrix}, \qquad \xi = x - u_\infty t, \nonumber \\
    \rho\LRp{x} &= \frac{m_0}{u\LRp{x}}, \qquad e\LRp{x} = \frac{1}{2\gamma} \LRp{\frac{\gamma+1}{\gamma-1} u_0^2 - u\LRp{x}^2}, \label{eq:viscshockwave}
\end{align}
where the velocity profile $u\LRp{x}$ is defined implicitly by the equation
\begin{align}
    x = \frac{2\kappa}{\LRp{\gamma+1}m_0} \LRs{\frac{u_L}{u_L-u_R} \log\LRp{\frac{u_L - u\LRp{x}}{u_L-u_0}} - \frac{u_R}{u_L-u_R} \log\LRp{\frac{u\LRp{x}-u_R}{u_0 - u_R}}}. \label{eq:viscshockwaveu}
\end{align}
Here, $u_L,u_R$ denote the velocity at $-\infty,\infty$ respectively. We assume the velocity at $\infty$ depends on the pre-shock Mach number $M_0$, $u_R = \frac{\gamma-1+2/M_0^2}{\gamma+1}$. We define $u_0 = \sqrt{u_L u_R}$ and $\text{Pr} = \frac{3}{4}$.

First, we verify the convergence of the low order positivity preserving discretizations. We set the parameters as follow: $\gamma = 1.4, \mu = 0.01, u_\infty = 0.2, u_L = 1.0, m_0 = 1.0, M_0 = 3$. We enforce inhomogeneous Dirichlet boundary condition on boundaries. We discretize the domain into uniform intervals, set $\text{CFL} = 0.5$, and run each simulation until $T = 1.0$. We compute the $L^1$ and $ L^2$ errors of the low order positivity-preserving scheme. As Table \ref{tab:viscshocktubeM03} shows, the low-order method is first order accurate for $N = 2,3,4$.

\begin{table}[!htb]
\begin{tabular}{|c|c|c|c|c|} 
 \hline
  \multicolumn{5}{|c|}{$N=2$} \\
 \hhline{|=|=|=|=|=|}
 K & $L^1$ error & Rate & $L^2$ error & Rate \\  
 \hline
 50  & $7.69\times 10^{-2}$ &        & $1.73\times 10^{-1}$ &        \\ 
 100 & $3.98\times 10^{-2}$ & $0.95$ & $1.05\times 10^{-1}$ & $0.72$ \\ 
 200 & $2.17\times 10^{-2}$ & $0.87$ & $6.34\times 10^{-2}$ & $0.73$ \\ 
 400 & $1.14\times 10^{-2}$ & $0.93$ & $3.62\times 10^{-2}$ & $0.81$ \\ 
 \hline
 \multicolumn{5}{|c|}{$N=3$} \\
 \hhline{|=|=|=|=|=|}
 50  & $5.97\times 10^{-2}$ &        & $1.42\times 10^{-1}$ &        \\
 100 & $3.32\times 10^{-2}$ & $0.84$ & $9.01\times 10^{-2}$ & $0.66$ \\
 200 & $1.70\times 10^{-2}$ & $0.96$ & $5.21\times 10^{-2}$ & $0.79$ \\
 400 & $8.70\times 10^{-3}$ & $0.97$ & $2.86\times 10^{-2}$ & $0.86$ \\
 \hline
 \multicolumn{5}{|c|}{$N=4$} \\
 \hhline{|=|=|=|=|=|}
  50 & $5.16\times 10^{-2}$ &        & $1.26\times 10^{-1}$ & \\  
 100 & $2.70\times 10^{-2}$ & $0.93$ & $7.70\times 10^{-2}$ & $0.71$\\
 200 & $1.39\times 10^{-2}$ & $0.96$ & $4.37\times 10^{-2}$ & $0.82$\\
 400 & $7.09\times 10^{-3}$ & $0.97$ & $2.36\times 10^{-2}$ & $0.89$\\
 \hline
\end{tabular}
  \caption{Viscous shockwave $M_0 = 3$, Low order solutions}%
\label{tab:viscshocktubeM03}
\end{table}


\rone{
Next we verify the convergence of the elementwise limited solutions. First, we set the parameters as before: $\gamma = 1.4, \mu = 0.01, u_\infty = 0.2, u_L = 1.0, m_0 = 1.0, M_0 = 3$. We discretize the domain using uniform intervals, set $\text{CFL} = 0.5$, and run the simulations until $T = 1$. For this case, we observe the limiting is activated only when the mesh is very coarse ($K = 10$), and we recover high order rates of convergence as we refine the mesh. Tables ~\ref{tab:1Dviscwave.1} and \ref{tab:1Dviscwave.5} show $O(h^N)$ to $O(h^{N+1/2})$ convergence rate when the relaxation factor is taken to be $\zeta = 0.1$ or $\zeta = 0.5$. 
}

\begin{table}[!htb]
\centering
\rone{
\begin{tabular}{|c|c|c |c|c|c|c|c|c|} 
 \hline
 & \multicolumn{2}{|c|}{$N=2$} & \multicolumn{2}{|c|}{$N=3$} & \multicolumn{2}{|c|}{$N=4$} & \multicolumn{2}{|c|}{$N=5$}\\
 \hline
 K & $L^2$ error & Rate & $L^2$ error & Rate & $L^2$ error & Rate & $L^2$ error & Rate \\  
 \hline
 10  & $1.089\times 10^{-1}$  &      & $6.245\times 10^{-2}$ &      & $6.281\times 10^{-2}$ &      & $4.378\times 10^{-2}$ & \\ 
 20  & $3.486\times 10^{-2}$  & 1.64 & $3.022\times 10^{-2}$ & 1.05 & $7.368\times 10^{-3}$ & 3.09 & $6.972\times 10^{-3}$ & 2.65\\
 40  & $1.066\times 10^{-2}$  & 1.71 & $3.847\times 10^{-3}$ & 2.97 & $1.370\times 10^{-3}$ & 2.43 & $9.533\times 10^{-4}$ & 2.87\\
 80  & $1.821\times 10^{-3}$  & 2.55 & $4.612\times 10^{-4}$ & 3.06 & $6.458\times 10^{-5}$ & 4.41 & $4.667\times 10^{-5}$ & 4.35\\
 160 & $2.604\times 10^{-4}$  & 2.81 & $2.873\times 10^{-5}$ & 4.00 & $2.579\times 10^{-6}$ & 4.65 & $6.652\times 10^{-7}$ & 6.13\\ 
 \hline
\end{tabular}
}
\caption{Viscous shockwave $M_0 = 3.0$, Elementwise (Zhang-Shu type) limiting with $\zeta = 0.1$}
\label{tab:1Dviscwave.1}
\end{table}

\begin{table}[!htb]
\centering
\rone{
\begin{tabular}{|c|c|c |c|c|c|c|c|c|} 
 \hline
 & \multicolumn{2}{|c|}{$N=2$} & \multicolumn{2}{|c|}{$N=3$} & \multicolumn{2}{|c|}{$N=4$} & \multicolumn{2}{|c|}{$N=5$}\\
 \hline
 K & $L^2$ error & Rate & $L^2$ error & Rate & $L^2$ error & Rate & $L^2$ error & Rate \\  
 \hline
 10  & $9.139\times 10^{-2}$  &      & $5.874\times 10^{-2}$ &      & $6.324\times 10^{-2}$ &      & $4.370\times 10^{-2}$ & \\ 
 20  & $3.481\times 10^{-2}$  & 1.39 & $3.022\times 10^{-2}$ & 0.96 & $7.368\times 10^{-3}$ & 3.10 & $6.972\times 10^{-3}$ & 2.65\\
 40  & $1.066\times 10^{-2}$  & 1.71 & $3.847\times 10^{-3}$ & 2.97 & $1.370\times 10^{-3}$ & 2.43 & $9.533\times 10^{-4}$ & 2.87\\
 80  & $1.821\times 10^{-3}$  & 2.55 & $4.612\times 10^{-4}$ & 3.06 & $6.458\times 10^{-5}$ & 4.41 & $4.667\times 10^{-5}$ & 4.35\\
 160 & $2.604\times 10^{-4}$  & 2.81 & $2.873\times 10^{-5}$ & 4.00 & $2.579\times 10^{-6}$ & 4.65 & $6.652\times 10^{-7}$ & 6.13\\ 
 \hline
\end{tabular}
}
\caption{Viscous shockwave $M_0 = 3.0$, Elementwise (Zhang-Shu type) limiting with $\zeta = 0.5$}
\label{tab:1Dviscwave.5}
\end{table}

\rzero{
Next, we set the viscous shock tube parameters as: $\gamma = 1.4, \mu = 0.001, u_\infty = 0.2, u_L = 1.0, m_0 = 1.0, M_0 = 20.0$ and enforce the inhomogeneous Dirichlet boundary condition on boundaries. Under this set of parameters, high-order entropy-stable discretizations will fail due to negative density and pressure on coarse meshes.

We compare the $L^1$ error and the convergence rate of the low-order solution and the solutions using elementwise (Zhang-Shu type) limiting with $\zeta = 0.1$ and $\zeta = 0.5$. As Table \ref{tab:viscwavelow} shows, the low order solution is first order accurate. Elementwise limited solutions yield between $O\LRp{h^N}$ and $O\LRp{h^{N+1/2}}$ convergence rates when the limiting is activated on coarse meshes. When $K = 1600$, the limited solutions are virtually identical to unlimited high order ESDG solutions, for which we have verified high order convergence in \cite{chan2022entropy}. Figure \ref{fig:viscwave.1} compares different elementwise limited solutions with $\zeta = 0.1$. As in the Leblanc shocktube test case, we observe higher-order approximations result in less oscillatory solutions. 
}

\begin{table}[!htb]
\qquad\qquad
    \begin{subtable}{.5\linewidth}
      \centering
\begin{tabular}{|c|c|c|c|c|} 
 \hline
 & \multicolumn{2}{|c|}{Low order, $N=2$} & \multicolumn{2}{|c|}{Low order, $N=3$}\\
 \hline
 K & $L^1$ error & Rate & $L^1$ error & Rate\\  
 \hline
 50  & $9.000\times 10^{-2}$ &      & $7.061\times 10^{-2}$ &      \\ 
 100 & $4.739\times 10^{-2}$ & 0.93 & $3.711\times 10^{-2}$ & 0.92 \\
 200 & $2.455\times 10^{-2}$ & 0.95 & $1.939\times 10^{-2}$ & 0.94 \\
 400 & $1.281\times 10^{-2}$ & 0.94 & $1.006\times 10^{-2}$ & 0.95 \\
 800 & $6.599\times 10^{-3}$ & 0.96 & $5.179\times 10^{-3}$ & 0.96 \\ 
1600 & $3.394\times 10^{-3}$ & 0.96 & $2.699\times 10^{-3}$ & 0.94 \\ 
 \hline
\end{tabular}
        \caption{Low order method}
        \label{tab:viscwavelow}
    \end{subtable}%
    \newline
    \rzero
    {
        \begin{subtable}{.5\linewidth}
      \centering
\begin{tabular}{|c|c|c|c|c|} 
 \hline
 & \multicolumn{2}{|c|}{Limited, $N=2$} & \multicolumn{2}{|c|}{Limited, $N=3$}\\
 \hline
 K & $L^1$ error & Rate & $L^1$ error & Rate\\  
 \hline
 50  & $4.753\times 10^{-2}$ &      & $3.272\times 10^{-2}$ & \\ 
 100 & $3.323\times 10^{-2}$ & 0.52 & $1.568\times 10^{-2}$ & 1.06\\
 200 & $1.349\times 10^{-2}$ & 1.30 & $6.788\times 10^{-3}$ & 1.21\\
 400 & $3.862\times 10^{-3}$ & 1.80 & $1.009\times 10^{-3}$ & 2.75\\
 800 & $5.768\times 10^{-4}$ & 2.74 & $1.163\times 10^{-4}$ & 3.12\\ 
1600 & $8.836\times 10^{-5}$ & 2.71 & $1.269\times 10^{-5}$ & 3.20\\ 
 \hline
\end{tabular}
        \caption{Elementwise (Zhang-Shu type) limiting with $\zeta = 0.1$}
        \label{tab:viscwavezeta1}
    \end{subtable}%
    \begin{subtable}{.5\linewidth}
      \centering
\begin{tabular}{|c|c|c|c|c|} 
 \hline
 & \multicolumn{2}{|c|}{Limited, $N=2$} & \multicolumn{2}{|c|}{Limited, $N=3$}\\
 \hline
 K & $L^1$ error & Rate & $L^1$ error & Rate\\  
 \hline
 50  & $4.209\times 10^{-2}$ &      & $3.987\times 10^{-2}$ & \\ 
 100 & $2.305\times 10^{-2}$ & 0.87 & $2.071\times 10^{-2}$ & 0.94\\
 200 & $9.858\times 10^{-2}$ & 1.23 & $6.749\times 10^{-3}$ & 1.62\\
 400 & $3.382\times 10^{-3}$ & 1.54 & $1.278\times 10^{-3}$ & 2.40\\
 800 & $5.765\times 10^{-4}$ & 2.55 & $1.163\times 10^{-4}$ & 3.45\\ 
1600 & $8.836\times 10^{-5}$ & 2.71 & $1.269\times 10^{-5}$ & 3.20\\ 
 \hline
\end{tabular}
        \caption{Elementwise (Zhang-Shu type) limiting with $\zeta = 0.5$}
        \label{tab:viscwavezeta5}
    \end{subtable}
    }
        \caption{Viscous shockwave $M_0 = 20.0$, convergence tables}
    \label{tab:viscwave}
\end{table}

\begin{figure}[!htb]
\centering
\begin{subfigure}{.5\textwidth}
  \centering
  \includegraphics[width=.9\linewidth]{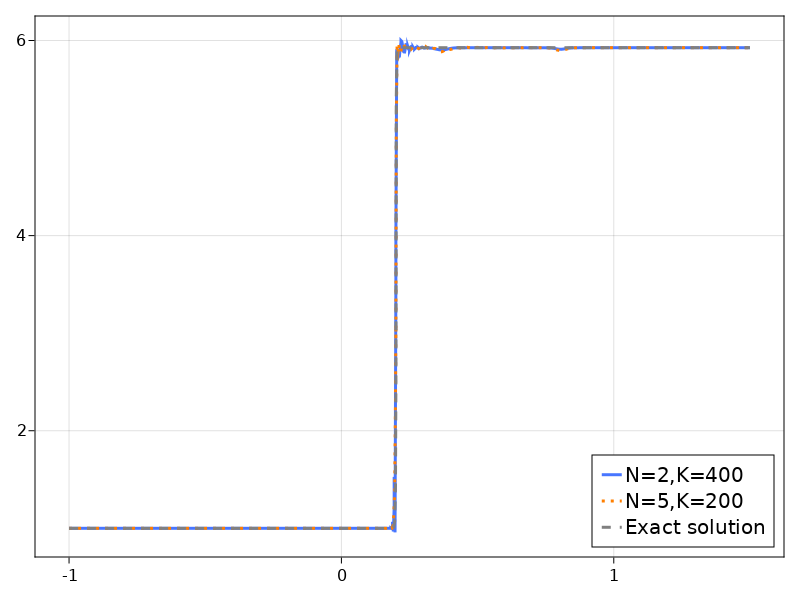}
  \caption{Solutions}
\end{subfigure}%
\begin{subfigure}{.5\textwidth}
  \centering
  \includegraphics[width=.9\linewidth]{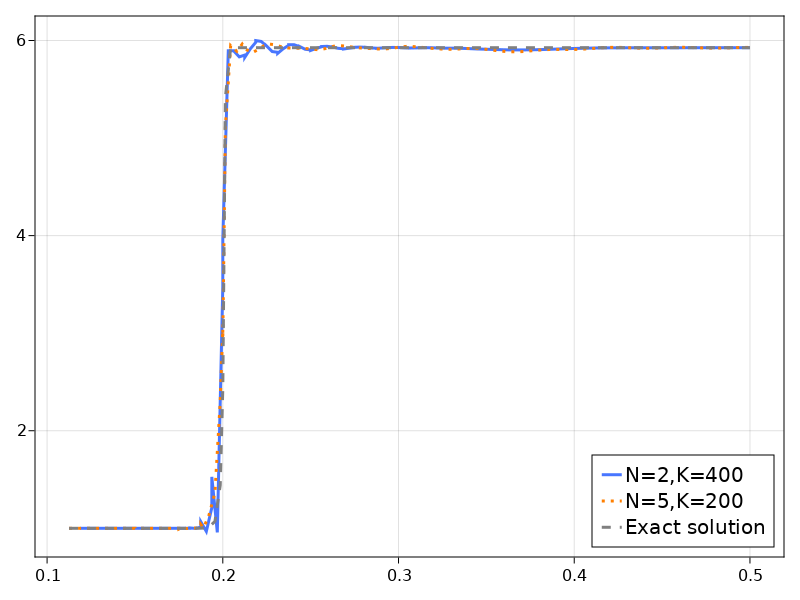}
  \caption{Zoom in view}
\end{subfigure}
\caption{Viscous shockwave, Elementwise (Zhang-Shu type) limiting with $\zeta = 0.1$}
\label{fig:viscwave.1}
\end{figure}


We have also performed convergence studies using the viscous shockwave for the 2D compressible Navier-Stokes equations. These results are included in Appendix~\ref{app:expCNSconv}.

\rone{
\subsection{Sine-shock interaction}
\label{sec:shuosher}

We now consider the sine-shock interaction problem \cite{shu1989efficient}. This problem illustrates the behaviour of the proposed limiting strategy for both smooth and non-smooth solution features. The initial condition is 
\begin{align}
    \vec{u}\LRp{x} = \begin{cases}
        \vec{u}_L, \quad &x < 4.0\\
        \vec{u}_R, \quad &\text{otherwise}
    \end{cases}, \quad \vec{u}_L = \begin{bmatrix}
    \rho_L \\
    u_L \\
    p_L \\
    \end{bmatrix} = \begin{bmatrix}
    3.857143 \\
    2.629369 \\
    10.3333 \\
    \end{bmatrix},\quad \vec{u}_R = \begin{bmatrix}
    \rho_R \\
    u_R \\
    p_R \\
    \end{bmatrix} = \begin{bmatrix}
    1+.2\sin\LRp{5x} \\
    0.0 \\
    1.0 \\
    \end{bmatrix}. \label{eq:shuosher}
\end{align}
We assign $\vec{u}_L$ as exterior value on the left boundary $x = -5$, and impose no boundary condition on the right boundary $x = 5$. 

We discretize the domain by $K = 64$ and $K = 128$ uniform intervals, and use polynomial degree $N = 3$ and ${\rm CFL} = 0.5$. We run each simulation until $T = 1.8$, , including a reference solution using a $5$-th order WENO scheme with 25000 cells \cite{shu2009high}. Each mesh is run with different relaxation factors $\zeta$. From Figure \ref{fig:shuosher64} and \ref{fig:shuosher128}, we observe all limited solutions are close to the reference solution. Because we only limit for positivity and not for monotonicity or a minimum entropy principle, we observe spurious oscillations near shocks. These oscillations are mitigated somewhat by using a larger relaxation factor.

\begin{figure}[!htb]
\centering
\begin{subfigure}{.5\textwidth}
  \centering
  \includegraphics[width=.9\linewidth]{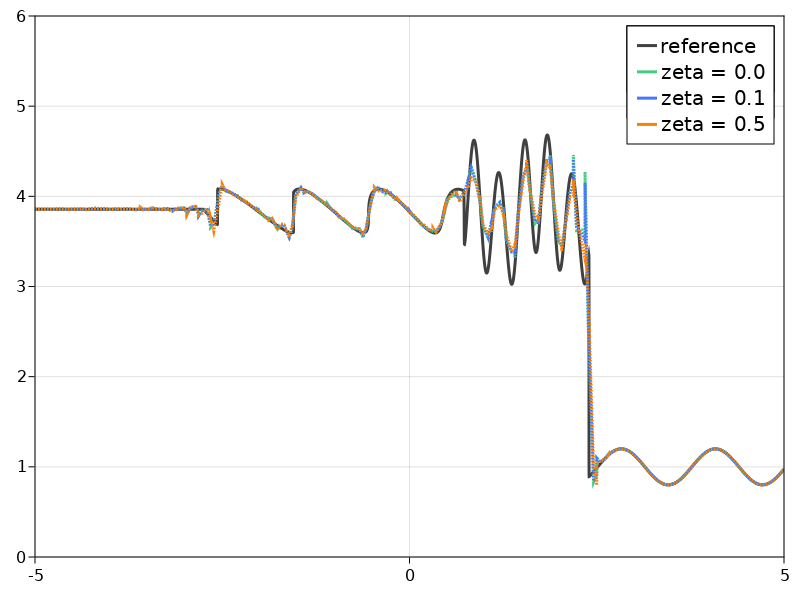}
  \caption{Solutions}
\end{subfigure}%
\begin{subfigure}{.5\textwidth}
  \centering
  \includegraphics[width=.9\linewidth]{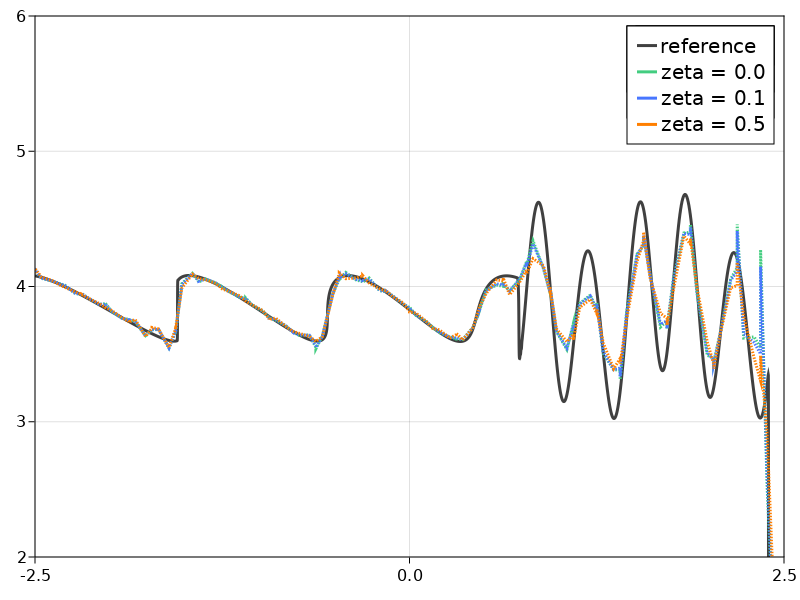}
  \caption{Zoom in view}
\end{subfigure}
\caption{Sine-shock interaction $N = 3, K = 64$}
\label{fig:shuosher64}
\end{figure}

\begin{figure}[!htb]
\centering
\begin{subfigure}{.5\textwidth}
  \centering
  \includegraphics[width=.9\linewidth]{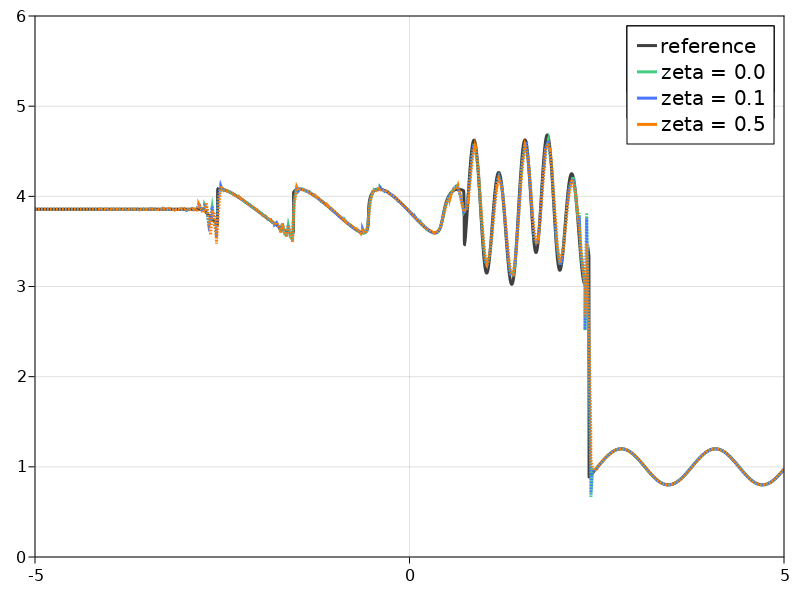}
  \caption{Solutions}
\end{subfigure}%
\begin{subfigure}{.5\textwidth}
  \centering
  \includegraphics[width=.9\linewidth]{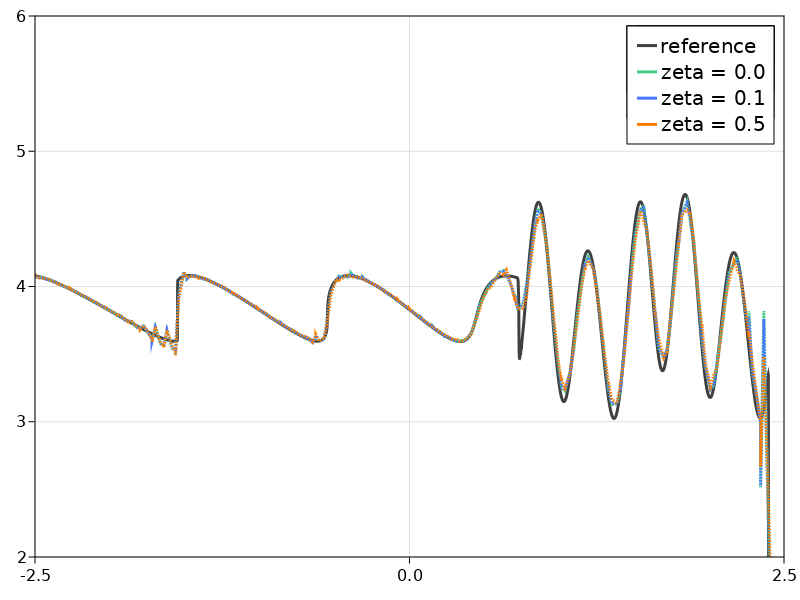}
  \caption{Zoom in view}
\end{subfigure}
\caption{Sine-shock interaction $N = 3, K = 128$}
\label{fig:shuosher128}
\end{figure}

\subsubsection{Isentropic vortex}
\label{sec:vortex}

We now test the convergence of the positivity-preserving limited scheme in 2D. We examine the convergence of the elementwise limited solutions in 2D using the isentropic vortex test case \cite{shu1988efficient}. The domain is $[0,20]\times[0,10]$, and the exact solution in primitive variables is defined as
\begin{align*}
    \vec{u} \LRp{x,t} = \begin{bmatrix}
        \rho\LRp{x,t} \\ u\LRp{x,t} \\ v\LRp{x,t} \\ p\LRp{x,t} 
    \end{bmatrix} = \begin{bmatrix}
    \LRs{1-\frac{1}{8\gamma \pi^2}\LRp{.5\LRp{\gamma-1}\LRp{\beta e^{1-r\LRp{x,t}^2}}^2}}^{1/\LRp{\gamma-1}} \\
    1 - \frac{\beta}{2\pi} e^{1-r\LRp{x,t}^2} \LRp{y-y_0} \\
    \frac{\beta}{2\pi} e^{1-r\LRp{x,t}^2} \LRp{y-y_0} \\
    \rho\LRp{x,t}^\gamma
    \end{bmatrix},\\
    r\LRp{x,t} = \sqrt{\LRp{x-x_0-t}^2 + \LRp{y-y_0}^2},
\end{align*}
where $\LRp{x_0,y_0} = \LRp{9.0,5.0}$ denotes the center of the vortex at time $t = 0$, and $\beta = 8.5$ denotes the strength of the vortex. Most numerical experiments in the literature choose the strength of the vortex $\beta = 5$, and the minimum density is $0.36$ \cite{chan2018discretely}. For $\beta = 8.5$, the maximum and minimum density in this test case are $1.0$ and $2.145 \times 10^{-3}$, respectively, and the unlimited entropy stable scheme fails due to large jumps in the density and pressure. 

\rzero{
We first test the convergence of the solutions on quadrilateral elements. We construct a uniform uniform quadrilateral mesh by discretizing the $x,y$ directions with $2K_{\text{1D}},K_{\text{1D}}$ uniform intervals. We enforce periodic boundary conditions on the domain and run the simulations until $T = 2.0$ with $\text{CFL} = 0.9$. Table \ref{tab:vortexquad1} and \ref{tab:vortexquad5} show that the elementwise limited solutions with $\zeta = 0.1$ and $\zeta= 0.5$ both have asymptotic convergence rates between $O(h^{N+1/2})$ and $O(h^{N+1})$. 
}

\begin{table}[!htb]
\centering
\rzero{
\begin{tabular}{|c|c|c |c|c|c|c|c|c|} 
 \hline
 & \multicolumn{2}{|c|}{$N=1$} & \multicolumn{2}{|c|}{$N=2$} & \multicolumn{2}{|c|}{$N=3$} & \multicolumn{2}{|c|}{$N=4$}\\
 \hline
 K & $L^2$ error & Rate & $L^2$ error & Rate & $L^2$ error & Rate & $L^2$ error & Rate \\  
 \hline
 2  & $2.134\times 10^{0}$  &      & $1.172\times 10^{0}$  &      & $1.511\times 10^{0}$  &      & $8.365\times 10^{-1}$ & \\ 
 4  & $1.410\times 10^{0}$  & 0.60 & $1.165\times 10^{0}$  & 0.01 & $5.718\times 10^{-1}$ & 1.40 & $3.359\times 10^{-1}$ & 1.32\\
 8  & $1.162\times 10^{0}$  & 0.28 & $4.603\times 10^{-1}$ & 1.34 & $1.609\times 10^{-1}$ & 1.83 & $9.425\times 10^{-2}$ & 1.83\\
 16 & $6.710\times 10^{-1}$ & 0.79 & $1.050\times 10^{-1}$ & 2.13 & $2.310\times 10^{-2}$ & 2.80 & $7.087\times 10^{-3}$ & 3.73\\
 32 & $3.004\times 10^{-1}$ & 1.16 & $1.727\times 10^{-2}$ & 2.61 & $2.477\times 10^{-3}$ & 3.22 & $1.915\times 10^{-4}$ & 5.21\\ 
 \hline
\end{tabular}
}
\caption{Isentropic vortex, quadrilateral mesh - Elementwise (Zhang-Shu type) limiting with $\zeta = 0.1$}
\label{tab:vortexquad1}
\end{table}

\begin{table}[!htb]
\centering
\rzero{
\begin{tabular}{|c|c|c |c|c|c|c|c|c|} 
 \hline
 & \multicolumn{2}{|c|}{$N=1$} & \multicolumn{2}{|c|}{$N=2$} & \multicolumn{2}{|c|}{$N=3$} & \multicolumn{2}{|c|}{$N=4$}\\
 \hline
 K & $L^2$ error & Rate & $L^2$ error & Rate & $L^2$ error & Rate & $L^2$ error & Rate \\  
 \hline
 2  & $2.134\times 10^{0}$  &      & $1.171\times 10^{0}$  &      & $1.443\times 10^{0}$  &      & $8.163\times 10^{-1}$ & \\ 
 4  & $1.310\times 10^{0}$  & 0.60 & $1.148\times 10^{0}$  & 0.03 & $5.958\times 10^{-1}$ & 1.28 & $4.073\times 10^{-1}$ & 1.00\\
 8  & $1.160\times 10^{0}$  & 0.28 & $4.865\times 10^{-1}$ & 1.24 & $1.905\times 10^{-1}$ & 1.64 & $8.987\times 10^{-2}$ & 2.18\\
 16 & $6.712\times 10^{-1}$ & 0.79 & $1.223\times 10^{-1}$ & 1.99 & $2.308\times 10^{-2}$ & 3.05 & $1.511\times 10^{-2}$ & 2.57\\
 32 & $3.009\times 10^{-1}$ & 1.16 & $1.706\times 10^{-2}$ & 2.84 & $2.393\times 10^{-3}$ & 3.27 & $1.915\times 10^{-4}$ & 6.30\\ 
 \hline
\end{tabular}
\caption{Isentropic vortex, quadrilateral mesh - Elementwise (Zhang-Shu type) limiting with $\zeta = 0.5$}
\label{tab:vortexquad5}
}
\end{table}

\rzero{
We next test convergence on simplicial elements. We construct each simplicial mesh by subdividing each element in a uniform quadrilateral mesh into two uniform triangles. We enforce periodic boundary conditions on the domain and run the simulations until $T = 2.0$ with $\text{CFL} = 0.5$. Table \ref{tab:vortextri1} and \ref{tab:vortextri5} show the optimal convergence rate between $O\LRp{h^{N+1/2}}$ to $O\LRp{h^{N+1}}$ on the simplicial mesh. 
}

\begin{table}[!htb]
\centering
\rzero{
\begin{tabular}{|c|c|c |c|c|c|c|c|c|} 
 \hline
 & \multicolumn{2}{|c|}{$N=1$} & \multicolumn{2}{|c|}{$N=2$} & \multicolumn{2}{|c|}{$N=3$} & \multicolumn{2}{|c|}{$N=4$}\\
 \hline
 K & $L^2$ error & Rate & $L^2$ error & Rate & $L^2$ error & Rate & $L^2$ error & Rate \\  
 \hline
 2  & $2.301\times 10^{0}$  &      & $1.004\times 10^{0}$  &      & $1.206\times 10^{0}$  &      & $7.700\times 10^{-1}$ & \\ 
 4  & $1.061\times 10^{0}$  & 1.12 & $7.873\times 10^{-1}$ & 0.35 & $4.773\times 10^{-1}$ & 1.33 & $3.998\times 10^{-1}$ & 0.95\\
 8  & $8.050\times 10^{-1}$ & 0.40 & $3.650\times 10^{-1}$ & 1.11 & $1.683\times 10^{-1}$ & 1.50 & $8.982\times 10^{-2}$ & 2.15\\
 16 & $4.445\times 10^{-1}$ & 0.86 & $8.736\times 10^{-2}$ & 2.06 & $2.842\times 10^{-2}$ & 2.57 & $9.365\times 10^{-3}$ & 3.26\\
 32 & $1.593\times 10^{-1}$ & 1.48 & $1.322\times 10^{-2}$ & 2.72 & $2.794\times 10^{-3}$ & 3.35 & $3.248\times 10^{-4}$ & 4.85\\ 
 \hline
\end{tabular}
\caption{Isentropic vortex, simplicial mesh - Elementwise (Zhang-Shu type) limiting with $\zeta = 0.1$}
\label{tab:vortextri1}
}
\end{table}

\begin{table}[!htb]
\centering
\rzero{
\begin{tabular}{|c|c|c |c|c|c|c|c|c|} 
 \hline
 & \multicolumn{2}{|c|}{$N=1$} & \multicolumn{2}{|c|}{$N=2$} & \multicolumn{2}{|c|}{$N=3$} & \multicolumn{2}{|c|}{$N=4$}\\
 \hline
 K & $L^2$ error & Rate & $L^2$ error & Rate & $L^2$ error & Rate & $L^2$ error & Rate \\  
 \hline
 2  & $2.297\times 10^{0}$  &      & $9.937\times 10^{-1}$ &      & $1.204\times 10^{0}$  &      & $7.391\times 10^{-1}$ & \\ 
 4  & $1.049\times 10^{0}$  & 1.13 & $7.887\times 10^{-1}$ & 0.33 & $5.034\times 10^{-1}$ & 1.26 & $4.059\times 10^{-1}$ & 0.86\\
 8  & $8.036\times 10^{-1}$ & 0.39 & $3.834\times 10^{-1}$ & 1.04 & $1.881\times 10^{-1}$ & 1.42 & $9.890\times 10^{-2}$ & 2.04\\
 16 & $4.434\times 10^{-1}$ & 0.86 & $8.993\times 10^{-2}$ & 2.09 & $2.944\times 10^{-2}$ & 2.68 & $1.578\times 10^{-2}$ & 2.65\\
 32 & $1.594\times 10^{-1}$ & 1.48 & $1.298\times 10^{-2}$ & 2.79 & $2.606\times 10^{-3}$ & 3.50 & $4.258\times 10^{-4}$ & 5.21\\ 
 \hline
\end{tabular}
\caption{Isentropic vortex, simplicial mesh - Elementwise (Zhang-Shu type) limiting with $\zeta = 0.5$}
\label{tab:vortextri5}
}
\end{table}

\subsection{Sedov blast wave}

We next run the Sedov blast wave problem \cite{fryxell2000flash} to test the proposed limiting strategy for the compressible Euler equations on both quadrilateral and simplicial meshes. The problem involves a large region of near-zero density and pressure, and is often used to evaluate the behaviour of numerical methods in near-vacuum regions. The domain is $\LRs{-1.5, 1.5}^2$, and the initial condition in primitive variables is \cite{fryxell2000flash}
\begin{align}
    \vec{u}\LRp{x,y} = \begin{cases}
        \vec{u}_{\rm int}, \quad &r < r_{0}\\
        \vec{u}_{\rm amb}, \quad &\text{otherwise}
    \end{cases}, \quad \vec{u}_I = \begin{bmatrix}
    \rho_{\rm int} \\
    u_{\rm int} \\
    v_{\rm int} \\
    p_{\rm int} \\
    \end{bmatrix} = \begin{bmatrix}
    1.0 \\
    0.0 \\
    0.0 \\
    \frac{\LRp{\gamma-1}E_0}{\pi r_{0}^2} \\
    \end{bmatrix},\quad \vec{u}_0 = \begin{bmatrix}
    \rho_{\rm amb} \\
    u_{\rm amb} \\
    v_{\rm amb} \\
    p_{\rm amb} \\
    \end{bmatrix} = \begin{bmatrix}
    1.0 \\
    0.0 \\
    0.0 \\
    10^{-5} \\
    \end{bmatrix}, \label{eq:sedov}
\end{align}
where we define $r = \sqrt{x^2 + y^2}$ and set $\gamma = 1.4, E_0 = 1.0, r_0 = 4h$, where $h$ is the mesh size. We discretize the domain with uniform quadrilateral and simplicial meshes as described in \ref{sec:vortex}, and we define the mesh size in both cases by $h = \frac{3}{K_{\rm 1D}}$. Periodic boundary conditions are enforced. 

We approximate the solution until final time $T=1$ with degree $N = 3$ polynomials. We use three different limiting configurations: 
\begin{enumerate}
\item $\zeta = 0.1$ without shock capturing,
\item $\zeta = 0.1$ with shock capturing, and
\item $\zeta = 0.5$ without shock capturing.
\end{enumerate}
We plot the density and its $10$ logarithmically spaced contours. We truncate the color range to $\LRs{0.01, 6}$ for clearer visualization. Figures \ref{fig:sedovquad} and \ref{fig:sedovtri} shows the results on quadrilateral and triangular meshes, respectively. The simulations remain robust in all configurations. Without shock capturing, when $\zeta = 0.1$, the simulations manifest spurious oscillations on both types of meshes. On the other hand, when shock capturing is activated, spurious oscillations are reduced on both meshes. Increasing the relaxation factor $\zeta$ to $0.5$ also suppresses the oscillations to some extent.

\begin{figure}[!htb]
\centering
\begin{subfigure}{.35\textwidth}
  \centering
  \includegraphics[width=.7\linewidth]{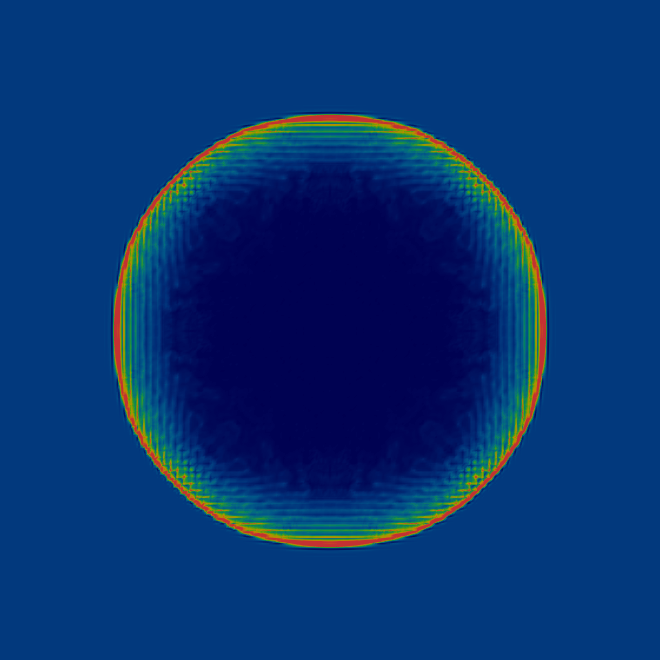}
\end{subfigure}%
\hspace{-1.5cm}
\begin{subfigure}{.35\textwidth}
  \centering
  \includegraphics[width=.7\linewidth]{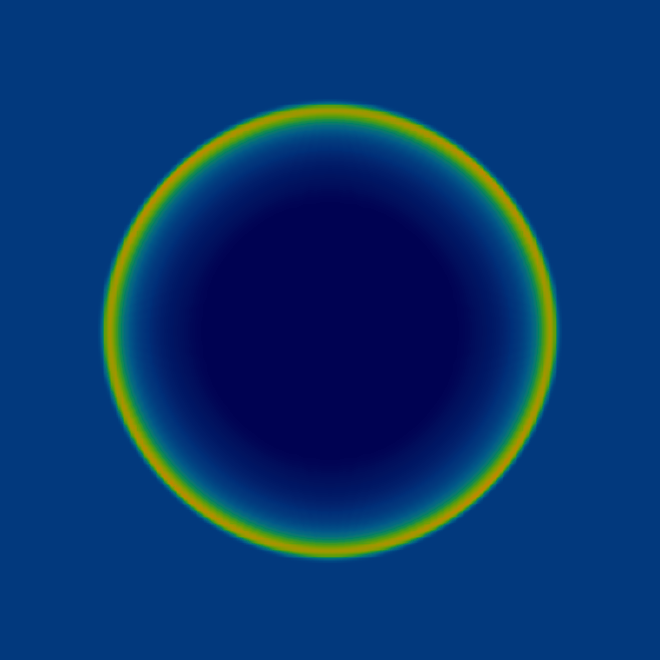}
\end{subfigure}
\hspace{-1.5cm}
\begin{subfigure}{.35\textwidth}
  \centering
  \includegraphics[width=.7\linewidth]{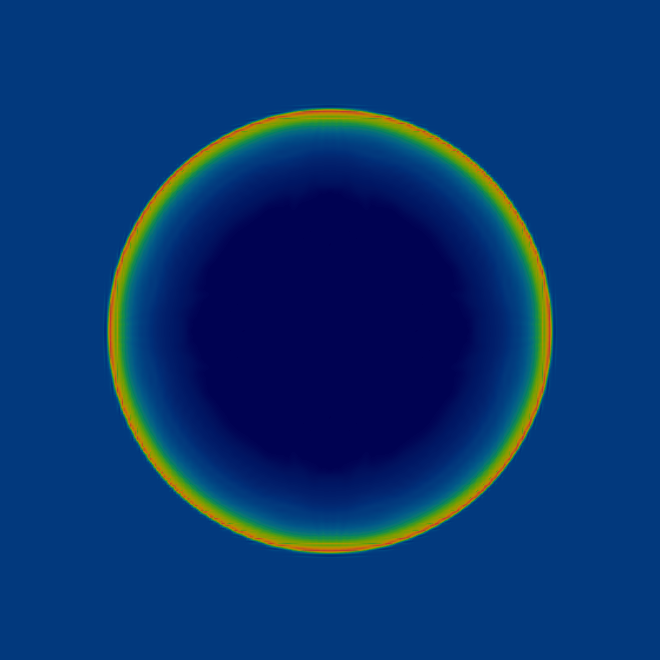}
\end{subfigure}
\newline
\begin{subfigure}{.35\textwidth}
  \centering
  \hspace{-0.7cm}
  \includegraphics[width=.7\linewidth]{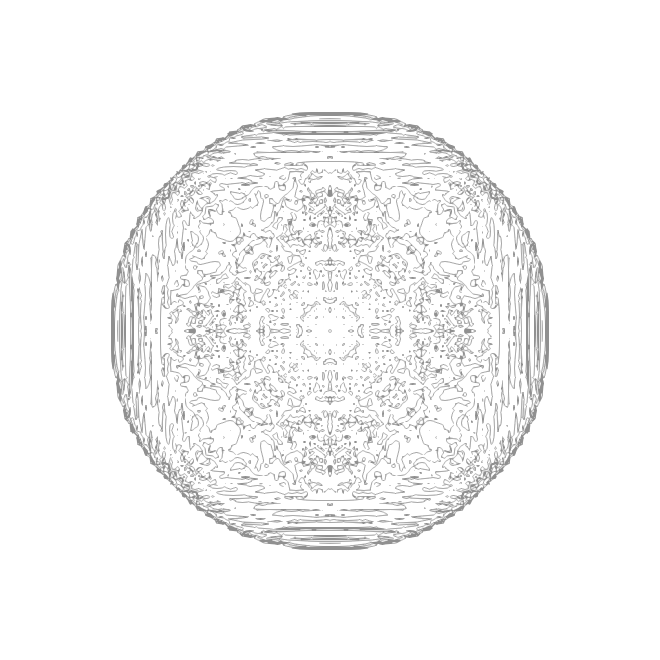}
  \caption{$\zeta = 0.1$, without shock capturing}
\end{subfigure}%
\hspace{-1.5cm}
\begin{subfigure}{.35\textwidth}
  \centering
  \hspace{-0.7cm}
  \includegraphics[width=.7\linewidth]{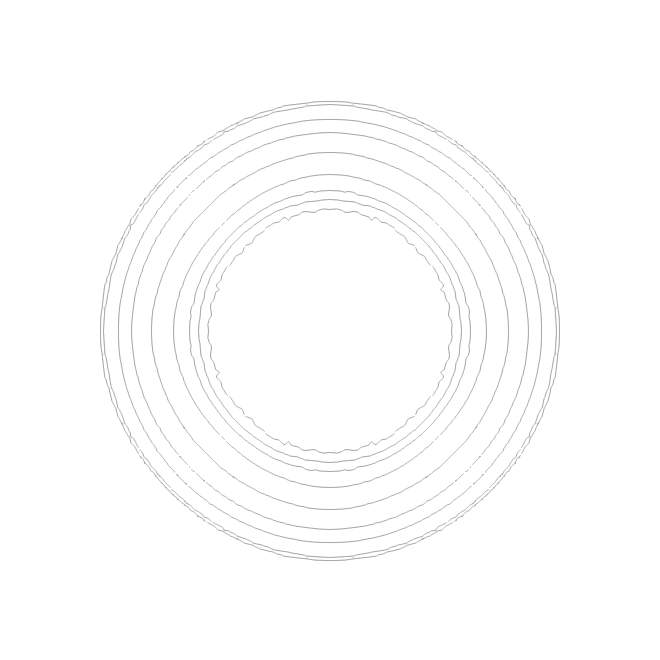}
  \caption{$\zeta = 0.1$, with shock capturing}
\end{subfigure}
\hspace{-1.5cm}
\begin{subfigure}{.35\textwidth}
  \centering
  \hspace{-0.7cm}
  \includegraphics[width=.7\linewidth]{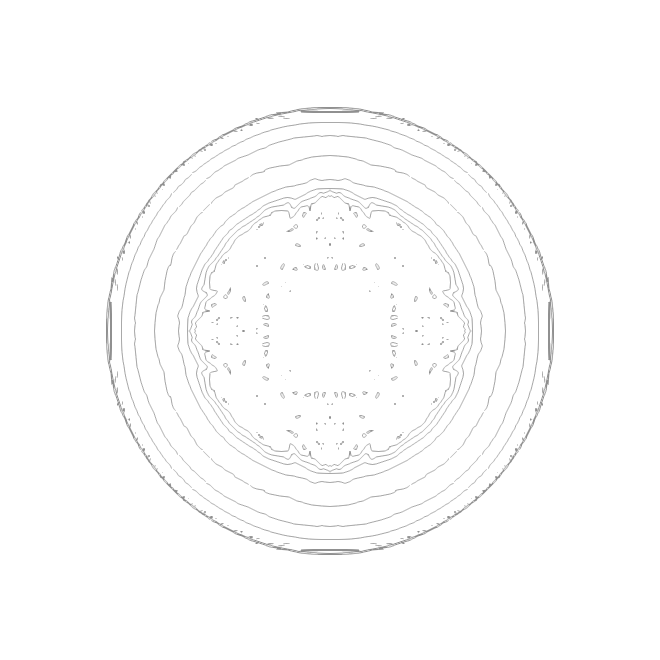}
  \caption{$\zeta = 0.5$, without shock capturing}
\end{subfigure}
\caption{Sedov blast wave, quadrilateral mesh, $N = 3, K_{\rm 1D} = 100$}
\label{fig:sedovquad}
\end{figure}

\begin{figure}[!htb]
\centering
\begin{subfigure}{.35\textwidth}
  \centering
  \includegraphics[width=.7\linewidth]{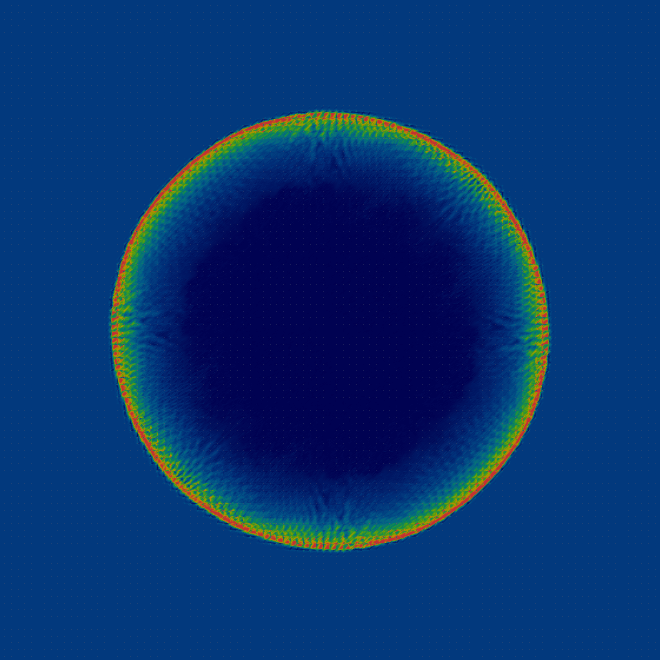}
\end{subfigure}%
\hspace{-1.5cm}
\begin{subfigure}{.35\textwidth}
  \centering
  \includegraphics[width=.7\linewidth]{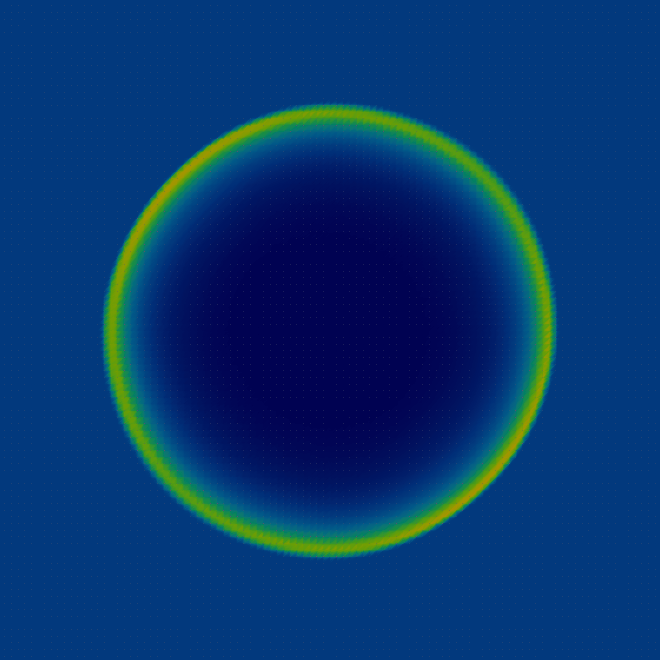}
\end{subfigure}
\hspace{-1.5cm}
\begin{subfigure}{.35\textwidth}
  \centering
  \includegraphics[width=.7\linewidth]{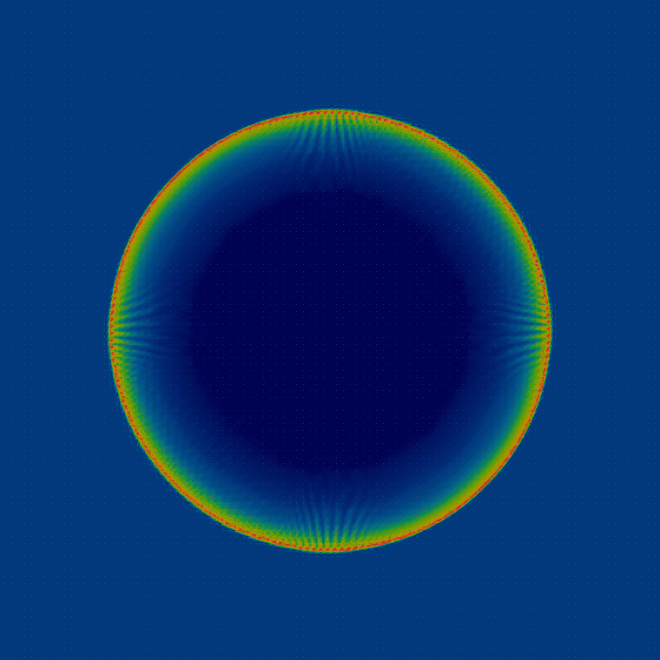}
\end{subfigure}
\newline
\begin{subfigure}{.35\textwidth}
  \centering
  \hspace{-0.7cm}
  \includegraphics[width=.7\linewidth]{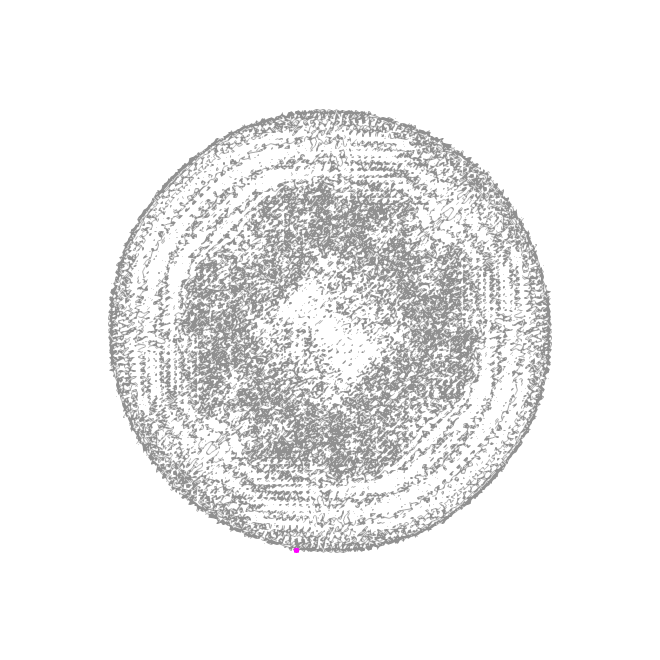}
  \caption{$\zeta = 0.1$, without shock capturing}
\end{subfigure}%
\hspace{-1.5cm}
\begin{subfigure}{.35\textwidth}
  \centering
  \hspace{-0.7cm}
  \includegraphics[width=.7\linewidth]{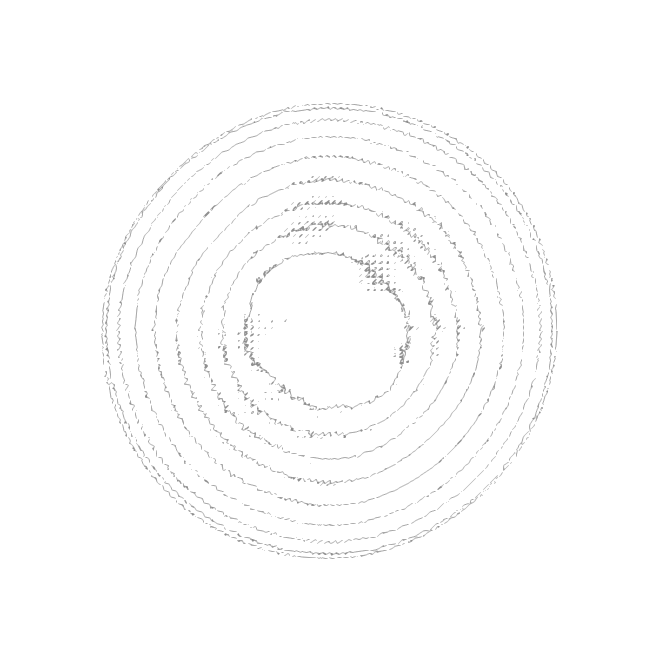}
  \caption{$\zeta = 0.1$, with shock capturing}
\end{subfigure}
\hspace{-1.5cm}
\begin{subfigure}{.35\textwidth}
  \centering
  \hspace{-0.7cm}
  \includegraphics[width=.7\linewidth]{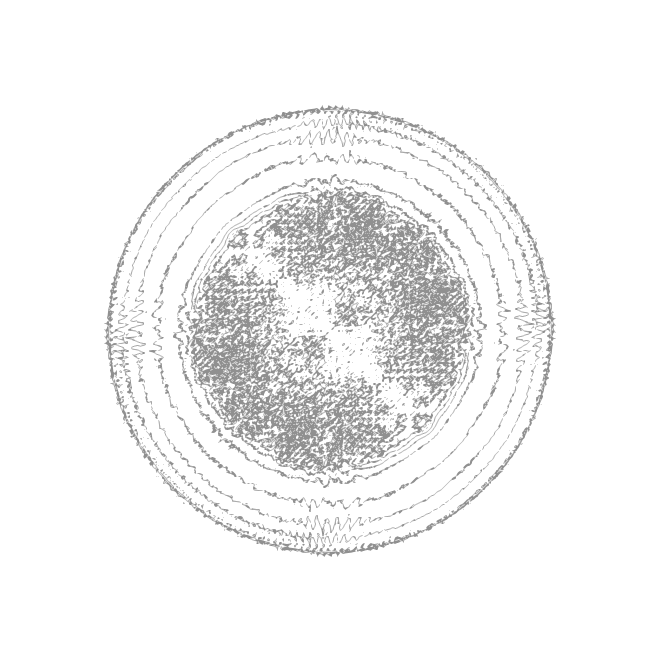}
  \caption{$\zeta = 0.5$, without shock capturing}
\end{subfigure}
\caption{Sedov blast wave, simplicial mesh, $N = 3, K_{\rm 1D} = 100$}
\label{fig:sedovtri}
\end{figure}

}

\subsection{Double Mach reflection}
\label{sec:dmr}

Next, we run the Double Mach Reflection \cite{woodward1984numerical} for the compressible Euler equations. The domain is $[0,3.5]\times[0,1]$, and the initial condition in primitive variables is
\begin{align}
    \vec{u}\LRp{x,y} = \begin{cases}
        \vec{u}_L, \quad &\xi > 0\\
        \vec{u}_R, \quad &\text{otherwise}
    \end{cases}, \quad \vec{u}_L = \begin{bmatrix}
    \rho_L \\
    u_L \\
    v_L \\
    p_L \\
    \end{bmatrix} = \begin{bmatrix}
    8.0 \\
    8.25\cos\LRp{\pi/6} \\
    -8.25\sin\LRp{\pi/6} \\
    116.5 \\
    \end{bmatrix},\quad \vec{u}_R = \begin{bmatrix}
    \rho_R \\
    u_R \\
    v_R \\
    p_R \\
    \end{bmatrix} = \begin{bmatrix}
    1.4 \\
    0.0 \\
    0.0 \\
    1.0 \\
    \end{bmatrix}, \label{eq:dmr}
\end{align}
where $\xi = y - \sqrt{3}x+\frac{\sqrt{3}}{6}$. We enforce reflective wall boundary conditions on $[\frac{1}{6},3.5]\times 0$ and assign $\vec{u}_L,\vec{u}_R$ as exterior values,
\begin{align}
    \vec{u}^+ \LRp{x,y} = \begin{cases}
    \vec{u}_L, \quad \text{if } x \in [0,\frac{1}{6}]\times \LRc{0} \ \bigcup \  \LRc{0}\times [0,1] \ \bigcup \  [0,s\LRp{t}]\times \LRc{1} \\
    \vec{u}_R, \quad \text{if } x\in [s(t),3.5]\times \LRc{1} \ \bigcup \  \LRc{3.5}\times [0,1]
    \end{cases}, \quad s\LRp{t} = \frac{1+\sqrt{3}/6}{\sqrt{3}} + \frac{10}{\cos\LRp{\pi/6}}t. \label{eq:dmrbc}
\end{align}

We discretize the domain by a uniform quadrilateral mesh with $875\times 250$ elements, with polynomial degree $N=3$. We run the simulation until $T = 0.2$ with \rzero{$\text{Re} = \infty$}. We truncate the color range to $[1,24]$ for clearer visualization, and we use $20$ contours linearly spaced between the interval $[1,24]$. 




\rzero{
Figure \ref{fig:dmr0.1woshock} shows results using Elementwise (Zhang-Shu type) limiting using the generalized positivity bound with $\zeta =0.1$. In the presence of strong shocks, the simulations remain robust. In addition, the simulations resolve fine scale vorticular behavior, which suggests that the proposed limiting strategy does not introduce excessive numerical dissipation. However, we observe numerical artifacts near shocks which perturb the vortices when $\zeta = 0.1$. The numerical oscillations can be either suppressed with additional shock capturing as Figure \ref{fig:dmr0.1wshock} shown, or increase the relaxation factor $\zeta = 0.5$ as Figure \ref{fig:dmr0.5woshock} shown. We observe in both cases the spurious oscillations near the shock are suppressed, and fine scale features are still well resolved with extra dissipation introduced. 
}

\begin{figure}[H]
  \centering
  \includegraphics[width=.8\linewidth]{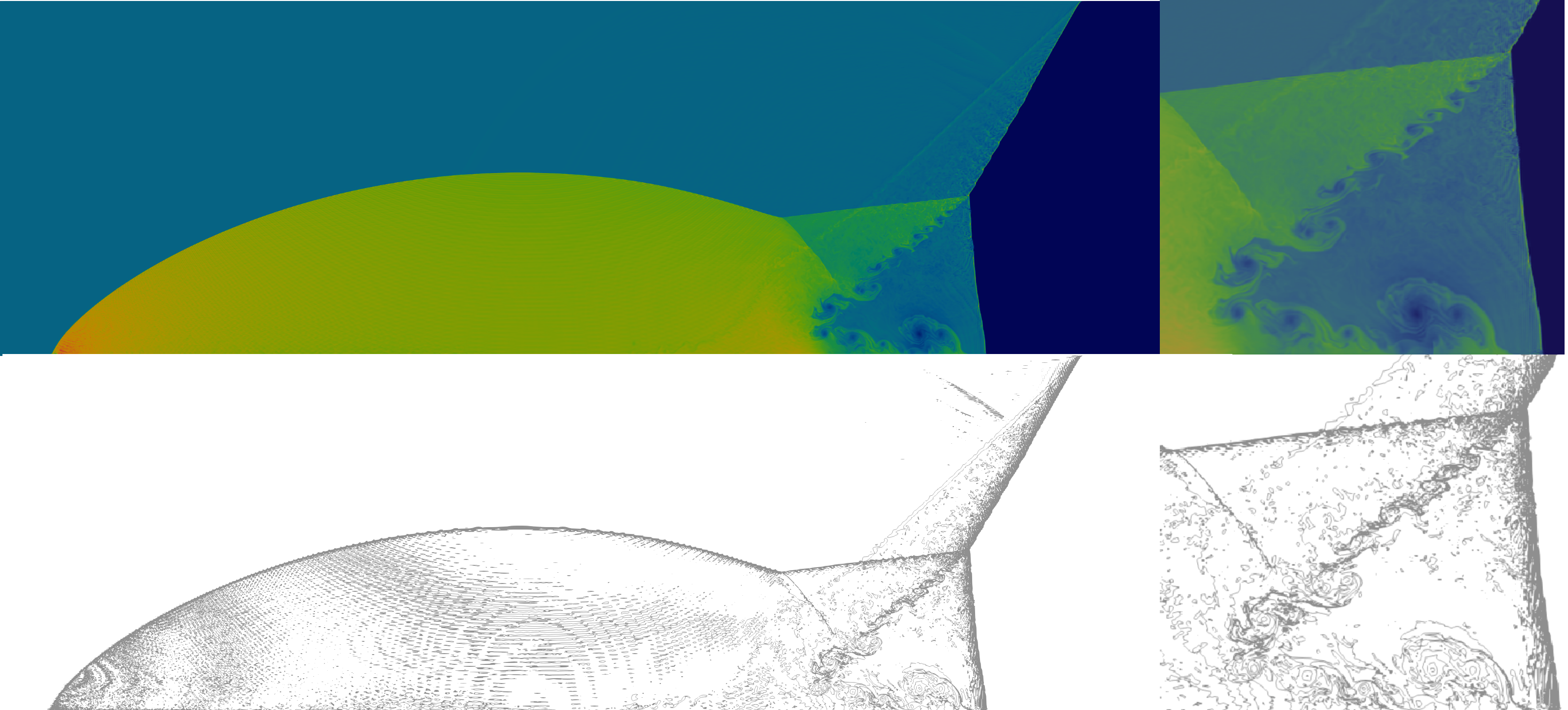}
  \caption{Double Mach Reflection, elementwise (Zhang-Shu type) limiting with $\zeta = 0.1$ (without shock capturing)}
  \label{fig:dmr0.1woshock}
\end{figure}

\begin{figure}[H]
  \centering
  \includegraphics[width=.8\linewidth]{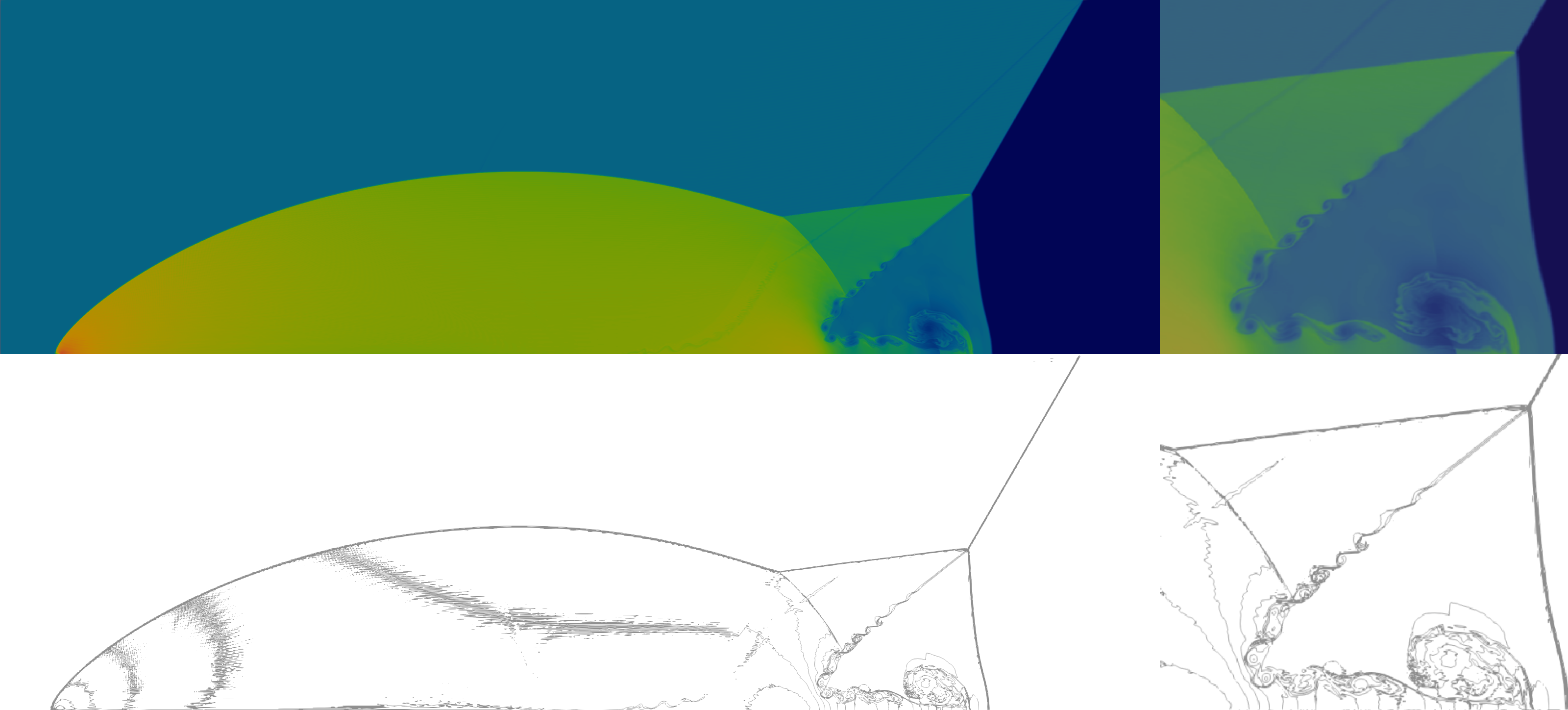}
  \caption{Double Mach Reflection, elementwise (Zhang-Shu type) limiting with $\zeta = 0.1$ (with shock capturing)}
  \label{fig:dmr0.1wshock}
\end{figure}

\begin{figure}[H]
  \centering
  \includegraphics[width=.8\linewidth]{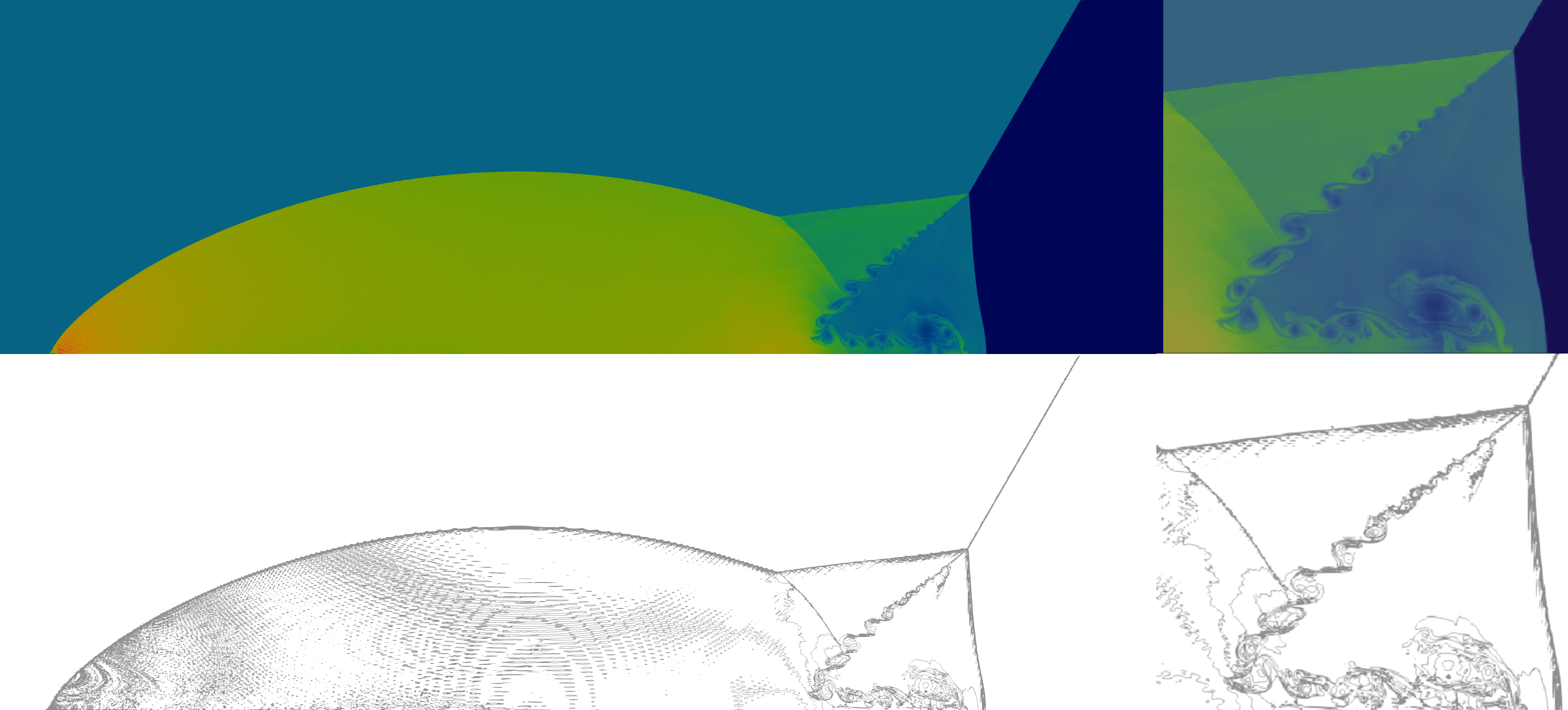}
  \caption{Double Mach Reflection, elementwise (Zhang-Shu type) limiting with $\zeta = 0.5$ (without shock capturing)}
  \label{fig:dmr0.5woshock}
\end{figure}

\rzero{
\subsection{Daru-Tenaud shocktube}

We conclude by examining the 2D shocktube problem proposed by Daru and Tenaud \cite{daru2000evaluation, daru2009numerical} for the compressible Navier-Stokes equation. This test case involves complex interactions between viscous shocks, contact waves, and viscous boundary layers. The domain is $[0,1]\times[0,0.5]$, and the initial condition is
\begin{align}
    \vec{u}\LRp{x,y} = \begin{cases}
        \vec{u}_L, \quad &x > 0.5\\
        \vec{u}_R, \quad &\text{otherwise}
    \end{cases}, \quad \vec{u}_L = \begin{bmatrix}
    \rho_L \\
    u_L \\
    v_L \\
    p_L \\
    \end{bmatrix} = \begin{bmatrix}
    120.0 \\
    0.0 \\
    0.0 \\
    \frac{120.0}{\gamma} \\
    \end{bmatrix},\quad \vec{u}_R = \begin{bmatrix}
    \rho_R \\
    u_R \\
    v_R \\
    p_R \\
    \end{bmatrix} = \begin{bmatrix}
    1.2 \\
    0.0 \\
    0.0 \\
    \frac{1.2}{\gamma} \\
    \end{bmatrix}, \label{eq:darushocktubeinit}
\end{align}
where $\gamma = 1.4, Pr = 0.73$. We enforce reflective wall boundary conditions on $[0,1]\times 0.5$ and adiabatic no-slip wall boundary condition on $[0,1] \times 0 \ \bigcup\  0 \times [0,0.5] \ \bigcup\  1 \times [0, 0.5]$ using previous works by Chan \cite{chan2022entropy} and Hindenlang et al. \cite{hindenlang2020stability}.

We discretize the domain using an uniform quadrilateral mesh with $2 K_{\rm 1D} \times K_{\rm 1D}$ elements. We run the simulation until $T = 1.0$ and $CFL = 0.5$. The solution is visualized using a numerical Schlieren plot, which visualizes gradients of the density field by plotting the quantity $\rho^{\rm schl}$\cite{guermond2018second}:
\begin{align}
    \rho^{\rm schl} = \exp\LRp{-10\frac{g - g_{\min}}{g_{\max}-g_{\min}}},\qquad g = \nor{\nabla\rho},\qquad g_{\min} = \min\limits_{x \in \Omega} g\LRp{x}, \qquad g_{\max} = \max\limits_{x\in\Omega} g\LRp{x}\label{eq:schlieren}
\end{align}

Figure \ref{fig:daru} shows the solution when $Re = 1000$ at $T = 1$ for various mesh resolutions and polynomial degrees. In this convergence study, we consider the generalized positivity bound with $\zeta = 0.1$ without introducing any extra dissipation or shock capturing. All simulations remained robust and did not crash, and we observe the convergence of the general flow structures. 

Three configurations are worth analyzing in more detail: $\LRp{N = 1, K_{\rm 1D} = 300}$, $\LRp{N = 2, K_{\rm 1D} = 200}$, and $\LRp{N = 1, K_{\rm 1D} = 150}$. All three configurations have the same number of degrees of freedom ($720000$). We notice that for higher degree polynomials $N = 2, 3$, there are noticeably more oscillations near the shock compared with an $N=1$ approximation. These oscillations appear to stem from the shock-shock interaction. 

Higher order schemes ($N=2, 3$) also produce qualitatively different flow structures near the bottom wall boundary compared to the second order approximation  with $N = 1$. The higher order solutions are qualitatively more similar to the fine-grid reference solution in \cite{guermond2022implementation}. However, we note that the limited scheme does not appear to converge uniformly towards this reference solution. For example, while the $N=3$ solutions for grid resolutions $K_{\rm 1D} = 100, 300$ both resemble the reference solution in \cite{guermond2022implementation}, grid resolution $K_{\rm 1D} = 200$ displays qualitatively different flow features near the $x = [0.45, 0.65]$ bottom wall boundary. 

\begin{landscape}
\begin{figure}
  \centering
  \raisebox{60pt}{\parbox[b]{.1\textwidth}{$N = 1$}}%
  \hspace{-0.8cm}\subfloat[][]{\includegraphics[width=.3\textwidth]{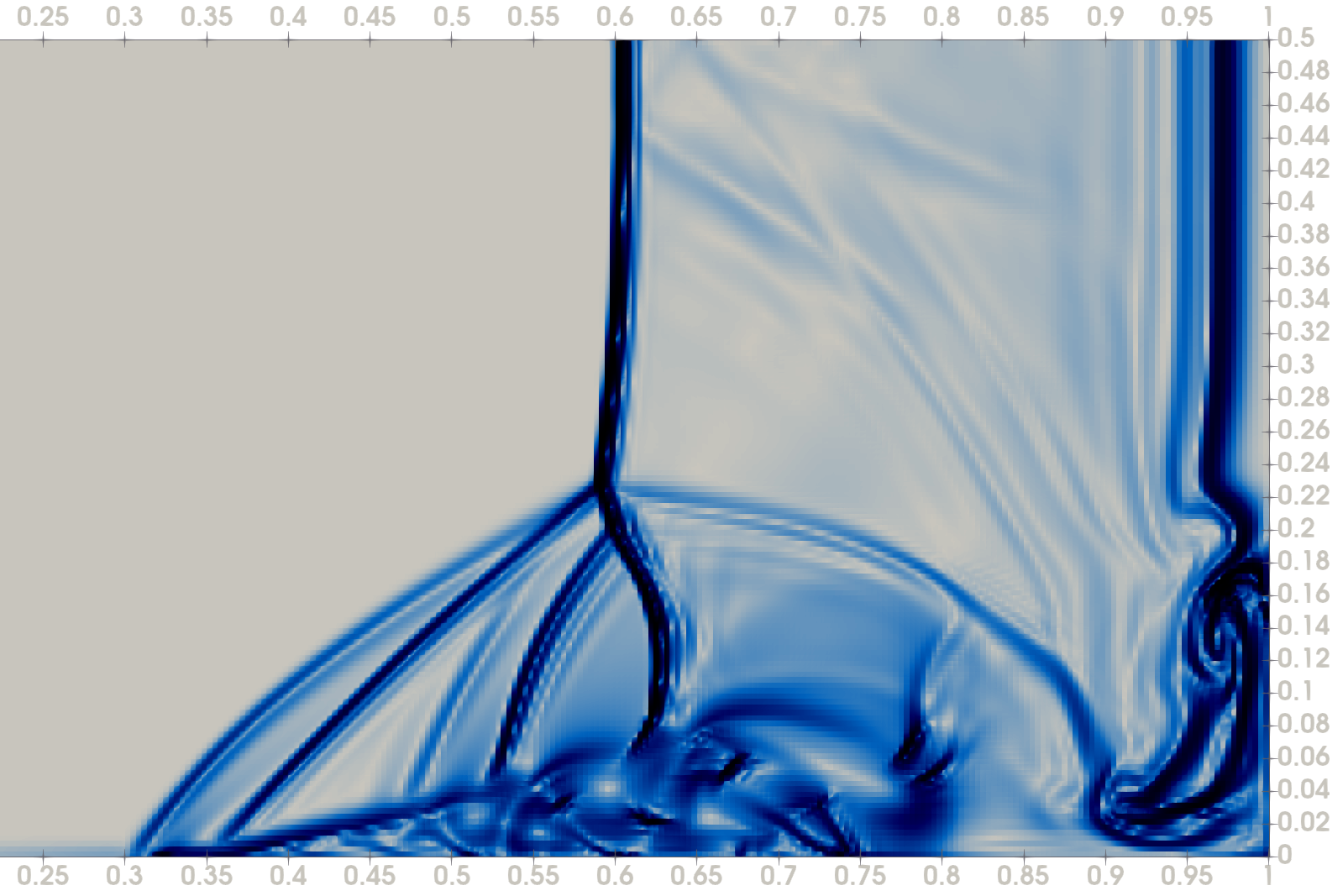}}\hfill
  \subfloat[][]{\includegraphics[width=.3\textwidth]{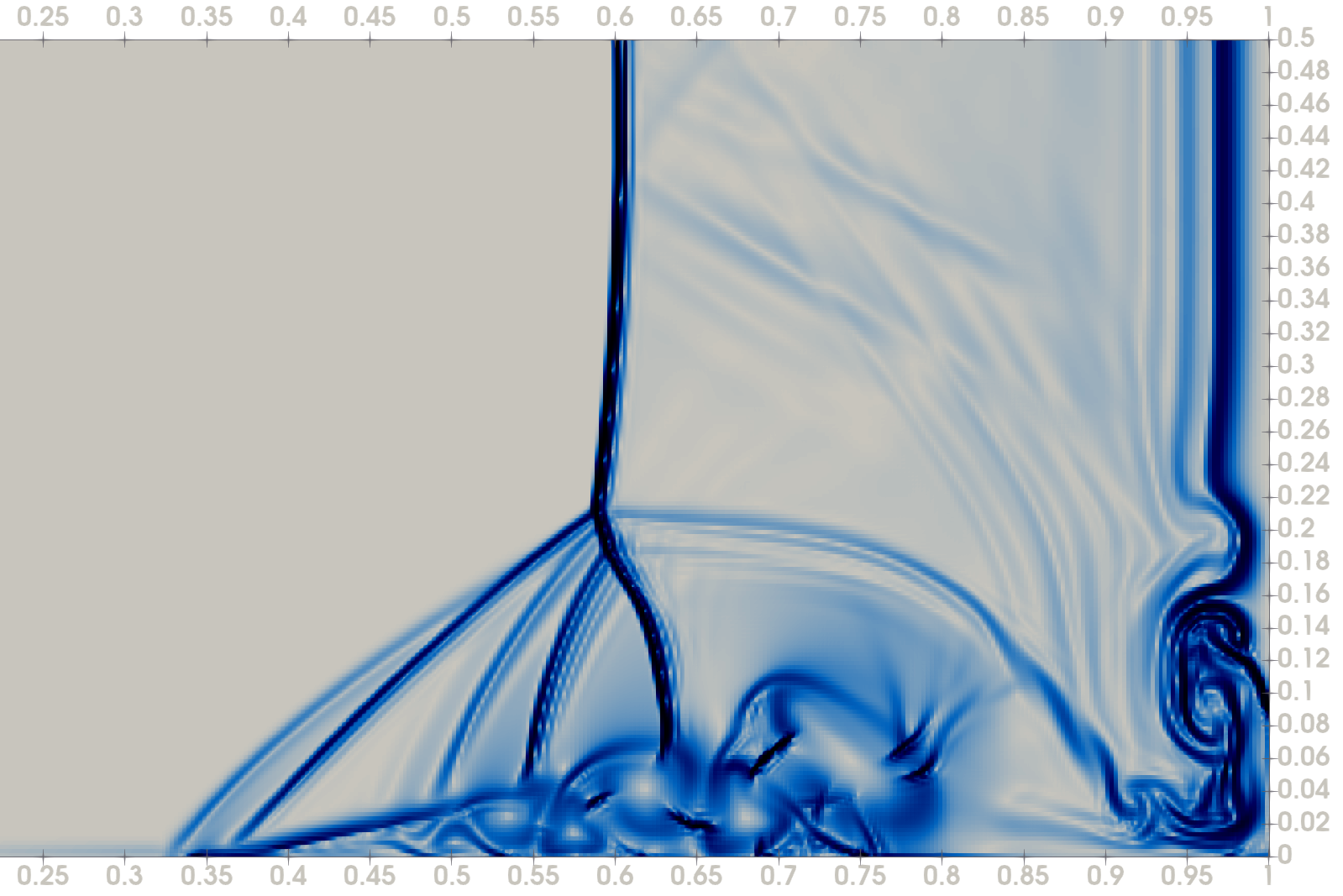}}\hfill
  \subfloat[][]{\includegraphics[width=.3\textwidth]{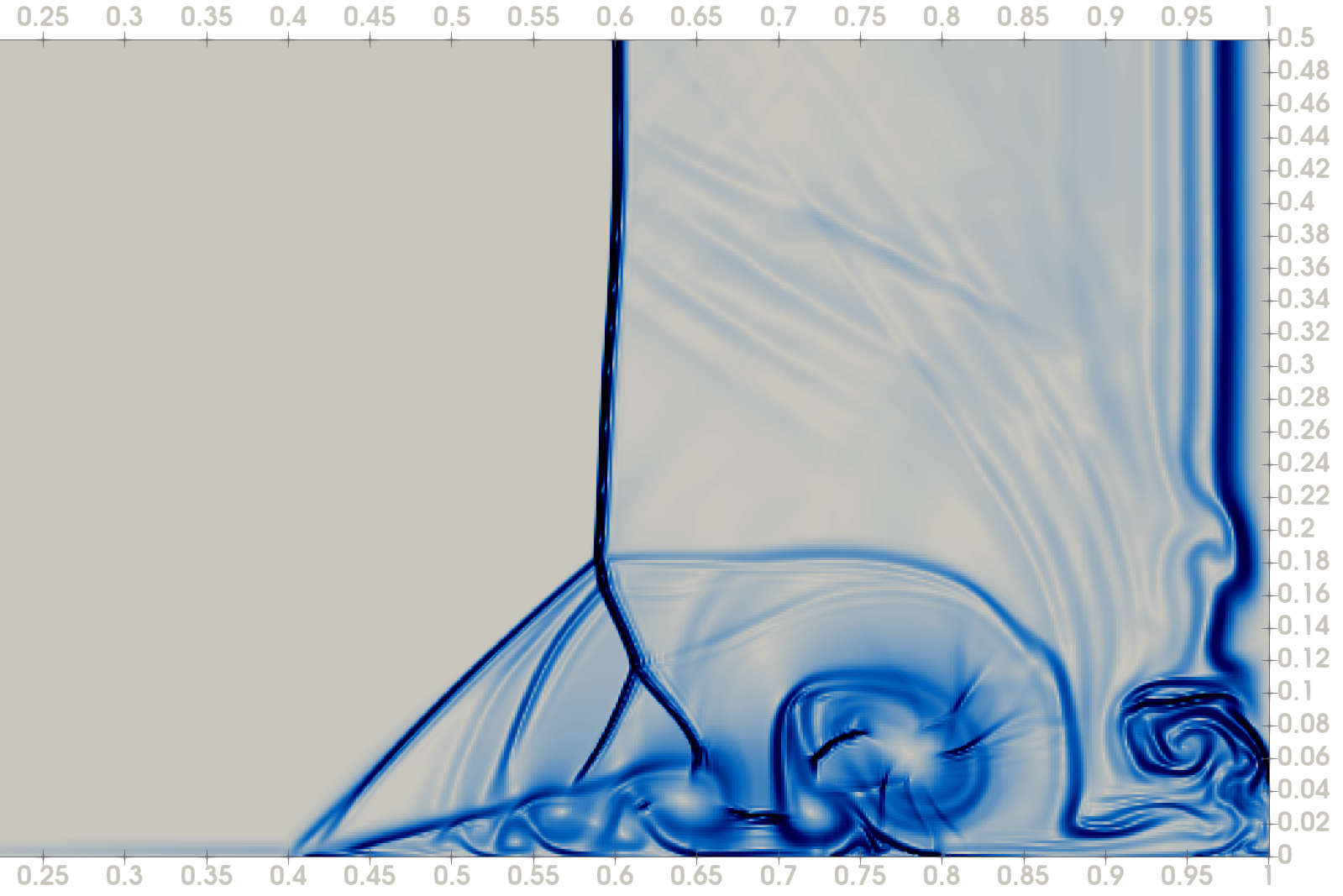}}\par
  \raisebox{60pt}{\parbox[b]{.1\textwidth}{$N = 2$}}%
  \hspace{-0.8cm}\subfloat[][]{\includegraphics[width=.3\textwidth]{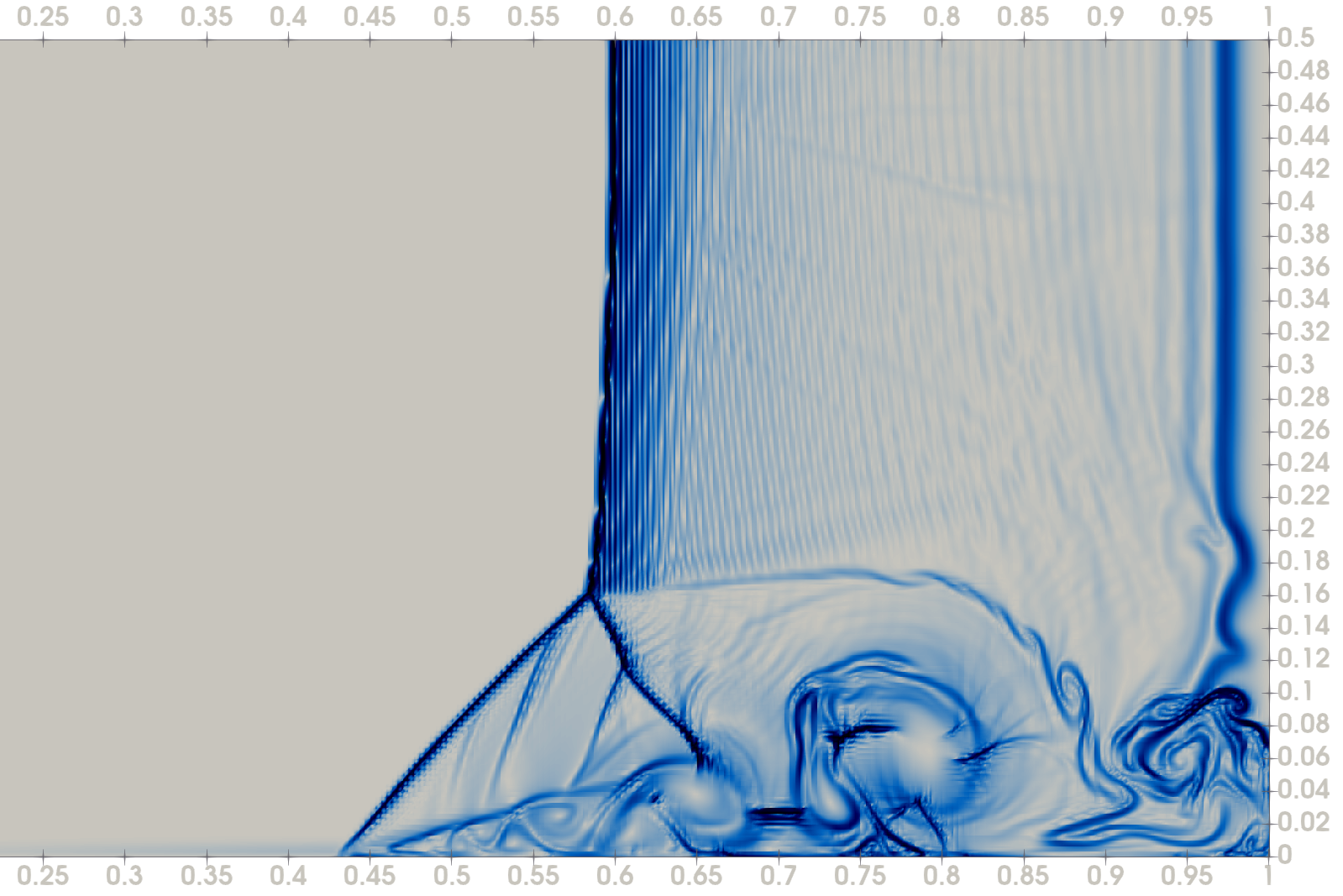}}\hfill
  \subfloat[][]{\includegraphics[width=.3\textwidth]{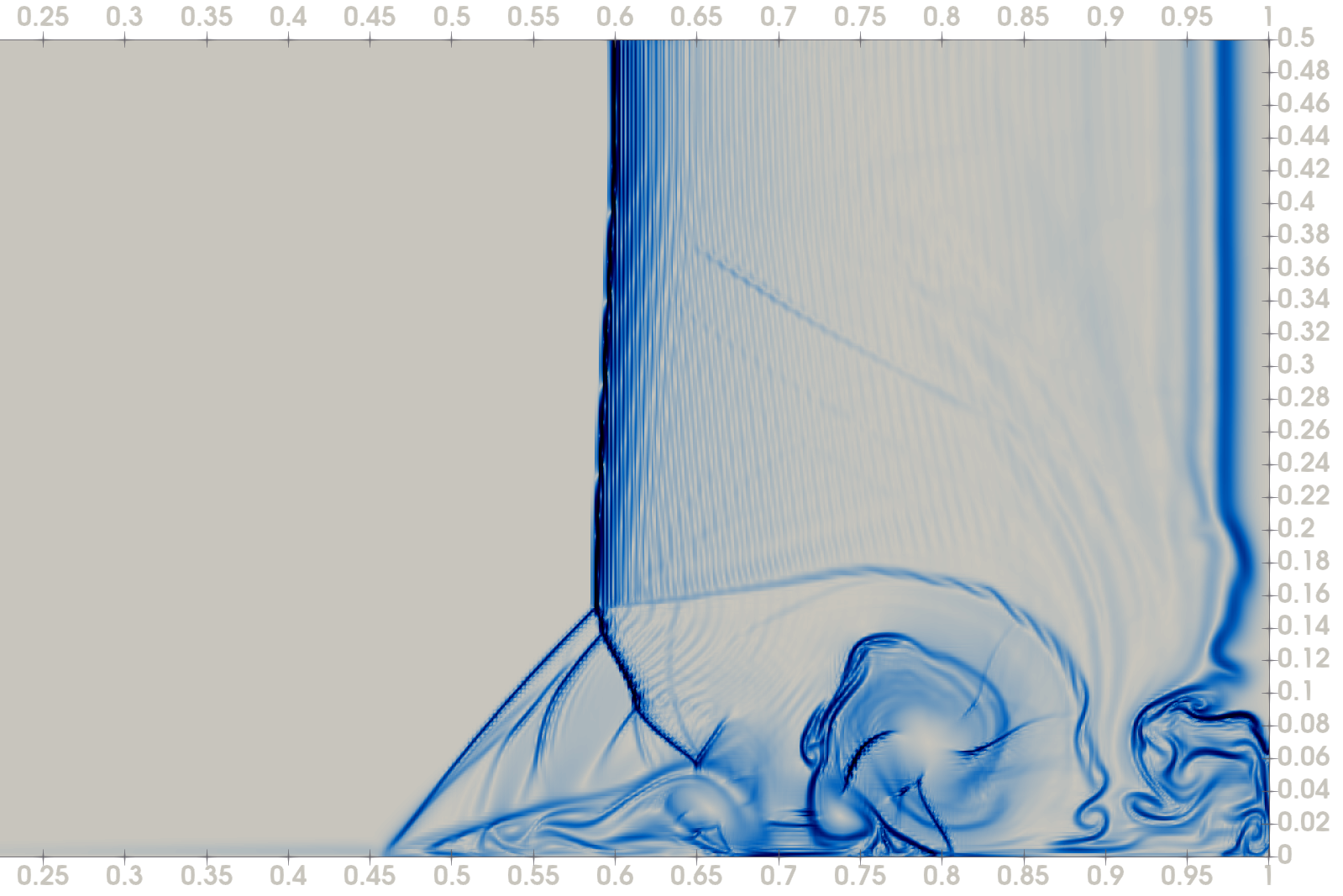}}\hfill
  \subfloat[][]{\includegraphics[width=.3\textwidth]{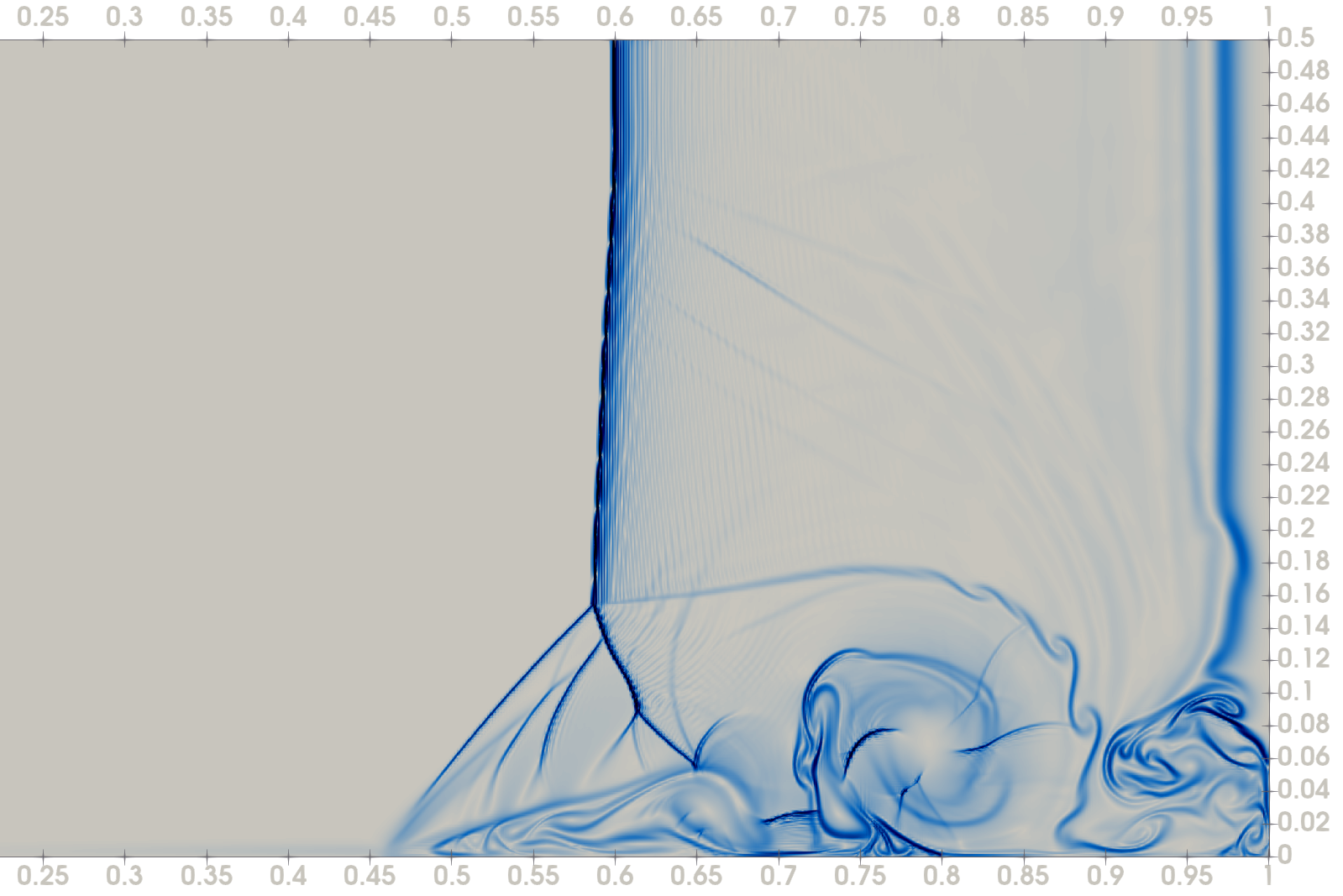}}\par
  \raisebox{60pt}{\parbox[b]{.1\textwidth}{$N = 3$}}%
  \hspace{-0.8cm}\subfloat[][]{\includegraphics[width=.3\textwidth]{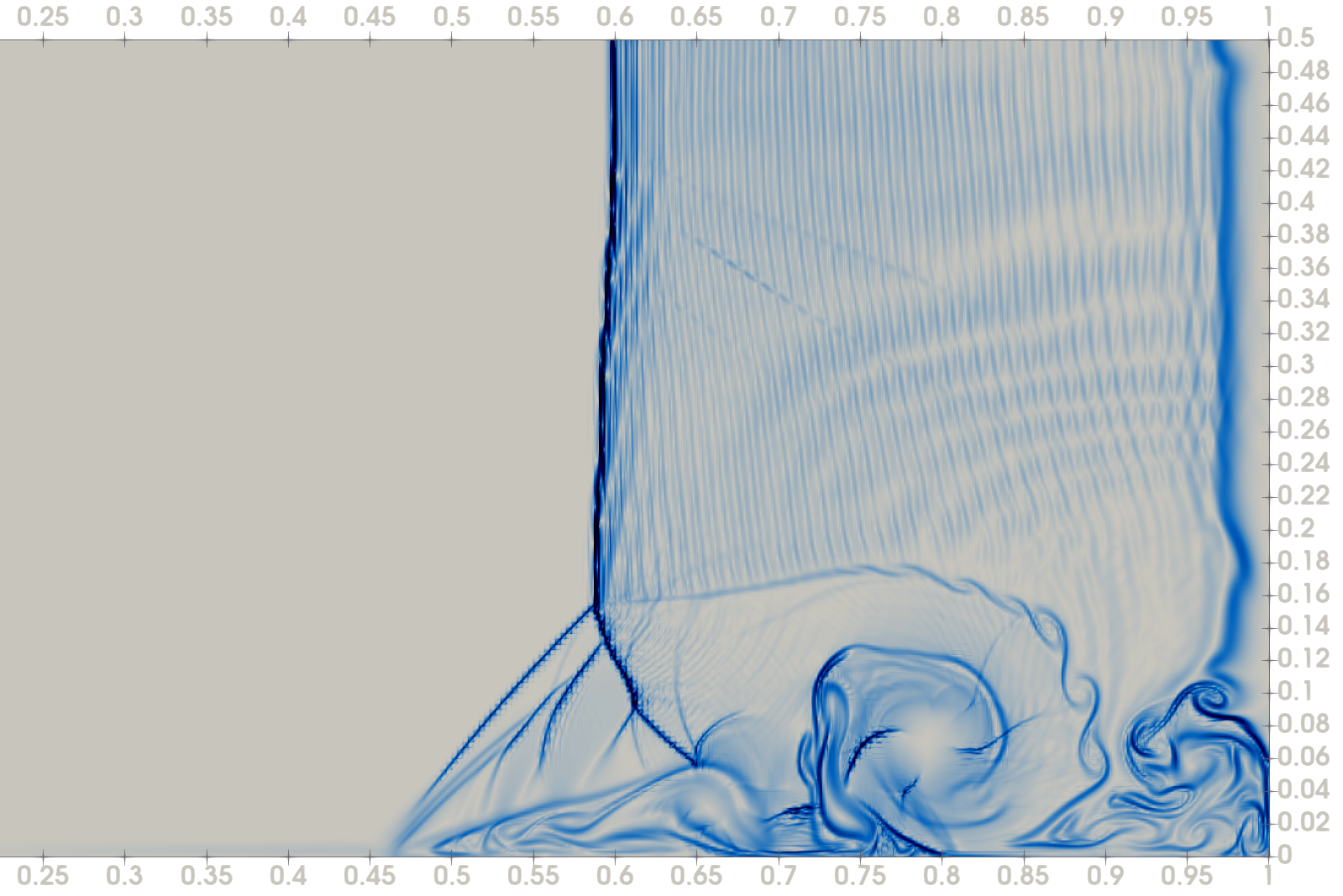}}\hfill
  \subfloat[][]{\includegraphics[width=.3\textwidth]{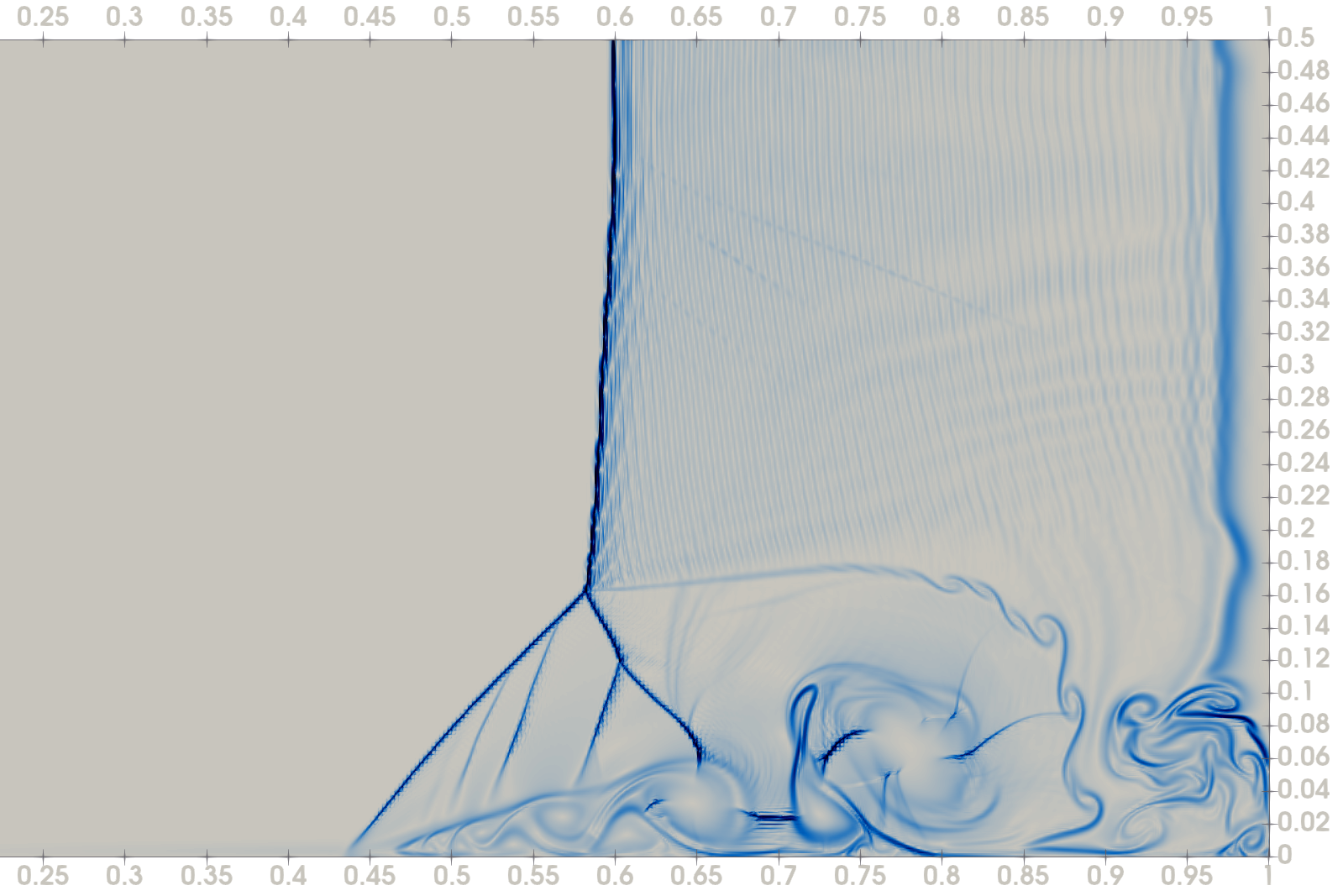}}\hfill
  \subfloat[][]{\includegraphics[width=.3\textwidth]{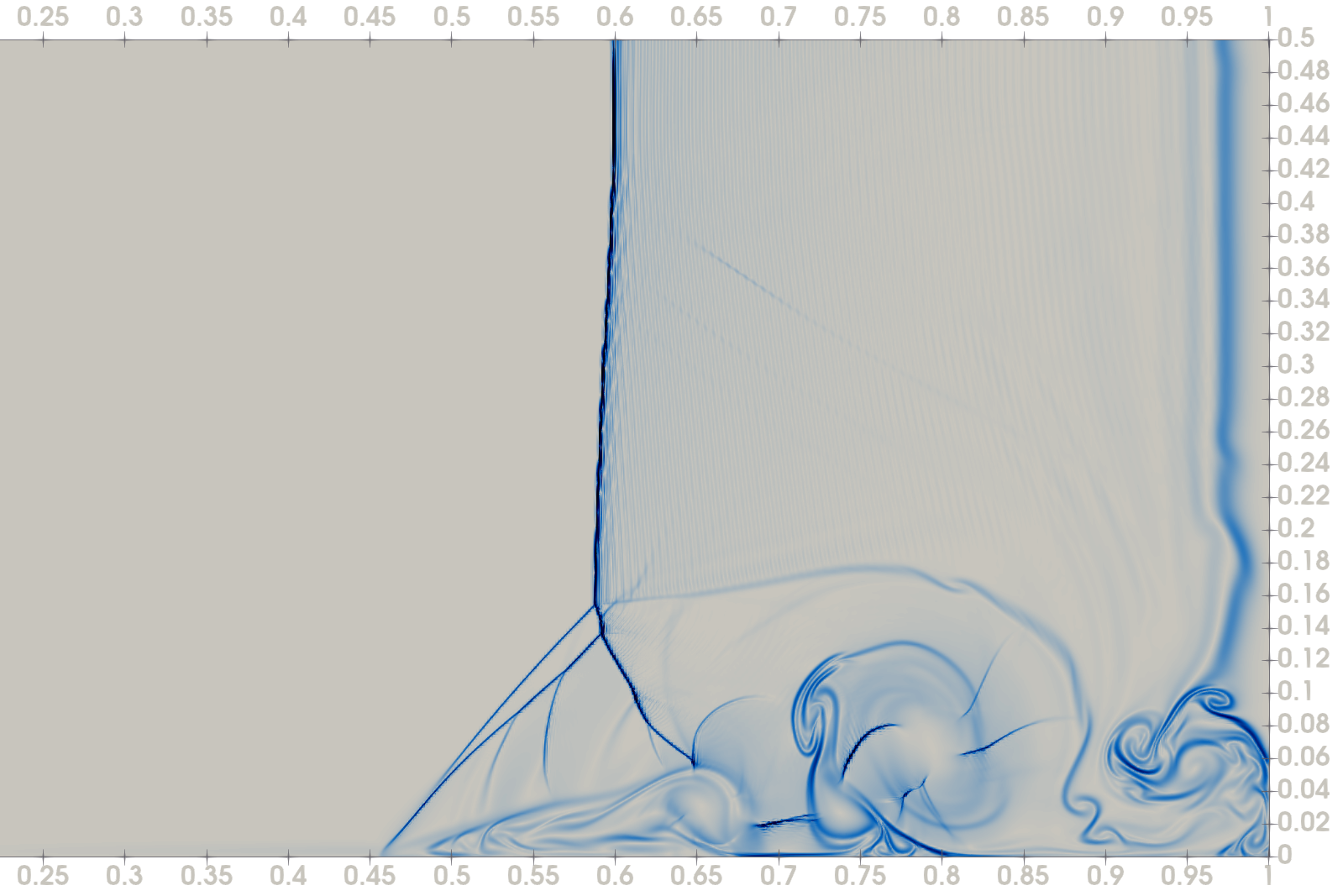}}\par
  \hspace{1.5cm}\raisebox{35pt}{\parbox[b]{.1\textwidth}{$K_{\rm 1D} = 150$}} \hspace{4.5cm} \raisebox{35pt}{\parbox[b]{.1\textwidth}{$K_{\rm 1D} = 200$}} \hspace{4.5cm} \raisebox{35pt}{\parbox[b]{.1\textwidth}{$K_{\rm 1D} = 300$}}
  \caption{Daru-Tenaud shocktube with $Re = 1000$, $\zeta = 0.1$ without shock capturing at $T = 1.0$}
  \label{fig:daru}
\end{figure}
\end{landscape}

Next, we study the sensitivity of Daru-Tenaud shocktube with respect to shock capturing and the strength of the relaxation factor $\zeta$. We run simulations under Reynolds numbers $\text{Re} = 1000$ and $\text{Re}= 10000$ using polynomial order $N = 2$ and $K_{\rm 1D} = 200$. For each configuration, we apply three limiting configurations: $\zeta = 0.1$ without shock capturing, $\zeta = 0.1$ with shock capturing, and $\zeta = 0.5$ without shock capturing. 

For both Reynolds numbers, simulations remain robust, and Figures \ref{fig:daru1000} and \ref{fig:daru10000} suggest that the choice of the relaxation factor $\zeta$ has little impact on the qualitative features. The solution features become more sensitive to the relaxation factor  when the Reynolds number is large, which may be due to the increased sensitivity of the solution as viscosity decreases. In contrast, Figure~\ref{fig:daru1000} demonstrates that even at lower Reynolds numbers, the solution is very sensitive to additional shock capturing. In particular, the qualitative structure of the shocks in the region $[0.45, 0.64] \times [0.04, 0.18]$ and the boundary phenomena in the region $[0.45, 0.64] \times [0.0, 0.04]$ is significantly changed by the addition of shock capturing. Moreover, the results with additional shock capturing appear to be incorrect when compared with fine-grid reference results from \cite{daru2009numerical, guermond2022implementation}.

\begin{figure}[!h]
\centering
\begin{subfigure}{.8\textwidth}
  \centering
  \includegraphics[width=.9\linewidth]{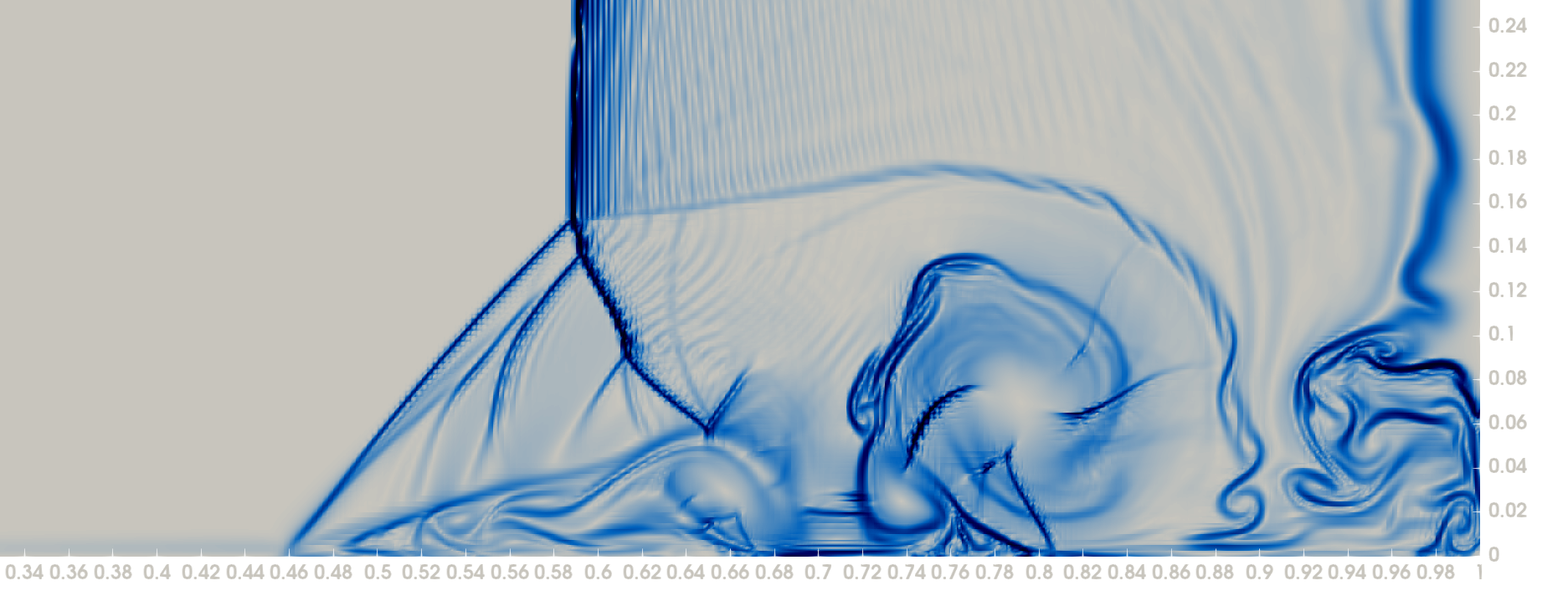}
  \caption{Elementwise (Zhang-Shu type) limiting with $\zeta = 0.1$, without shock capturing}
\end{subfigure}%
\vspace{0.5cm}
\begin{subfigure}{.8\textwidth}
  \centering
  \includegraphics[width=.9\linewidth]{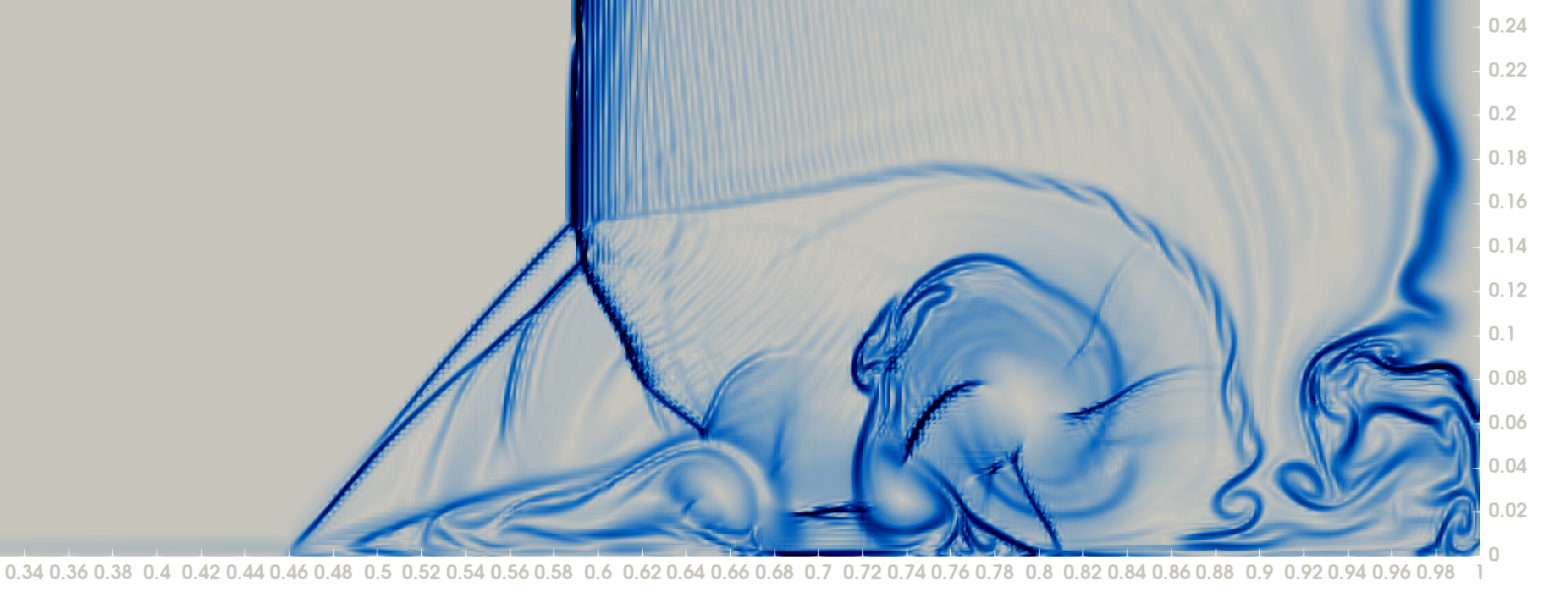}
  \caption{Elementwise (Zhang-Shu type) limiting with $\zeta = 0.5$, without shock capturing}
\end{subfigure}%
\vspace{0.5cm}
\begin{subfigure}{.8\textwidth}
  \centering
  \includegraphics[width=.9\linewidth]{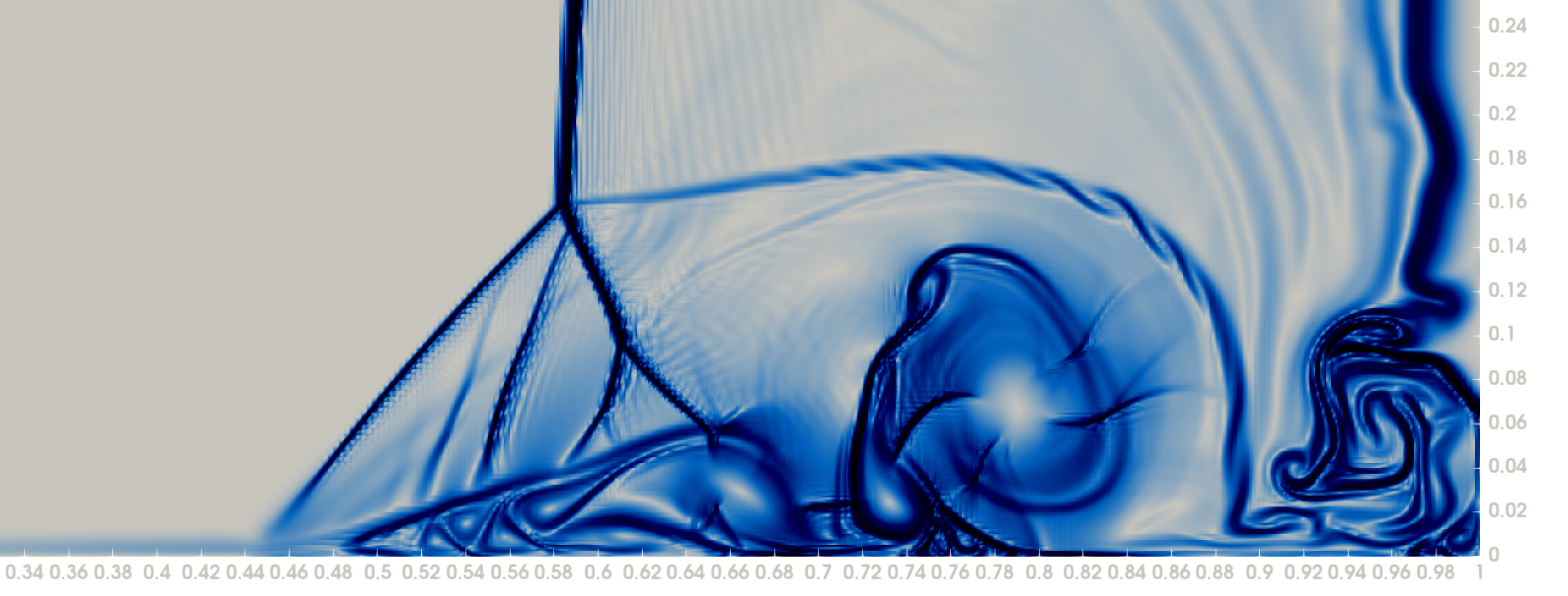}
  \caption{Elementwise (Zhang-Shu type) limiting with $\zeta = 0.1$, with shock capturing}
\end{subfigure}%
\caption{Daru-Tenaud shocktube, $Re = 1000, N = 2, K_{\rm 1D} = 200$}
\label{fig:daru1000}
\end{figure}


\begin{figure}[!htb]
\centering
\begin{subfigure}{.8\textwidth}
  \centering
  \includegraphics[width=.9\linewidth]{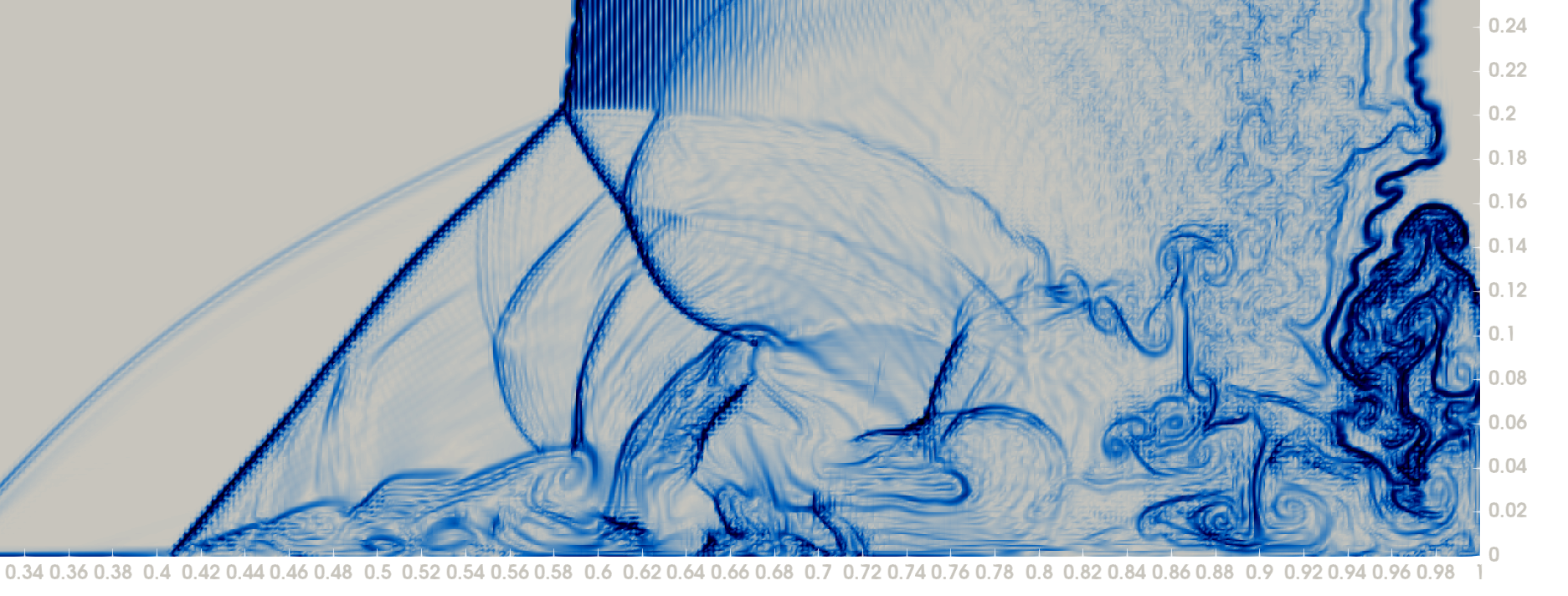}
  \caption{Elementwise (Zhang-Shu type) limiting with $\zeta = 0.1$, without shock capturing}
\end{subfigure}%
\vspace{0.5cm}
\begin{subfigure}{.8\textwidth}
  \centering
  \includegraphics[width=.9\linewidth]{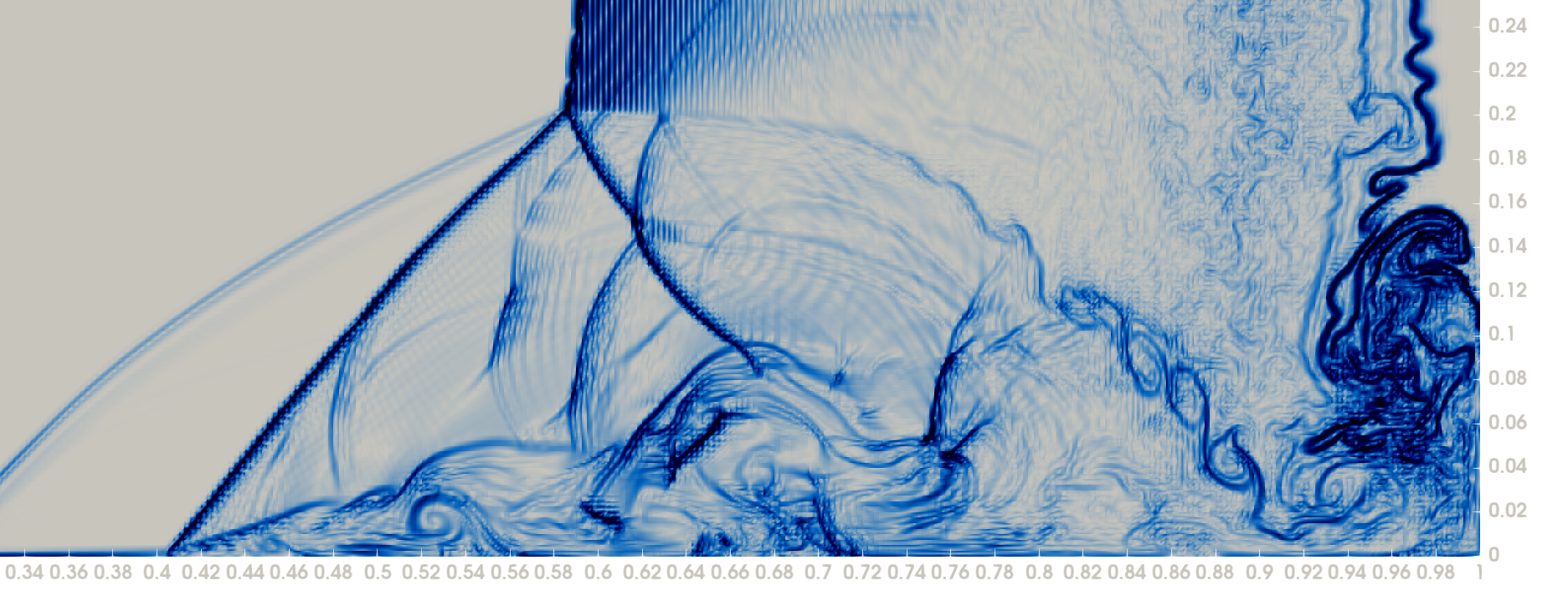}
  \caption{Elementwise (Zhang-Shu type) limiting with $\zeta = 0.5$, without shock capturing}
\end{subfigure}%
\vspace{0.5cm}
\begin{subfigure}{.8\textwidth}
  \centering
  \includegraphics[width=.9\linewidth]{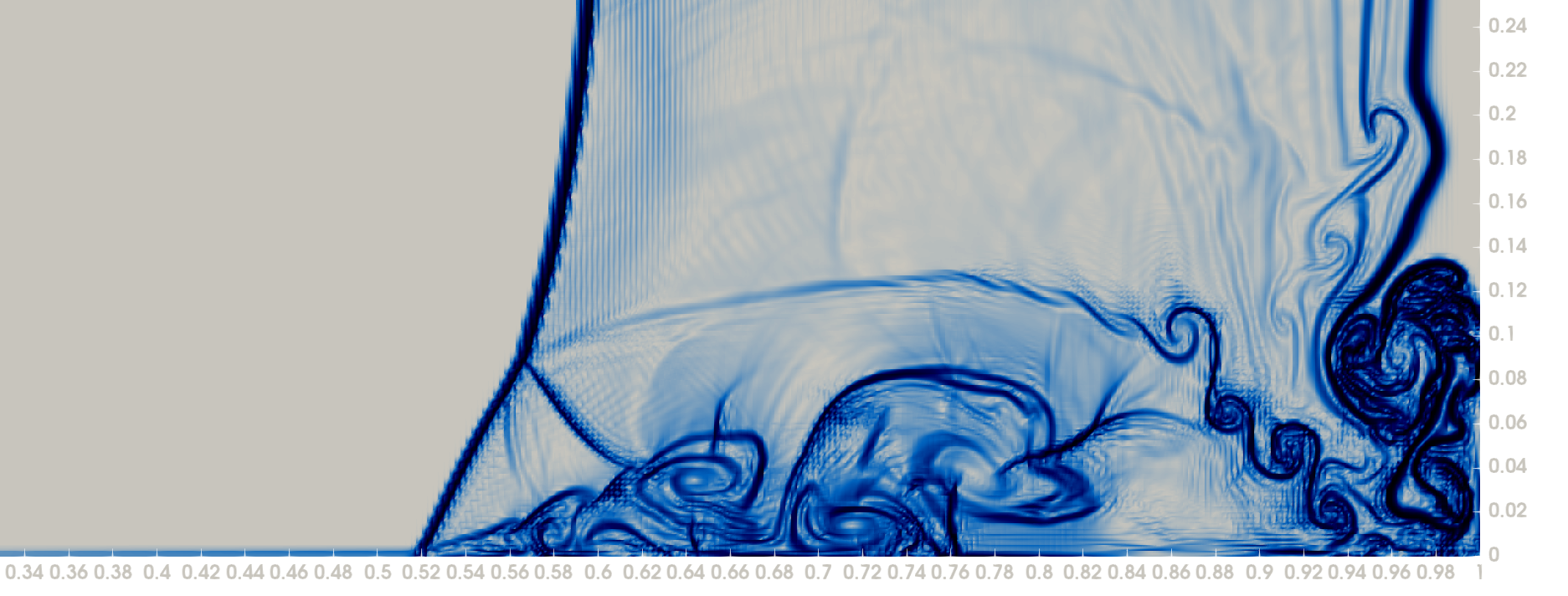}
  \caption{Elementwise (Zhang-Shu type) limiting with $\zeta = 0.1$, with shock capturing}
\end{subfigure}%
\caption{Daru-Tenaud shocktube, $Re = 10000, N = 2, K_{\rm 1D} = 200$}
\label{fig:daru10000}
\end{figure}

}

\section{Conclusion} 
\label{sec:con}

In this paper, we present a positivity-preserving limiting strategy for entropy stable schemes applied to the compressible Navier-Stokes equations. We construct a low order positivity preserving discretization, which is then blended  \rzero{elementwise} with a higher order entropy stable discretization. The proposed limiting scheme preserves positivity of the density and pressure under a CFL condition while retaining conservation and a \rzero{semi-}discrete entropy balance. 

The scheme is purely explicit, and as such is subject to a parabolic time-step restriction. Future work will investigate techniques to ameliorate the parabolic time-step restriction, as well as the extension of this limiting procedure to generalized SBP schemes \cite{chan2019efficient} and modal entropy stable discretizations.

\section*{Acknowledgement}
Yimin Lin and Jesse Chan gratefully acknowledge support from National Science Foundation under awards DMS-1719818 and DMS-CAREER-1943186. This work used the Extreme Science and Engineering Discovery Environment (XSEDE) Expanse at the San Diego Supercomputer Center through allocation TG-MTH200014 \cite{6866038}.
The work of IT was partially supported by the U.S. Department of Energy, Office of Science, Office of Advanced Scientific Computing Research, Applied Mathematics Program and by the U.S. Department of Energy, Office of Science, Office of Advanced Scientific Computing Research and Office of Fusion Energy Sciences, Scientific Discovery through Advanced Computing (SciDAC) program. This manuscript has been authored by National Technology \& Engineering Solutions of Sandia, LLC., under contract DE-NA0003525 with the U.S. Department of Energy/National Nuclear Security Administration. The United States Government retains and the publisher, by accepting the article for publication, acknowledges that the United States Government retains a non-exclusive, paid-up, irrevocable, world-wide license to publish or reproduce the published form of this manuscript, or allow others to do so, for United States Government purposes. 

\appendix

\section{Extension to curved meshes}
\label{app:curved}

While the limiting strategy in this paper has been derived for Cartesian and affine meshes, the algebraic formulation makes it possible to extend the approach to curved meshes. The main steps follow \cite{pazner2021sparse}, which we briefly review here. First, recall that the high order formulation \eqref{eq:ESDGskeweuler} and low order formulations \eqref{eq:lowpp} are posed in terms of  discretization matrices. To extend these formulations to a curved element, it suffices to define appropriate curved discretization matrices. Positivity of the low order scheme then follows under modified conditions on the graph viscosity and time-step size involving these new matrices. 

We construct discretization matrices over each curved ``physical'' element from discretization matrices defined on a reference element. First, we assume that each physical element is the image of the reference element under a differentiable mapping. Let $\bm{x}$ and $\widehat{\bm{x}}$ denote the physical and reference coordinates, respectively. Then, geometric change of variables terms $\pd{{x}_i}{\widehat{{x}}_j}$ can be computed on each element. Outward normal vectors can also be computed from these geometric terms \cite{chan2019discretely}. We note that it is important to compute these geometric terms such that a discrete version of the geometric conservation law 
\[
\pd{}{\widehat{x}_j} \pd{{x}_i}{\widehat{{x}}_j} = 0
\]
is satisfied. For the high order method, this can be done using a curl representation of the geometric terms \cite{chan2019discretely}, while for the low order method this can be done by solving algebraically for a separate set of low order geometric terms \cite{pazner2021sparse, rueda2021subcell}. 

These geometric terms can then be used to construct physical discretization matrices. Let $\fnt{g}_{ij}$ denote the vector containing values of the geometric terms $\pd{{x}_i}{\widehat{{x}}_j}$ at nodal points on a single element, and let $\fnt{J}$ denote the vector of nodal values of the determinant of the Jacobian of the reference-to-physical mapping. Let $\widehat{\fnt{M}}$, $\widehat{\fnt{Q}}_i$, and $\widehat{\fnt{B}}_i$ denote the reference mass matrix, $i$th reference high order differentiation matrix, and $i$th reference boundary matrix. A physical mass matrix can be constructed on each element via
\[
\fnt{M} = \widehat{\fnt{M}} \diag{\fnt{J}}.
\]
Physical high order differentiation matrices can be constructed via
\[
\fnt{Q}_i = \sum_{j=1}^d \frac{1}{2}\LRp{\diag{\fnt{g}_{ij}}\widehat{\fnt{Q}}_j + \widehat{\fnt{Q}}_j \diag{\fnt{g}_{ij}}}.
\]
Physical boundary matrices can be constructed similarly
\[
\fnt{B}_i = \sum_{j=1}^d \widehat{\fnt{B}}_j \diag{\fnt{g}_{ij}}.
\]
It can be shown that these matrices satisfy both summation by parts and conservation conditions if the geometric terms satisfy discrete versions of the GCL \cite{crean2018entropy, chan2019discretely, pazner2021sparse}. The same procedure (with appropriate low order geometric terms) can be used to construct sparse low order physical differentiation matrices. 

\section{Convex limiting strategy}
\label{app:convexlimit}

\rzero{In the appendix, we will explore another limiting technique called convex limiting and provide an alternative approach to limit entropy stable discretizations.} To simplify the notation, we adapt the notation used in \cite{guermond2019invariant}, so we can write the low order and high-order approximations over node $i$ in algebraic forms:
\begin{align}
    \fnt{m}_i \frac{\fnt{u}^{{\low},n+1}_i-\fnt{u}_i}{\tau} + \sum\limits_{j\in \mathcal{I}(i)} \fnt{F}_{ij}^{\low} + \sum\limits_{j \in \mathcal{B}(i)} \fnt{F}_{ij}^{{\rm B},{\low}} &= 0 \label{eq:lppaf}\\
    \fnt{m}_i \frac{\fnt{u}^{{\high},n+1}_i-\fnt{u}_i}{\tau} + \sum\limits_{j\in \mathcal{I}(i)} \fnt{F}_{ij}^{\high} + \sum\limits_{j \in \mathcal{B}(i)} \fnt{F}_{ij}^{{\rm B},{\high}} &= 0 \label{eq:esdgaf}
\end{align}

In particular, for the compressible Navier-Stokes equation, the low and high order algebraic fluxes are
\begin{alignat}{2}
    \fnt{F}^{\low}_{ij} &= \sum\limits_{k=1}^2 \frac{1}{2}\LRp{\fnt{Q}_k^{\low} - \LRp{\fnt{Q}_k^{\low}}^T}_{ij} \LRs{\vec{f}_k\LRp{\fnt{u}_i}+\vec{f}_k\LRp{\fnt{u}_j} - \LRp{\fnt{\sigma}_k}_i - \LRp{\fnt{\sigma}_k}_j} - \lambda_{ij}\LRp{\fnt{u}_j - \fnt{u}_i}, \qquad \label{eq:FijL}\\
    \fnt{F}^{{\rm B}, \low}_{ij} &= \sum\limits_{k=1}^2 \frac{1}{2}\LRp{\fnt{E}^T\fnt{B}_k\fnt{E}}_{ii} \LRs{\vec{f}_k\LRp{\fnt{u}_i} + \vec{f}_k\LRp{\fnt{u}_i^+} - \LRp{\fnt{\sigma}_k}_i - \LRp{\fnt{\sigma}_k}_i^+} - \LRp{\fnt{\lambda}_k}_i \LRp{\fnt{u}_i^+ - \fnt{u}_i} \label{eq:FijBL}\\
    \fnt{F}_{ij}^{\high} &= \sum\limits_{k=1}^2 \rtwo{\LRp{\fnt{Q}_k - \LRp{\fnt{Q}_k}^T}_{ij}} \LRs{\vec{f}_{k,S}\LRp{\fnt{u}_i,\fnt{u}_j}- \frac{\LRp{\fnt{\sigma}_k}_i + \LRp{\fnt{\sigma}_k}_j}{2}}, \qquad \label{eq:FijH}\\ \fnt{F}_{ij}^{{\rm B},\high} &=  \sum\limits_{k=1}^2\LRp{\fnt{E}^T\fnt{B}_k\fnt{E}}_{ii} \LRs{\vec{f}_{k,S}\LRp{\fnt{u}_i, \fnt{u}_i^+} - \frac{\LRp{\fnt{\sigma}_k}_i + \LRp{\fnt{\sigma}_k}_i^+}{2}} - \frac{\vec{w}_i^{ f}\LRp{\fnt{\lambda}_{\max,k}}_i}{2} \LRp{\fnt{u}_i^+ - \fnt{u}_i}, \label{eq:FijBH}
\end{alignat}
and we can derive the low and high order algebraic fluxes for the compressible Euler equation by eliminating the viscous terms. We can establish the relation between the low and high order updates through the low and high order algebraic fluxes:
\begin{equation}
    \fnt{m}_i \fnt{u}_i^{{\high},n+1} = \fnt{m}_i \fnt{u}_i^{{\low},n+1} + \tau \LRp{\sum\limits_{j\in\mathcal{I}(i)} \LRp{\fnt{F}_{ij}^{\low}-\fnt{F}_{ij}^{\high}} + \sum\limits_{j\in\mathcal{B}(i)} \LRp{\fnt{F}_{ij}^{\rm B,L} - \fnt{F}_{ij}^{\rm B,H}}} \label{eq:highinlow}
\end{equation}

\subsection{Convex limiting}

Inspired by flux corrected transport \cite{boris1973flux}, the limited solution can be written as 
\begin{equation}
    \fnt{m}_i \fnt{u}_i^{n+1} = \fnt{m}_i \fnt{u}_i^{{\low},n+1} + \tau \LRp{\sum\limits_{j\in\mathcal{I}(i)} l_{ij}\LRp{\fnt{F}_{ij}^{\low}-\fnt{F}_{ij}^{\high}} + \sum\limits_{j\in\mathcal{B}(i)} l_{ij}\LRp{\fnt{F}_{ij}^{\rm B,L} - \fnt{F}_{ij}^{\rm B,H}}}, \label{eq:limitedrewrite}
\end{equation}
where $l_{ij}$ are limiting parameters in the range $[0,1]$. Following the idea in \cite{guermond2018second}, we can rewrite the limited solution as a convex combination of substates of form $\fnt{u} + l\fnt{P}$:
\begin{align}
    \fnt{u}_i^{n+1} &= \sum\limits_{j\in\mathcal{I}(i)} \lambda_j \LRp{\fnt{u}_i^{\low,n+1}+l_{ij}\fnt{P}_{ij} } + \sum\limits_{j\in\mathcal{B}(i)} \lambda_j \LRp{\fnt{u}_i^{\low,n+1}+l_{ij}\fnt{P}_{ij}^{\rm B} },\label{eq:limitedinconvex}\\
    \fnt{P}_{ij} &= \frac{\tau }{\fnt{m}_i \lambda_j} \LRp{\fnt{F}_{ij}^{\low}-\fnt{F}_{ij}^{\high}},\qquad
    \fnt{P}_{ij}^{\rm B} = \frac{\tau }{\fnt{m}_i \lambda_j} \LRp{\fnt{F}_{ij}^{\rm B,L}-\fnt{F}_{ij}^{\rm B,H}}
    ,\qquad
    \sum\limits_{j \in\mathcal{I}\LRp{i}\bigcup \mathcal{B}\LRp{i}} \lambda_j = 1, \quad \lambda_j > 0, \nonumber
\end{align}
where $\LRc{\lambda_j}_{\mathcal{I}\LRp{i}\bigcup \mathcal{B}\LRp{i}}$ is a set of strictly positive convex coefficients. For example, in the numerical experiments in this work, we will use one of the most obvious choices of convex coefficients: $\lambda_j = \frac{1}{\text{card}\LRp{\mathcal{I}(i)\bigcup \mathcal{B}(i)}}$. In \cite{guermond2019invariant}, other choices of convex coefficients are explored and they do not show significant advantages over the uniform choice. 

Since the admissible set $\mathcal{A}$ is convex, the limited solution lies in the admissible set if every sub-state of the form $\fnt{u}^{\low} + l\fnt{P}$ is admissible. When the limiting parameter $l = 0$, we recover the admissible low order approximation, and when the limiting parameter $l = 1$, we recover the high order entropy-stable approximation, which may not be admissible. In order to stay as close to the high-order scheme as possible, we determine the value of limiting parameters by finding the largest possible $l$ that satisfies the positivity constraints:
\begin{align}
        l_{ij} = 
    \max\LRc{ l \in \LRs{0,1} \ |\ \fnt{u}_i^{{\low},n+1}+l \fnt{P}_{ij} \in \mathcal{A}, \quad \fnt{u}_j^{{\low},n+1}+l \fnt{P}_{ji} \in \mathcal{A}}. \label{eq:lij}
\end{align}

The constraint $\fnt{u}^{\low}_i + l\fnt{P} \in \mathcal{A}$ is a quadratic constraint with respect to $l$. \rzero{As discussed in previous sections, \eqref{eq:lparam} gives the exact solution of the limiting parameters.} The limiting parameters $l_{ij}$ are typically enforced to be symmetric \cite{guermond2019invariant},  \rzero{then the limited solution \eqref{eq:limitedrewrite} is conservative and admissible.}


This limiting strategy is adopted in \cite{guermond2018second, guermond2019invariant, pazner2021sparse}, and can be advantageous because it offers sub-cell ``blending'' of high and low order schemes in addition to any sub-cell resolution provided by the low order positivity-preserving scheme. However, in this work we only enforce minimal positivity conditions (i.e., global positivity of density and internal energy), while previous literature used convex-limiting to enforce stronger constraints (e.g. local bounds-preservation on density and local minimum principle on the specific entropy). 

We note that \rzero{the elementwise (Zhang-Shu type)} limiting strategies considered in this work utilize a single blending parameter on each element, which naturally preserves a semi-discrete cell entropy inequality. In contrast, the convex limiting strategy can not be shown to preserve a semi-discrete cell entropy inequality \rtwo{if the positivity constraint is the only constraint imposed}. The lack of a numerical entropy inequality has been shown to introduce spurious phenomena \cite{harten1976finite}. \rtwo{This can be remedied by also enforcing entropic constraints through the limiting process \cite{rueda2021subcell, kuzmin2022limiter}.}


\subsection{Localization of the limiting parameters}
\label{sec:localization}

Symmetrizing the limiting parameters in \eqref{eq:lij} ensures the conservation property of the limited solution. However a naive implementation of symmetrization requires the exchange of information on neighboring elements. \rzero{We use the local} Lax-Friedrichs type interface fluxes in our high order ESDG formulation, which avoids this exchange of information and localizes the convex limiting procedure. Then, for the compressible Euler and Navier-Stokes equations, the interface numerical fluxes of the high order entropy stable methods are identical to the interface numerical fluxes of the low-order scheme:
\begin{align}
    \fnt{F}^{{\rm B}, {\rm H}}_{ij} = \fnt{F}^{{\rm B}, \low}_{ij} &= \sum\limits_{k=1}^2 \frac{1}{2}\LRp{\fnt{E}^T\fnt{B}_k\fnt{E}}_{ii} \LRs{\vec{f}_k\LRp{\fnt{u}_i} + \vec{f}_k\LRp{\fnt{u}_i^+} - \LRp{\fnt{\sigma}_k}_i - \LRp{\fnt{\sigma}_k}_i^+} - \LRp{\fnt{\lambda}_k}_i \LRp{\fnt{u}_i^+ - \fnt{u}_i}. \label{eq:interfacemod}
\end{align}
Since the low and high order algebraic fluxes are identical on the interface, the limited solution can be written in terms of only low and high order algebraic fluxes in the interior of the element: 
\begin{align}
    \fnt{m}_i \fnt{u}_i^{n+1} = \fnt{m}_i \fnt{u}_i^{{\low},n+1} + \tau \sum\limits_{j\in\mathcal{I}(i)} l_{ij}\LRp{\fnt{F}_{ij}^{\low}-\fnt{F}_{ij}^{\high}}. \label{eq:newlimited}
\end{align}
The high-order method is still entropy stable, as shown in \cite{chen2017entropy,renac2019entropy}.

We note that this localization property is not specific to the Lax-Friedrichs flux, and that convex limiting reduces to a local procedure for any high and low order schemes which share the same interface flux and interface discretization matrices.

\rzero{
\section{2D viscous shockwave}
\label{app:expCNSconv}

We now examine the convergence of the limited solution in 2D for the compressible Navier-Stokes equations using the 1D viscous shockwave in Section \ref{sec:viscshockwave} extruded in y-direction. In other words, the initial condition can be written as 
\begin{align}
    \vec{u} \LRp{x,y,t} &= \begin{bmatrix}
        \rho\LRp{\xi} \\
        \rho\LRp{\xi}\LRp{u_\infty + u \LRp{\xi}}  \\
        0 \\
        \rho\LRp{\xi} \LRp{e \LRp{\xi} + \frac{1}{2} \LRp{u_\infty+u\LRp{\xi}}^2}
    \end{bmatrix},\label{eq:viscshockwave2d}
\end{align}
with variables defined as in \eqref{eq:viscshockwave}, \eqref{eq:viscshockwaveu}. We enforce the boundary condition similarly as in Section \ref{sec:viscshockwave} and set ${\rm CFL} = 0.75$. The parameters we used are $\gamma = 1.4, \mu = 0.01, u_\infty = 0.2, u_L = 1.0, m_0 = 1.0, M_0 = 3$. We present the $L^1$ error and the convergence rate of the elementwise limited solution with $\zeta = 0.1$ on the uniform quadrilateral and triangular meshes defined in Section \ref{sec:vortex}. Table \ref{tab:viscquad1} and \ref{tab:visctri1} show the elementwise limited solutions with $\zeta = 0.1$ have asymptotic convergence rates between $O(h^{N+1/2})$ and $O(h^{N+1})$ on both meshes, which is optimal for smooth solutions. 
}

\begin{table}[!htb]
\centering
\rzero{
\begin{tabular}{|c|c|c |c|c|c|c|c|c|} 
 \hline
 & \multicolumn{2}{|c|}{$N=1$} & \multicolumn{2}{|c|}{$N=2$} & \multicolumn{2}{|c|}{$N=3$} & \multicolumn{2}{|c|}{$N=4$}\\
 \hline
 K & $L^2$ error & Rate & $L^2$ error & Rate & $L^2$ error & Rate & $L^2$ error & Rate \\  
 \hline
 10  & $7.368\times 10^{-2}$ &      & $4.751\times 10^{-2}$ &      & $1.796\times 10^{-2}$ &      & $5.448\times 10^{-3}$ & \\ 
 20  & $3.204\times 10^{-2}$ & 1.20 & $1.007\times 10^{-2}$ & 2.24 & $2.168\times 10^{-3}$ & 3.05 & $1.203\times 10^{-3}$ & 2.18\\
 40  & $1.145\times 10^{-2}$ & 1.48 & $1.349\times 10^{-3}$ & 2.90 & $3.533\times 10^{-4}$ & 2.61 & $1.011\times 10^{-4}$ & 3.57\\
 80  & $2.921\times 10^{-3}$ & 1.97 & $1.976\times 10^{-4}$ & 2.77 & $3.882\times 10^{-5}$ & 3.19 & $3.231\times 10^{-6}$ & 4.97\\
 \hline
\end{tabular}
}
\caption{2D viscous shockwave, quadrilateral mesh - Elementwise (Zhang-Shu type) limiting with $\zeta = 0.1$}
\label{tab:viscquad1}
\end{table}

\begin{table}[!htb]
\centering
\rzero{
\begin{tabular}{|c|c|c |c|c|c|c|c|c|} 
 \hline
 & \multicolumn{2}{|c|}{$N=1$} & \multicolumn{2}{|c|}{$N=2$} & \multicolumn{2}{|c|}{$N=3$} & \multicolumn{2}{|c|}{$N=4$}\\
 \hline
 K & $L^2$ error & Rate & $L^2$ error & Rate & $L^2$ error & Rate & $L^2$ error & Rate \\  
 \hline
 10  & $6.372\times 10^{-2}$ &      & $3.053\times 10^{-2}$ &      & $1.834\times 10^{-2}$ &      & $4.957\times 10^{-3}$ & \\ 
 20  & $2.384\times 10^{-2}$ & 1.42 & $5.663\times 10^{-3}$ & 2.43 & $3.805\times 10^{-3}$ & 2.27 & $9.617\times 10^{-4}$ & 2.37\\
 40  & $7.324\times 10^{-3}$ & 1.70 & $1.308\times 10^{-3}$ & 2.11 & $4.876\times 10^{-4}$ & 2.96 & $7.997\times 10^{-5}$ & 3.59\\
 80  & $2.020\times 10^{-3}$ & 1.86 & $2.163\times 10^{-4}$ & 2.60 & $5.127\times 10^{-5}$ & 3.25 & $2.807\times 10^{-6}$ & 4.74\\
 \hline
\end{tabular}
}
\caption{2D viscous shockwave, simplicial mesh - Elementwise (Zhang-Shu type) limiting with $\zeta = 0.1$}
\label{tab:visctri1}
\end{table}

\newpage
\bibliographystyle{elsarticle-num}
\bibliography{reference.bib}

\end{document}